\pdfoutput=1
\documentclass[11pt,reqno]{amsart}

\numberwithin{equation}{section}
\usepackage[tt=false]{libertine}
\usepackage{mathtools}

\usepackage{amssymb,mathrsfs}
\usepackage[varbb]{newpxmath}
\usepackage[normalem]{ulem}
\let\savedbigtimes\bigtimes
\let\bigtimes\relax
\usepackage{mathabx} 
\let\bigtimes\savedbigtimes

\usepackage[margin=1in]{geometry}
\usepackage{graphicx}
\usepackage{enumerate}
\usepackage{bbm}
\usepackage{hyperref,color}
\usepackage[capitalize,nameinlink]{cleveref}
\usepackage[dvipsnames]{xcolor}
\hypersetup{
	colorlinks=true,
	pdfpagemode=UseNone,
    citecolor=OliveGreen,
    linkcolor=NavyBlue,
    urlcolor=black,
	pdfstartview=FitW
}
\usepackage{appendix}
\crefname{appsec}{Appendix}{Appendices}
\usepackage{tikz}
\usepackage{bm}
\usepackage{mathtools}
\usepackage{lmodern}

\usepackage[ruled]{algorithm2e}
\SetKwComment{Comment}{/* }{ */}

\newtheorem*{assumption*}{Assumption}
\newtheorem*{condition*}{Condition}

\newtheorem{theorem}{Theorem}[section]
\newtheorem{proposition}[theorem]{Proposition}
\newtheorem{lemma}[theorem]{Lemma}
\newtheorem{corollary}[theorem]{Corollary}

\theoremstyle{definition}
\newtheorem{definition}[theorem]{Definition}

\newtheorem{assumption}[theorem]{Assumption}
\newtheorem{remark}[theorem]{Remark}

\crefname{lemma}{Lemma}{Lemmas}
\crefname{theorem}{Theorem}{Theorems}
\crefname{definition}{Definition}{Definitions}
\crefname{fact}{Fact}{Facts}
\crefname{claim}{Claim}{Claims}
\crefname{proposition}{Proposition}{Propositions}

\newcommand{\dif}{\,\mathrm{d}}
\newcommand{\E}{\mathbb{E}}
\newcommand{\Var}{\mathrm{Var}}
\newcommand{\Cov}{\mathrm{Cov}}

\DeclareMathOperator*{\argmax}{arg\,max}

\newcommand{\norm}[1]{\left\lVert #1 \right\rVert}

\newcommand{\ceil}[1]{\left\lceil #1 \right\rceil}
\newcommand{\floor}[1]{\left\lfloor #1 \right\rfloor}
\newcommand{\ip}[2]{\left\langle #1 , #2 \right\rangle}

\newcommand{\trans}{\intercal}

\newcommand{\poly}{\mathrm{poly}}

\newcommand{\eps}{\varepsilon}
\renewcommand{\epsilon}{\varepsilon}

\newcommand{\N}{\mathbb{N}}

\newcommand{\R}{\mathbb{R}}

\newcommand{\beq}{\begin{equation}}
\newcommand{\eeq}{\end{equation}}

\usepackage{appendix}
\crefname{appsec}{Appendix}{Appendices}

\begin{document}

\title[On the Low-Temperature MCMC Threshold]{On the Low-Temperature MCMC threshold:\\the cases of sparse tensor PCA, sparse regression,\\and a geometric rule}

\author[Z. Chen, C. Sheehan, I. Zadik]{Zongchen Chen$^\dagger$, Conor Sheehan$^\star$, and Ilias Zadik$^\star$}
\thanks{
$^\dagger$School of Computer Science, Georgia Tech, Email: \texttt{chenzongchen@gatech.edu}.\\
$^\star$Department of Statistics and Data Science, Yale University,
Email: \texttt{\{conor.sheehan,ilias.zadik\}@yale.edu}}

\date{ 
\today}

\begin{abstract}
    Over the last years, there has been a significant amount of work studying the power of specific classes of computationally efficient estimators for multiple statistical parametric estimation tasks, including the estimator classes of low-degree polynomials, spectral methods, and others. Despite that, our understanding of the important class of MCMC methods remains quite poorly understood. For instance, for many simple statistical models of interest, the performance of even zero-temperature (randomized greedy-like) MCMC methods that simply maximize the posterior remains elusive.
    
    In this work, we provide a structural geometric condition under which the low-temperature Metropolis chain maximizes the posterior in polynomial-time with high probability. The result is generally applicable, and in this work, we use it to derive positive MCMC results for two classical sparse estimation tasks: the sparse tensor PCA model and sparse regression. Interestingly, in both cases, we also carefully apply the Overlap Gap Property framework for inference (Gamarnik, Zadik AoS '22) to prove that our results are tight: no low-temperature local MCMC method can achieve better performance. In particular, our work identifies the ``low-temperature (local) MCMC threshold'' for both sparse models, significantly improving upon a series of related prior works on the topic.

    Notably, in the sparse tensor PCA model our results indicate that low-temperature local MCMC methods greatly underperform compared to other studied time-efficient methods, such as the class of low-degree polynomials. Moreover, the ``amount'' by which they underperform appears connected to the size of the computational-statistical gap of the model. Specifically, an intriguing mathematical pattern is observed in sparse tensor PCA, which we call \emph{``the geometric rule''}: the information-theory threshold, the threshold for low-degree polynomials and the MCMC threshold form a geometric progression for all values of the sparsity level and tensor power. 
\end{abstract}

\maketitle

\newpage

\tableofcontents

\setcounter{page}{1}

\section{Introduction}

Over the recent years, it has been revealed that a plethora of statistical estimation models exhibit computational-statistical gaps, i.e., parameter regimes for which some estimator can achieve small error, but no such polynomial-time estimator is known to achieve similar error guarantees. It is an intriguing and rapidly growing research direction to understand whether these gaps are ``inherent'' or an inability of our current polynomial-time statistical estimation techniques. Given that the (simpler\footnote{Often the statistical models of interest include random noise assumptions, hence cannot be treated as ``classic'' computational questions about tasks with a worst-case input, such as e.g., the traveling salesman problem over a worst-case graph. It is this difference that makes statements about the classes such as $\mathcal{P}$ and $\mathcal{NP}$ not directly applicable to these statistical estimation settings.}) worst-case question of $\mathcal{P} \neq \mathcal{NP}$ remains open in computational complexity theory, unfortunately proving that all polynomial-time estimators fail to close any given computational-statistical gap appears well beyond our current mathematical abilities.

Despite that, over the recent years, researchers have constructed several tools to provide ``rigorous evidence'' for a computational-statistical gap of interest.  Some of the profound approaches include average-case reductions to ``hard regimes'' of other statistical problems (see e.g., \cite{berthet2013optimal, macomputa2015, brennan2020reducibility}), statistical physics inspired ``hardness'' heuristics \cite{gamarnik2022disordered}, and also proofs that natural \emph{restricted classes} of polynomial-time estimators fail to close these gaps. Some of the most popular studied such restricted classes are MCMC methods \cite{Jer92, chen2023almost}, low-degree polynomials \cite{schramm2022computational}, semidefinite programs \cite{barak2019nearly}, message passing procedures \cite{feng2022unifying}, and the class of statistical query methods \cite{feldman2017statistical}.

Among these restricted classes of estimators perhaps the most recent attention has been received by the class of estimators that take the form of a low-degree (e.g., $O(1)$ or $O(\log n)$ degree) polynomial function of the data (first used in the context of hypothesis testing in \cite{hopkins2018statistical} and for estimation settings in \cite{schramm2022computational}). A series of papers have built interesting tools to understand the exact threshold (in terms of some critical natural quantity of the model, often the sample size or signal to noise ratio) for a given statistical problem above which some low-degree polynomial works, and below which all low-degree polynomials fail (see e.g., \cite{kunisky2019notes} for a survey). On top of that, a conjecture has been put forward in the context of hypothesis testing \cite{hopkins2018statistical} that -- under mild noisy and symmetry assumptions \cite{holmgren2020counterexamples, zadik2022lattice} -- the threshold below which all $O(\log n)$-degree polynomials fail should in fact coincide with the threshold for which all polynomial-time methods fail as well. In other words, subject to this ``low-degree conjecture'', in many classical statistical models the performance of  $O(\log n)$-degree polynomials (which one can often directly analyze) allow us to characterize the performance of all polynomial-time estimators. Given the recent rapid growth of the low-degree literature for estimation problems following \cite{schramm2022computational} (e.g., \cite{montanari2022equivalence, mao2023detection, luo2023computational}), it is expected that a similar low-degree conjecture will be put forward in the near future in terms of statistical estimation problems as well.

Despite this recent popularity of analyzing the performance of low-degree polynomials for statistical estimation settings, it is natural to wonder what happens with other natural classes of polynomial-time estimators that are significantly different from low-degree polynomials (e.g., estimators that are not even well-approximated by them). One such \emph{highly canonical class} of statistical estimators (both in theory and applications of statistics) is the class of Markov Chain Monte Carlo (MCMC) methods that try to sample from the posterior (or slight variants of it), such as the Metropolis chain in discrete time and Langevin dynamics in continuous time. Researchers have studied this important class for a large family of models, but often arrive at significantly less satisfying theoretical results than the ones for the literature of low-degree methods. Among these MCMC results, most of them concern lower bounds, that is provable failures of classes of MCMC method to achieve a certain guarantee. Some influential such results include MCMC lower bounds for the planted clique problem \cite{Jer92, gamarnik2019landscape, chen2023almost}, sparse principal component analysis (PCA) \cite{gamarnik2021overlap, arous2023free} and tensor PCA \cite{benarousPCA}. Importantly, many of the recent such results leverage a suggested proof framework known as Overlap Gap Property for inference to obtain bottlenecks for the studied Markov chains, first introduced in \cite{10.1214/21-AOS2130}. On the other hand, there are significantly fewer positive MCMC results known for such estimation problems, including though interesting positive results about Bayesian variable selection \cite{jordanMCMC}, tensor PCA \cite{benarousPCA}, and recently also the planted clique model \cite{gheissari2023finding} and the stochastic block model \cite{liu2024locally}. Yet, even concerning these relatively few positive results, most of them are an outcome of a technical detailed analysis tailored to the specific model of interest. In particular, as opposed to lower bounds and the proof framework of \cite{10.1214/21-AOS2130}, there is a lack of a general ``easy-to-use'' strategy of how can one obtain a positive MCMC result for a new model of interest, besides directly trying to bound the mixing time (or analyze its locally stationary measures as in the recent work \cite{liu2024locally}). A significant motivation for this work is to answer the following question:

\begin{quote}
    \centering
    Can we provide a general user-friendly framework \\ to prove positive MCMC results for statistical estimation problems?
\end{quote}

Besides only obtaining positive MCMC results, it is of course  highly desirable to obtain the ``optimal'' MCMC results, in the sense of locating exactly what is the minimal sample size (or signal to noise ratio) so that an element of some natural class of MCMC methods succeeds. Unfortunately, it is significantly understudied how one could locate exactly the ``MCMC threshold'' for a generic class of estimation settings, in sharp contrast with the generally applicable tools designed for the class of low-degree polynomials \cite{schramm2022computational,xifan24}.  On top of that, the pursuit of the MCMC threshold in estimation settings is even more motivated because of recent surprising discoveries in the context of tensor PCA \cite{benarousPCA} and planted clique \cite{chen2023almost}. The reason is that in both of these models the cited works have proven that a severe underperformance of certain natural MCMC methods takes place compared to other polynomial-time methods (something also noticed in the statistical physics literature \cite{mannelli2019passed}). To make it more precise, let us focus on the tensor PCA case where for some integer $t \geq 2$ one observes the $t$-tensor $Y=\lambda v^{\otimes t}+W$ for some $v \in S^{d-1}$ chosen uniformly at random from the sphere and $W$ has i.i.d. $N(0,1)$ entries. The task is to recover $v$ by observing $Y$. While the minimum $\lambda=\lambda_d>0$ for which this is possible by some (time-unconstrained) method scales like $\lambda_{\mathrm{STATS}} \approx \sqrt{d}$, polynomial-time methods are known to begin to work only when $\lambda_{\mathrm{ALG}} \approx d^{t/4}$ and below $\lambda_{\mathrm{ALG}}$ all $O(\log n)$-degree polynomials are known to fail, suggesting that $\lambda_{\mathrm{ALG}}$ is the computational threshold for this setting. Yet, if one runs Langevin dynamics that tries to sample from the (slightly temperature misparametrized) posterior, \cite{benarousPCA} proves that the dynamics can recover $v$ in polynomial-time at the threshold $\lambda_{\mathrm{MCMC}} \approx d^{t/2-1/2}$. Notice that for $t \geq 3$, $d^{t/4}=o(d^{t/2-1/2})$. It is perhaps striking the underperformance by a polynomial-in-$d$ factor of the Langevin dynamics for this seemingly simple task. For such reasons, this phenomenon that tensor PCA undergoes is known in the community as a \emph{``local-to-computational gap''}. A partial motivation of this work is to understand how generic the existence of local-to-computational gaps is in statistical estimation settings. 

Lastly, from a mathematical standpoint, one can observe a somewhat strange but elegant formula for the MCMC threshold for tensor PCA, which suggests that the MCMC thresholds may be inherently connected to the other thresholds via some intriguing mathematical structure. \emph{For each $t \geq 2$,} observe the computational threshold is given simply by the geometric average of the statistical threshold and the MCMC threshold, i.e., for each $t \geq 2$.
\begin{align}\label{eq:geom}
    \sqrt{\lambda_{\mathrm{MCMC}} \times \lambda_{\mathrm{STATS}}}  \approx \lambda_{\mathrm{ALG}} \approx d^{t/4}.
\end{align}We call this identity as \emph{``the geometric rule''}. This seemingly coincidental formula suggests a perhaps overly optimistic question:
\begin{quote}\centering
    Could it be that this geometric rule is fundamental \\ (e.g., does it hold for other priors on $v$, or beyond tensor PCA)?
\end{quote}
Besides the mathematical elegance, a positive answer to such a question would mean that for a class of models, as long as there is a computational statistical trade-off, i.e., $\lambda_{\mathrm{ALG}}/ \lambda_{\mathrm{STATS}}=\omega(1)$, there is a local-to-computational gap as well, i.e., $\lambda_{\mathrm{MCMC}}/ \lambda_{\mathrm{ALG}}=\omega(1)$. In this work, we interestingly reveal that the geometric rule does indeed hold beyond just the tensor PCA model, and in particular for \emph{the sparse matrix and tensor PCA models} as well.

\subsection{Contributions}\label{sec:contrib}

In this work, we provide a framework for the systematic study of the low-temperature MCMC threshold for high-dimensional parametric estimation problems. We first explain exactly the family of estimation tasks of interest, as well as the family of MCMC methods we are focusing on. 

We study the \emph{discrete parameter estimation} set-up, where for some parameter $\theta^* \in \Theta \subseteq \mathbb{R}^d$ sampled from some prior $\mathcal{P}$, we observe a dataset $\mathcal{D} \sim \mathbb{P}_{\theta^*}$. We assume that $\Theta$ is a finite set of parameters and the prior is fully supported on $\Theta$. The goal of the statistician is to exactly recover the value of $\theta^*$ from observing $\mathcal{D}$, with high probability over both the randomness of the prior and the channel as $d$ grows to infinity. Here and everywhere in the paper, by saying that an event holds with high probability as $d$ grows to infinity, we mean that it's probability converges to one as $d$ grows to infinity. 

It is folklore that the optimal estimator to exactly recover $\theta^*$ with the maximal possible probability of success (among all time-unconstrained recovery procedures) is the Maximum A Posteriori (MAP) estimator which maximizes among all $\theta \in \Theta$ the posterior $\mathbb{P}(\theta|\mathcal{D})$. Of course the issue leading to computational-statistical gaps is that often this maximization problem defining MAP can be intractable to compute directly, e.g., because $\Theta$ is exponentially large. A natural strategy in both theory and practice to perform this maximization in a time-efficient manner is to run a ``local'' \footnote{``Locality'' refers to the property that each state is connected in the underlying neighborhood graph of the chain with only a few neighbors. This allows for the step transitions of the chain to be computationally efficient even if $\Theta$ is large.} reversible Markov chain on $\Theta$ which for some sufficiently large $\beta>0$ (or ``low'' enough temperature) it has stationary measure $\pi_{\beta}$ \footnote{In practice, one may consider a simulated annealing schedule in terms of the inverse temperature $\beta$. Yet, in this work, as common in theoretical MCMC works, we simplify this setting and assume one fixed $\beta$ throughout the iterations of the Markov chain.} given by
\begin{align*}
    \pi_{\beta}(\theta) \propto \exp\left(\beta \log\mathbb{P}(\mathcal{D}|\theta)\right), \theta \in \Theta.
\end{align*}For convenience in what follows, we refer to such a family of local Markov chains as ``low-temperature MCMC methods''. We are interested in under which assumptions low-temperature MCMC methods can recover $\theta^*$ with high probability.

\subsubsection{A general positive result}As we mentioned in the Introduction, the positive results for such low-temperature MCMC methods are few, often quite technical, and usually an outcome of an analysis crucially tailored to each case of interest. 

Our first contribution is a \emph{general positive ``black-box'' result} (\cref{thm:canonical}) which offers a structural geometric condition under which some low-temperature MCMC method succeeds in polynomial-time to optimize the posterior, and in particular recover $\theta^*$ with high probability. Our result in fact applies for any optimization problem over a discrete feasible set and we expect to be of independent interest to the discrete optimization community as well. 

The sufficient condition simply requests the existence of a ``neighborhood'' graph defined on the parameter space such that from any parameter $\theta$ there exists an ascending path, in terms of the posterior value (or the arbitrary objective of interest), towards the optimizer that has polynomial length (see Figure \ref{fig:spanning_tree} below). Note that the condition is simple but non-trivial exactly because the whole parameter space could be exponentially large.

As long as solely this condition is satisfied, we prove that simply running the (low-temperature) Metropolis chain on this neighborhood graph for polynomial time is guaranteed to find the optimizer of the posterior $\theta^*$ with high probability. As we mentioned above, to the best of our knowledge, no similar result is known at that level of generality even in the discrete/combinatorial optimization community. The proof is simple and is based on a general canonical paths construction \cite{jerrum2003counting} and could be potentially of independent interest across communities.

To show the generality of \cref{thm:canonical} we derive two positive MCMC results for two of the most classical, but significantly different, \emph{sparse estimation problems} with well-studied computational-statistical gaps. First, we focus on the \emph{sparse tensor PCA model} (also called, constant high order clustering) \cite{luo2022tensor, choo2021complexity} which is a special case of a Gaussian additive model, i.e., one observes a sparse signal drawn from some prior where Gaussian noise is added to it. Second, we also apply it to \emph{Gaussian sparse regression} \cite{10.1214/21-AOS2130} which is a special case of a Linear Model, i.e., one observes i.i.d. samples of a (random) linear transformation of a sparse signal where Gaussian noise is added to each sample. Interestingly, using the Overlap Gap Property framework from \cite{10.1214/21-AOS2130} we can also show that both the obtained low-temperature MCMC results are tight, in the sense that a large class of low-temperature MCMC results fail to work in polynomial-time, whenever the conditions of our positive result do not hold. In particular, these results allow us to characterize the exact power of a large class of low-temperature MCMC methods for both of these settings. We finally highlight that both applications of the Overlap Gap Property framework are rather non-trivial and require highly technical applications of the second moment method to prove existence of ``bad" bottlenecks in the landscape of the models.

\subsubsection{Sparse Tensor PCA}We start by describing the contribution for the case of sparse tensor PCA (also called constant high order clustering setting). The model is that for some $p \rightarrow +\infty$ and $\omega(1)=k=k_p=o(p)$, we choose a $k$-sparse $v \in \{0,1\}^p$ uniformly at random, and assume the statistician observes for some $t \geq 2$ the $t$-tensor \begin{align}\label{eq:stpca} Y=\lambda v^{\otimes t}/k^{t/2}+W\end{align} where $W$ has i.i.d. $N(0,1)$ entries. The goal is to recover $v$ from $Y$ with high probability as $k \rightarrow +\infty$.

The sparse tensor PCA model we focus on is a canonical Gaussian principal component analysis setting with a large literature in theoretical statistics (see e.g., \cite{10.3150/12-BEJ470,butucea2015sharp,cai2017computational,xia2019sup,niles2020all,luo2022tensor} and references therein). Moreover, the setting is greatly motivated by a wide range of applications in many domains, such as genetics, social sciences, and engineering (see e.g., \cite{henriques2018triclustering} for a survey). More specific interesting examples include topics such as microbiome studies \cite{faust2012microbial}
and multi-tissue multi-individual
gene expression data \cite{wang2019three}.

Now, from the theoretical standpoint it is known \cite{luo2022tensor, niles2020all} that the optimal estimator (MAP) can recover $v$ if and only if $\lambda \geq \lambda_{\mathrm{STATS}}(t,k)$ for some\begin{align*}
    \lambda_{\mathrm{STATS}}(t,k)=\tilde{\Theta}\left(\sqrt{k}\right)
\end{align*}where the $\tilde{\Theta}$ notation here and throughout the paper hides $\poly \log p$-dependent terms. 

Moreover, based on average-case reductions to hypergraphs variants of the planted clique problem \cite{brennan2020reducibility, luo2022tensor}, polynomial-time estimators are expected to work if and only if $\lambda \geq \lambda_{\mathrm{ALG}}(t,k)$ for some\begin{align*}
    \lambda_{\mathrm{ALG}}(t,k)=\tilde{\Theta} \left(\min\{k^{t/2},p^{(t-1)/2}/k^{t/2-1}\}\right).
\end{align*}Note that for all $k=o(p), t \geq 2 $ it holds $\lambda_{\mathrm{ALG}}(t,k)=\tilde{\omega}(\lambda_{\mathrm{STATS}}(t,k))$ and therefore the model is conjectured to exhibit a computational-statistical trade-off.

\begin{remark}  

It should be noted that the mentioned algorithmic lower bound when $\lambda=o(\lambda_{\mathrm{ALG}}(t,k))$ from \cite{brennan2020reducibility,luo2022tensor} 
 applies in principle for a slight variant of the model of sparse tensor PCA compared to one we introduced above, which the authors of \cite{luo2022tensor} denote by $\mathrm{CHC}_R$. In $\mathrm{CHC}_R$ one observes $Y=\lambda v_1 \otimes v_2 \otimes \ldots \otimes v_t+W$ for $v_i, i=1,\ldots,t$ which are \emph{independent} $k$-sparse binary vectors chosen uniformly at random. Note that in our case we focus instead on the case $v_1=\ldots=v_t=v$ where $v$ is a $k$-sparse binary vectors chosen uniformly at random. Often in the literature, the model we study is called the symmetric version of a tensor PCA model, while the one studied in \cite{luo2022tensor} is called the asymmetric version (see e.g., the discussion in \cite{dudeja2021statistical}).

Yet, a quick check in the proof of their lower bound \cite[Theorem 16]{luo2022tensor} shows that in fact the authors first prove the conditional failure of polynomial-time algorithms when $\lambda=o(\lambda_{\mathrm{ALG}}(t,k))$ for our case of interest, that is when $v_1=\ldots=v_t$, and then  
add an extra step to also prove it for the case that $v_i,i=1,\ldots,t$ are independent. In particular, their result applies almost identically to the case of symmetric sparse tensor PCA we focus on this work. For reasons of completeness, in \cref{sec:reduction}, we describe exactly the slightly modified argument from \cite{luo2022tensor} and how it applies in our symmetric setting. We finally also note that the reduction argument for sparse tensor PCA in \cite{luo2022tensor} is based on some instrumental techniques introduced in \cite{brennan2020reducibility} for the hardness of tensor PCA in the special case of $v\in\{-1,1\}^n$ (that is the case $v $ is Rademacher-valued and not sparse).
\end{remark}

A careful application of our ``black box'' positive MCMC result interestingly implies that for all $t \geq 2$ and various values of the sparsity level $\omega(1)=k=o(p),$ if $\beta>0$ is large enough, the Metropolis chain on the Johnson graph with stationary measure $\pi_{\beta}$ hits the planted vector $v$ in polynomial time as long as  $\lambda \geq \lambda_{\mathrm{MCMC}}(t,k)$ for a significantly different threshold compared to $\lambda_{\mathrm{ALG}}(t,k)$ given by  
\begin{align*}
   \lambda_{\mathrm{MCMC}}(t,k)=\tilde{\Theta}\left(\min \left\{ k^{t-\frac{1}{2}},\, \frac{p^{t-1}}{k^{t-\frac{3}{2}}} \right\}\right). 
\end{align*}See \cref{thm:alg-sparse-dense} for the exact statement and specifically the Johnson adjacency graph the Metropolis chain is naturally defined on. Note also that for all $k=o(p), t \geq 2 $ it holds $\lambda_{\mathrm{MCMC}}(t,k)=\tilde{\omega}(\lambda_{\mathrm{ALG}}(t,k))$.

 Satisfyingly, we also prove that when $\lambda=\tilde{o}(\lambda_{\mathrm{MCMC}}(t,k))$ a negative result for all low-temperature reversible local MCMC methods (including the Metropolis chain) holds. Specifically, we prove in \cref{thm:lb} that for large enough $\beta>0$ all local reversible Markov chains with stationary measure $\pi_{\beta}$ require super-polynomial time to hit $v$. Notably, our result significantly improves upon the prior work \cite{arous2023free} on sparse PCA where a similar negative result for local Markov chains was proven but only when $t=2$ and importantly under the more restrictive condition $\lambda=o(\lambda_{\mathrm{ALG}}(2,k))$. The proof of \cref{thm:lb} is the most technical contribution of the present work. We prove -- via a delicate conditional second moment method -- a bottleneck for the MCMC dynamics based on the Overlap Gap Property framework from \cite{10.1214/21-AOS2130}. The proof carefully leverages the ``flatness'' idea from random graph theory \cite{balister2019dense, gamarnik2019landscape}, but adjusted to high dimensional Gaussian tensors, and could be of independent interest. Moreover, the fact that the Overlap Gap Property from \cite{10.1214/21-AOS2130} appears in the landscape of the problem well into it's ``easy" phase where some polynomial-time method works (specifically when $\tilde{\omega}(\lambda_{\mathrm{ALG}}(t,k))=\lambda=\tilde{o}(\lambda_{\mathrm{MCMC}}(t,k))$) is a surprising fact, first conjectured in \cite{gamarnik2019landscape} for the planted clique model that remains unproved. This is a result which in fact aligns with similar results in the general realm of average-case complexity (see e.g., \cite{li2024easyoptimizationproblemsoverlapgap} for a similar recent result for average-case optimization) and naturally leads to multiple future questions about the relation between ``geometric" complexity and algorithmic hardness. 

The above results combined, conclude that $\lambda_{\mathrm{MCMC}}(t,k)$ is the low-temperature MCMC threshold for this model. The specific formula of the threshold may appear puzzling to the reader, as in most regimes it is significantly bigger than $\lambda_{\mathrm{ALG}}(t,k)$. This is reminiscent of the literature of tensor PCA, and an \emph{arguably striking conclusion} of our work is that \emph{the geometric rule} we observed for tensor PCA still applies also for sparse tensor PCA. Specifically, for sparse tensor PCA, a quick verification gives that it holds for all $t \geq 2$ and $k=o(p)$ that our \cref{thm:lb} implies that

\begin{align}\label{eq:geom_2}
    \sqrt{\lambda_{\mathrm{MCMC}}(t,k) \times \lambda_{\mathrm{STATS}}(t,k)} = \tilde{\Theta}(\lambda_{\mathrm{ALG}}(t,k)).
\end{align}We remark that this is a significant departure compared to the established tensor PCA result \cite{benarousPCA} as described in \eqref{eq:geom}, as it holds firstly under the sparsity assumption, and also it importantly applies even to the matrix case $t=2$ where tensor PCA does not exhibit either a computational-statistical gap or local-to-computational gap. Albeit interesting in its own right, the fact that the result applies when $t=2$ allows us also to see the geometric rule in simulations for relatively small values of $p,k$, which we find particularly satisfying (see Section \ref{sec:sim}). We suspect that a generalization of the geometric rule to a large family of Gaussian additive models is possible. The exact reason such an elegant formula appears in that generality remains an intriguing question for future work.

 \subsubsection{Sparse Regression}

 In this setting, we study the sparse linear model where for some $p \rightarrow +\infty$ and $\omega(1)=k=k_n=o(p)$, we choose a $k$-sparse $v \in \{0,1\}^p$ uniformly at random, and then observe $n$ i.i.d. samples $(x_i,y_i)$ of the form $y_i=\ip{x_i}{v}+w_i$ for independent $x_i \sim N(0,I_p)$ and $w_i \sim N(0,\sigma^2).$ 
 
The studied model is a standard Gaussian version of sparse regression with a very large literature in theoretical statistics, compressed sensing, and information theory (see \cite{wainwright_lasso,  Hassibi98,BarronMonograph,Casto11,JASAgenomics, reeves2021all, 10.1214/21-AOS2130} and references therein). Of course also the setting is greatly motivated by the study of sparse linear models in a plethora of applied contexts such as radiology  and biomedical imaging (see e.g., \cite{DonohoMRI} and references therein) and genomics \cite{JASAgenomics, BickelGenome}. 

Now for our setting it is well-known (e.g., an easy adaptation of the vanilla union bound argument in \cite[Theorem 4]{reeves2020all}) that as long as $\sigma^2=o(k),$ MAP works if $n \geq n_{\mathrm{STATS}}$ for 
 \begin{align*}
     n_{\mathrm{STATS}}=\Theta(k\log (p/k)/\log (k/\sigma^2)).
 \end{align*} Yet, based on the performance of standard algorithms such as LASSO \cite{wainwright_lasso, 10.1214/21-AOS2130} and the low-degree framework \cite{bandeira2022franz}, polynomial-time algorithms are expected to succeed if and only if $n \geq n_{\mathrm{ALG}}$ for \begin{align*}
     n_{\mathrm{ALG}}=\Theta(k\log (p/k)).
 \end{align*} 

By applying our black-box result we can prove that in the case of sparse regression low-temperature MCMC methods are in fact achieving the algorithmic threshold, in sharp contrast with the case of sparse tensor PCA we discussed above. Specifically in \cref{thm:SRpos} we prove as long as $n \geq n_{\mathrm{MCMC}}$ for some 
\begin{align*}
    n_{\mathrm{MCMC}}=\Theta(n_{\mathrm{ALG}})
\end{align*}then for any large enough $\beta>0$ the Metropolis chain on the Johnson graph with stationary measure $\pi_{\beta}$ hits $v$ in polynomial-time, extending a result from \cite{10.1214/21-AOS2130} which only holds for greedy dynamics and significantly lower noise levels than our result.

Importantly, low-temperature MCMC methods that run for polynomially many iterations \emph{cannot} be in principle approximated by low-degree polynomials. For this reason, even in light of the low-degree lower bound of \cite{bandeira2022franz}, we focus also on whether local Markov chains can succeed when $n=o(n_{\mathrm{ALG}})$. In \cref{thm:negresult} we show that as long as $k=o(p^{1/3})$ whenever $n=o(n_{\mathrm{ALG}})$ for large enough $\beta>0$ a large family of low-temperature local reversible MCMC methods on the Johnson graph with stationary measure $\pi_{\beta}$ (including the Metropolis chain) take super-polynomial time to hit $v$, bringing new evidence for the polynomial-time optimality of low-degree polynomials even compared to seemingly very different classes of estimators. Similar to the case of sparse tensor PCA, this result follows by leveraging the Overlap Gap Property framework from \cite{10.1214/21-AOS2130}, alongside a lengthy and delicate modification of a second moment method inspired by a similar one in \cite{10.1214/21-AOS2130}. Our result yet crucially generalizes the calculation and a similar lower bound proven by \cite{10.1214/21-AOS2130} which was leveraging the more restrictive assumption $k=p^{o(1)}$.

In conclusion, our work identifies for sparse regression the threshold for the success of a large class of low-temperature MCMC methods and proves that it coincides with the algorithmic threshold, i.e., it holds up to constants \begin{align*}n_{\mathrm{MCMC}}=\Theta(n_{\mathrm{ALG}})=\Theta(k\log (p/k)).\end{align*}In particular, as the model exhibits a computational-statistical gap, i.e., $n_{\mathrm{STATS}}=o(n_{\mathrm{ALG}}),$ the geometric rule for sparse regression does not hold.

\subsection{Asymptotic Notation}

 We will use the standard asymptotic notations $O, o, \Omega, \omega, \Theta$. On top of that, for any $(a_p)_{p \in \mathbb{N}}$, $(b_p)_{p \in \mathbb{N}}$ sequences of real numbers, we write $a_p=\tilde{o}(b_p)$ when $a_p = o(b_p (\log p)^x)$ for some $x \in \mathbb{Z}$, $a_p=\tilde{\omega}(b_p)$ if $b_p=\tilde{o}(a_p)$ and $b_p=\tilde{\Theta}(a_p)$  if $\Omega(b_p (\log p)^y)=a_p = O(b_p (\log p)^x)$ for some $y,x \in \mathbb{Z}$.
\section{Main Positive result}
We now start formally presenting our results. We start with presenting our main results with our general positive result for local Markov chains on optimizing certain Hamiltonians of interest. Specifically, our positive results apply for the so-called Metropolis chain (also called Metropolis--Hastings process) \cite{Metro53,Hastings70,Tierney94,Jer92}.

We present the result in a general optimization context of a function over a finite state space. We later apply it in two statistical estimation settings as mentioned above in \cref{sec:contrib}.

Suppose $\mathcal{X}$ is a finite state space. Let $H: \mathcal{X} \rightarrow \mathbb{R}$ be a Hamiltonian, and suppose that $v^* =\argmax_{x \in \mathcal{X}} H(x)$ is the unique global maximum of the Hamiltonian $H$. 
Furthermore, we denote the range of $H$ as
\begin{align*}
    \mathcal{R}_H := \max_{x \in \mathcal{X}} H(x) - \min_{x \in \mathcal{X}} H(x) = H(v^*) - \min_{x \in \mathcal{X}} H(x).
\end{align*}
For $\beta > 0$, our goal is to sample from 
\[\pi_\beta (x) \propto e^{\beta H(x)}, \quad \forall x \in \mathcal{X}.\]
\begin{remark} (The corresponding Bayesian estimation quantities)
    For correspondence to \cref{sec:contrib}, $\mathcal{X}$ would be the parameter space $\Theta$, $H$ would correspond to the logarithm of the posterior (assuming it is non-zero on the whole parameter space) and $v^*$ would be the MAP estimator.
\end{remark}

We focus on running a ``local'' Markov chain, called the Metropolis chain \cite{Metro53,Hastings70,Tierney94,Jer92}, on a (simple) graph $G$ with vertex set $\mathcal{X}$. 
For each $x \in \mathcal{X}$ denote the neighborhood of $x$ by $\Gamma_G(x) = \{y \in \mathcal{X}: xy \in E\}$.
Let $\Delta$ denote the maximum degree of $G$, i.e., $\Delta = \max_{x \in \mathcal{X}} |\Gamma_G(x)|$.
The Metropolis chain associated with graph $G$ is a random walk on $G$ with transition probabilities specified by the Hamiltonian; more precisely it updates as follows in each step. 

\medskip
\emph{One step of Metropolis chain}
\begin{enumerate}
\item With probability $1 - |\Gamma_G(X_{t-1})|/\Delta$: let $X_t \gets X_{t-1}$;
\item With remaining probability $|\Gamma_G(X_{t-1})|/\Delta$: Pick $x \in \Gamma_G(X_{t-1})$ uniformly at random,
and update 
\begin{align*}
X_t \gets 
\begin{cases}
x, & \text{with probability}  \min\{1,\pi(x)/\pi(X_{t-1})\}; \\
X_{t-1}, & \text{with remaining probability.}
\end{cases}
\end{align*}
\end{enumerate}

It is standard that under mild conditions (namely $G$ is finite and connected), the Metropolis chain is ergodic and converges to the target distribution $\pi_\beta$ under any initialization, see \cite{LP-book17}. In particular, if $G$ is $\Delta$-regular then the transition of the chain can be simplified as step~1 can be removed. We note that the uniform random proposal of $x$ in the neighborhood of $X_{t-1}$ is common in the literature (e.g. in \cite{jordanMCMC}), and contrasts with an ``informed" proposal such as that in \cite{zhouDimensionFree}.

We operate under the following assumption, which should be interpreted as that from every state there is an ``increasing'' path of length at most $D$ that leads to the maximizer $v^*$ of $H$.

\begin{assumption}
\label{assum:paths} 
There exists a spanning tree $T \subseteq G$ rooted at $v^*$ and constants $D,\delta > 0$ such that: 
\begin{itemize}
    \item The diameter of $T$ is at most $D$;
    \item For all $x \in \mathcal{X} \setminus \{v^*\}$, 
    \begin{align}\label{eq:inc}
        H(\mathrm{parent}(x))-H(x) \ge \delta
    \end{align}
    where $\mathrm{parent}(x)$ denotes the parent of $x$ in $T$.
\end{itemize}
\end{assumption}

\cref{assum:paths} is visualized in \cref{fig:spanning_tree}. Our first result is to show that solely under \cref{assum:paths}, we can directly upper bound the time the Metropolis chain takes to hit $v^*$ for sufficiently large values of $\beta>0$. 
Intuitively, the Metropolis chain can find its way towards the root $v^*$ of the spanning tree $T$.
\begin{figure}
    \centering
    \includegraphics[width=1\linewidth]{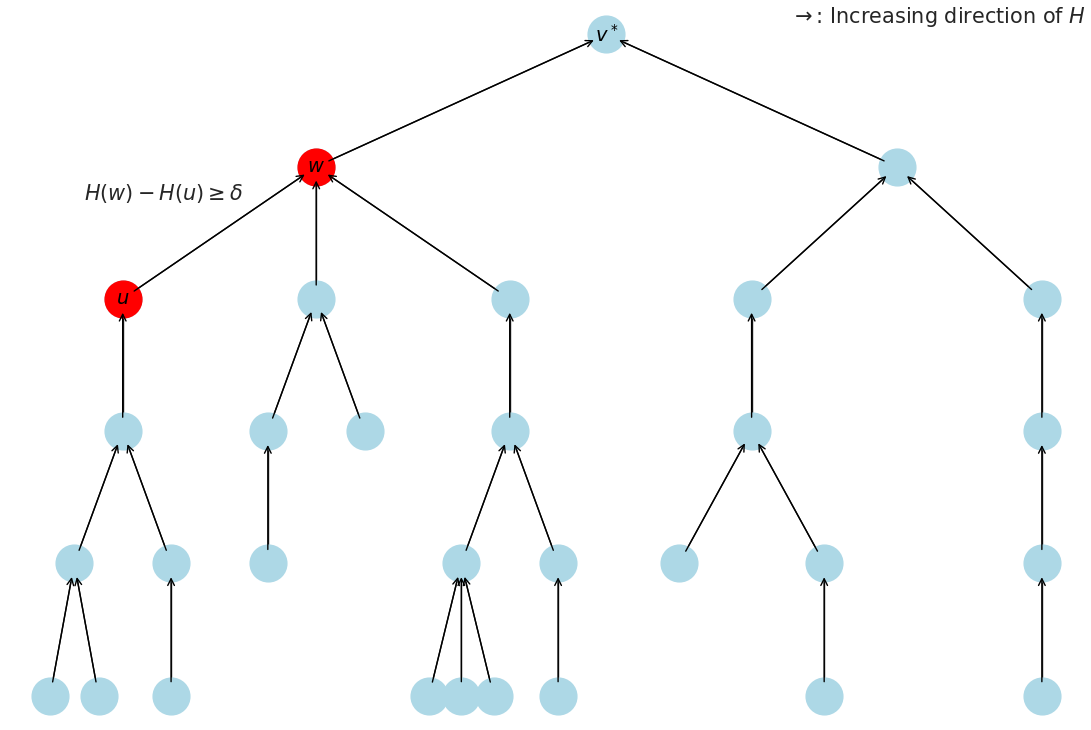}
    \caption{$H$ increases by at least $\delta$ at every step towards $v^*$, and the diameter of the tree is at most $D$.}
    \label{fig:spanning_tree}
\end{figure}
 
\begin{theorem}[General positive result]\label{thm:canonical}
Consider the Metropolis chain on $G$ with maximum degree 
$\Delta = \max_{x \in \mathcal{X}} |\Gamma_G(x)|$.
Suppose that \cref{assum:paths} holds with constants $D,\delta > 0$. 
Then, for any $\beta \ge \frac{1}{\delta}(\log \Delta + 2)$ we have:
\begin{enumerate}[(1)]
    \item $\pi_\beta(v^*) \ge \frac{3}{4}$;
    \item For any $T=\omega(D\Delta \left(\log |\mathcal{X}| + \beta \mathcal{R}_H)\right)$ it holds that for large enough $n$ the Metropolis chain after $T$ iterations is at state $v^*$ with probability at least $2/3$;

\end{enumerate}

\begin{enumerate}[(1)]
    \item[(3)]
    For any $T=\omega\left(\frac{\Delta \mathcal{R}_H}{\delta}\right)$ and any $\beta \ge \frac{2}{\delta}\left(\log |\mathcal{X}| + \log\left(\frac{40\mathcal{R}_H}{\delta}\right)\right)$, it holds that for large enough $n$ the Metropolis chain after $T$ iterations is at state $v^*$ with probability at least $2/3$. Moreover, for the Randomized Greedy algorithm (i.e., $\beta = \infty$), once the chain hits $v^*$, it remains at $v^*$ for all future iterations after $T$ with probability 1.
\end{enumerate}
\end{theorem}
The proof of this theorem is given in \cref{sec:GenPosResult} and relies on a canonical paths argument. 

\begin{remark}(``One good path suffices")
    We highlight that an interesting aspect of \cref{thm:canonical} is that under \cref{assum:paths} we only need the existence of one ``short'' ascending path towards the root for the Metropolis chain to hit the root fast. 
    Note that this is analogous to optimizing convex functions using gradient descent, where the gradient flows can be viewed as (continuous) descending paths from any point in the space to the minimizer. Our result shows that, in discrete settings, the mere existence of these monotonic paths suffice to find the optimizer in polynomial time.
\end{remark}

\begin{remark}(``How high is the lower bound on $\beta$?") While our positive result holds only for ``large enough'' $\beta,$ we highlight that our lower bound on the applicable values of $\beta$ so that \cref{thm:canonical} applies is not that restrictive. For instance, as we are about to see in two applications (in sparse PCA and sparse regression), the bound is satisfied in both examples for any $\beta$ bigger than the Bayesian optimal posterior temperature of the statistical task, and it is satisfied for certain smaller values of $\beta$ than the posterior temperature as well. 
\end{remark}
\section{Application: Sparse Gaussian additive model}

We now exploit the generality of \cref{thm:canonical}, and apply to identify the threshold for the low-temperature Metropolis chain for two important statistical estimation models: the sparse Gaussian additive model, which is a generalization of sparse tensor PCA, and sparse linear regression. We start with sparse Gaussian additive model in this section.

\subsection{The model}
Here we set out the assumptions of the sparse Gaussian Additive Model.
For integers $0 \le k \le p$, let 
\begin{align}\label{eq:sparse_set_def}
    \mathcal{X}_{p,k} = \left\{ x \in \{ 0,1 \}^p: \norm{x}_0 = k \right\}
\end{align}
be the set of $p$-dimensional binary vectors with exactly $k$ entries being $1$.
It is helpful to identify each vector $x \in \mathcal{X}_{p,k}$ with a $k$-subset of $[p]$ corresponding to nonzero entries of $x$; namely, each element in $\mathcal{X}_{p,k}$ is an indicator vector of a $k$-subset of $[p]$.  
The structure of $\mathcal{X}_{p,k}$ naturally defines a graph $G_{p,k}$ on it (known as Johnson graph) where two vertices $x,y \in \mathcal{X}_{p,k}$ are adjacent iff $x,y$ differ at exactly two entries (i.e., $|x\cap y| = k-1$). In what follows we discuss the performance of MCMC methods with neighborhood graph $G_{p,k}$, often referring to them as ``local'' MCMC methods.

Let $F: \mathbb{S}^{p-1} \to \mathbb{R}^d$ be a mapping such that for some ($p$-independent) smooth $f: [-1,1] \rightarrow \R$ it holds 
for all $a,b \in \mathbb{S}^{p-1}$ that
\[
\langle F(a), F(b) \rangle =f(\langle a,b \rangle).
\]
Assume $v^*$ is sampled from the uniform distribution over $\mathcal{X}_{p,k}$, and let 
\[
Y=\lambda F\left( \frac{v^*}{\sqrt{k}} \right) + W
\]
where $W$ is iid standard Gaussian noise and $\lambda=\lambda_{p,k} > 0$ is the signal-to-noise ratio (SNR). The goal is then to recover $v^*$ from the noisy observation $Y$.

\begin{remark}
    We remark for the reader's convenience that in \cref{sec:contrib} we specialized for simplicity the results of this section to the case of sparse tensor PCA \cite{luo2022tensor} where for some integer $t \geq 2$, $F(v)=v^{\otimes t}$ and in particular $f(x)=x^t$. Yet, our positive result applies in the more general case of the sparse Gaussian additive model described above. 
\end{remark}
Notice that the posterior distribution of $v^*$ given $Y$ is given by 
\begin{align}\label{eq:PCApost}
\pi_{\mathrm{post}}(x) 
\propto \exp(\lambda H(x))
, \quad \forall x \in \mathcal{X}_{p,k}
\end{align} where
\begin{align}\label{eq:H-GAM}
H(x) = \ip{F\left( \frac{x}{\sqrt{k}} \right)}{Y}
= \lambda f\left( \frac{\ip{x}{v^*}}{k} \right) + \ip{F\left( \frac{x}{\sqrt{k}} \right)}{W}.
\end{align}

\subsection{The positive result}
As we discussed the optimal estimator is to maximize the posterior. Hence, we study here the performance of the Metropolis chain on $G_{p,k}$ to sample from the low-temperature variant of the posterior; that is, from 
\[
\pi(x) 
\propto \exp(\beta H(x))
, \quad \forall x \in \mathcal{X}_{p,k}
\]for sufficiently large values of $\beta=\beta_{p,k}>0$.

We also require some smoothness assumptions on $f$.
\begin{assumption}[Smoothness of $f$]
\label{ass:f}
Assume $f \in C^2[0,1]$ is monotone increasing and convex,  
and there exist absolute constants $\eps \in (0,1)$ and $t \ge 1$ such that
\begin{enumerate}[(a)]
    \item $f(0) = 0$ and $f(1) = O(1)$; \label{ass:f-f'(0)-f'(1)-f(1)}
    \item $f'(x) = \Omega\left( x^{t-1} \right)$ for all $x \in [0,\eps]$, and $f'(1) = O(1)$; \label{ass:f-f'(0)-f'(1)}
    \item $f''(x) = O(1)$ for all $x \in [1-\eps,1]$. \label{ass:f-f''(1)} 
\end{enumerate}
\end{assumption}For example, it is immediate to check that the monomials $f(x)=x^m, m \ge 1$ directly satisfy the \cref{ass:f} for $t=m$ and any $\epsilon \in (0,1)$.

We now state our main positive result for the sparse Gaussian additive model.
Given $p,k$, let $\ell_{p,k}$ denote the typical overlap between $v^*$ and a random element from $\mathcal{X}_{p,k}$, defined formally as
\begin{align*}
    \ell_{p,k} =
    \begin{cases}
        1, & \text{if $k = O(\sqrt{p \log p})$}; \\
        \frac{k^2}{p}, & \text{if $k = \omega(\sqrt{p \log p})$}.
    \end{cases}
\end{align*}

\begin{theorem}[Sparse Gaussian Additive Model: Positive Result]
\label{thm:alg-sparse-dense}
    Consider the sparse Gaussian additive model where $f$ satisfies \cref{ass:f} with constants $\eps \in (0,1)$ and $t \ge 1$. Suppose the signal-to-noise ratio satisfies
    \begin{align*}
        \lambda 
        = \Omega\left( \frac{k^{t-\frac{1}{2}}}{\ell_{p,k}^{t-1}} \sqrt{\log p} \right)
        =
        \begin{cases}
            \Omega\left( k^{t-\frac{1}{2}} \sqrt{\log p} \right), & \text{if $k = O(\sqrt{p \log p})$}; \\
            \Omega\left( \frac{p^{t-1} }{k^{t-\frac{3}{2}}} \sqrt{\log p} \right), & \text{if $k = \omega(\sqrt{p \log p})$}.
        \end{cases}
    \end{align*}
    Then the following statements hold:
    \begin{enumerate}
        \item (Moderate $\beta$): For any $\beta$ satisfying
            \begin{align*}
                \beta \lambda = \Omega\left( \frac{k^t}{\ell_{p,k}^{t-1}} \log p \right),
            \end{align*}
            and any $T=\omega(\beta\lambda k^2p)$, the Metropolis chain with random initialization after $T$ iterations  is at $v^*$ with probability at least $2/3-o(1)$.
        \item (Large $\beta$): For any $\beta$ satisfying
            \begin{align*}
                \beta \lambda = \Omega\left( \frac{k^{t+1}}{\ell_{p,k}^{t-1}} \log\left( \frac{p}{k} \right) \right),
            \end{align*}
            and any $T=\omega\left( \frac{k^{t+1}}{\ell_{p,k}^{t-1}} p \right)$, the Metropolis chain with random initialization after $T$ iterations is at $v^*$ with probability at least $2/3-o(1)$. Moreover, for the Randomized Greedy algorithm (i.e., $\beta = \infty$), once the chain hits $v^*$, with high probability it remains at $v^*$ for all future iterations.
    \end{enumerate}
\end{theorem}

The proof of \cref{thm:alg-sparse-dense} is given in \cref{sec:PosGAMproof}.
\begin{remark}[Hitting time dependence on $\beta$]
    We note that for the case (1), that is for moderately large values of $\beta,$ our hitting time guarantee of $v^*$ increases with $\beta$, but in case (2), where $\beta$ exceeds some larger value than case (1), the hitting time ceases to depend on $\beta$ and it reduces to $O\left( \frac{k^{t+1}}{\ell_{p,k}^{t-1}} p \right)$ for all $\beta$ in case (2).
\end{remark}

\begin{remark}[Positive result for the posterior temperature]
    We note that the posterior temperature, $\beta_{\mathrm{post}}=\lambda$, by assumption of \cref{thm:alg-sparse-dense} satisfies
    \[\beta_{\mathrm{post}}\lambda=\lambda^2=\Omega\left(\frac{k^{2t-1}}{\ell_{p,k}^{2t-2}}\log p\right).\]
    An easy calculation shows that since $k=o(p)$ and $t\geq 2$, this satisfies the moderate $\beta$ threshold for all $k$. In particular, for this regime of $\lambda,$ the Metropolis chain on the Johnson graph sampling from $\pi_{\mathrm{post}}$ hits $v^*$ in polynomial-time with constant probability.
\end{remark}

\begin{remark}\label{rem:random}(Random Initialization)
    The first part of \cref{thm:alg-sparse-dense} requires a randomized initialization for the Metropolis chain, i.e., the initial point $X_0$ is chosen uniformly at random from $\mathcal{X}_{p,k}$. If we further know $\lambda = \Omega( k^{t-\frac{1}{2}} \sqrt{\log p} )$ (e.g., when $k = O(\sqrt{p})$), then this assumption can be removed and the result holds under any initialization.
\end{remark}

\begin{remark}(On applying \cref{thm:canonical})
   We highlight at this point that the proof of \cref{thm:alg-sparse-dense} proceeds by constructing the paths to verify the assumptions of our general positive result \cref{thm:canonical}. The proof relies on a careful first moment method argument alongside the Hanson-Wright inequality to show that for ``most" binary $k$-sparse vectors $v'$ there is an ascending polynomially-long path in the Johnson graph to the planted vector $v^*$, in the sense that the posterior of the sparse Gaussian additive model strictly increases along the path. In particular, for zero-temperature, our result guarantees the success of the randomized greedy algorithm on the Johnson graph for maximizing the posterior. It is important that such a path construction is only possible for ``most" vectors $v'$ leading to guarantees only when the Metropolis process is initialized randomly (see \cref{rem:random}). It is perhaps natural to wonder how one could use our main positive result \cref{thm:canonical}, which assumes the existence of ascending polynomially-long paths to $v^*$ from all states. This is a robustness feature of our main result \cref{thm:canonical}, as we first prove that the Metropolis process will remain in this ``typical'' part of the state space (where polynomially-long ascending paths to $v^*$ exist) for an exponential amount of time with high probability and then use our main positive result on this ``truncated" state space. 
\end{remark}

\subsection{The negative result}
It is natural to wonder whether our obtained positive results is tight. In this section, we prove that in the special case of sparse tensor PCA \cite{luo2022tensor}, the threshold for the Metropolis chain to hit the planted vector $v^*$ fast is indeed captured by the above result.

To focus on the case of sparse tensor PCA we choose for some constant $t \geq 2$, $d=p^{t}$ and $F(v)=\mathrm{vec}(v^{\otimes t})$. This yields of course that $f$ satisfies $f(x)=x^t, x \in \mathbb{R}$. Under this choice of parameters, our model is equivalent with the sparse $t$-tensor PCA model where one observes a $t$-tensor \begin{align} Y=\frac{\lambda}{k^{t/2}} (v^*)^{\otimes t} +W\end{align} for signal $v^* \in \{0,1\}^p, \|v^*\|_0=k$ and noise $W \in \R^{p^t}$ with i.i.d. $N(0,1)$ entries. 

We prove the following lower bound for all local Markov chains on the graph $G_{p,k}$. 
\begin{theorem}[Sparse Tensor PCA Model: Negative Result]
\label{thm:lb}
Consider the sparse $t$-tensor PCA model where $t \ge 2$ is constant.
Assume either that
\begin{enumerate}
    \item $t=2$ and $\omega(1) = k \le p^{\alpha_0}$ where $\alpha_0 \in (0,1)$ is some sufficiently small absolute constant;

    \item $t \geq 3$ and $k = o\left( p^{\frac{t-2}{t+2}} \right)$.
\end{enumerate}
Suppose the signal-to-noise ratio satisfies 
\begin{align*}
    \lambda = o \left( \frac{1}{\sqrt{\log(p/k)}} \min \left\{ k^{t-\frac{1}{2}},\, \frac{p^{t-1}}{k^{t-\frac{3}{2}}} \right\} \right).
\end{align*}
    Then, with high probability for any $\beta$ satisfying
    \begin{align*}
        \beta = \omega\left( \lambda^{\frac{1}{t-1}} k^{\frac{t-2}{2(t-1)}} (\log(p/k))^{\frac{3t-2}{2(t-1)}} \right),
    \end{align*}
    for any reversible Markov chain on the graph $G_{p,k}$ with stationary distribution $\pi_{\beta}$, the first time the chain hits $v^*$ is $p^{\omega(1)}$.
\end{theorem}
\cref{thm:lb} is proved in \cref{sec:NegGAMproof}.

\begin{remark}(Negative result for the posterior temperature)
    By an easy calculation, we note that the conditions of \cref{thm:lb} are satisfied with $\beta_{\mathrm{post}}=\lambda$ whenever
    \[\lambda=\omega\left(k^{1/2}\left(\log\left(\frac{p}{k}\right)\right)^{\frac{3t-2}{2(t-2)}}\right)=\Tilde\omega(\lambda_{\mathrm{STATS}}(t,k)),\]where the $\tilde{\omega}$ notation hides poly-logarithmic terms in $p,k.$
    Therefore, \cref{thm:lb} implies that any reversible Markov chain on $G_{p,k}$ whose stationary distribution is $\pi_{\mathrm{post}}$ has at least $p^{\omega(1)}$ mixing time whenever 
    \[\Tilde\omega(\lambda_{\mathrm{STATS}}(t,k))= \lambda = \Tilde o(\lambda_{\mathrm{MCMC}}(t,k)).\]
\end{remark}

\section{Application: Sparse Linear Regression}

\subsection{The model}
We now focus on sparse regression. Similar to \cref{eq:sparse_set_def}, we have $\mathcal{X}_{p,k}$ for integers $0 \le k \le p$ with the associated Johnson graph $G_{p,k}$. Let $n$ be a positive integer and let $\sigma^2 > 0$. Following the Gaussian sparse regression setting  \cite{10.1214/21-AOS2130}, let us assume that the covariate matrix $X\in\mathbb{R}^{n\times p}$ is a random matrix with i.i.d. standard Gaussian entries. Let the noise $W\in\mathbb{R}^n$ be distributed according to $N(0,\sigma^2I)$. Assume $v^*$ is sampled from the uniform distribution over $\mathcal{X}_{p,k}$, and let the observations be
\[
Y = Xv^* + W
\] Our task is to exactly recover $v^*$ given $(Y,X)$, with high probability as $k$ grows to infinity.

The posterior is now given by 
\[
\pi_{\mathrm{post}}(v) 
\propto \exp(1/2 H(v))
, \quad \forall v \in \mathcal{X}_{p,k}. \label{eq:post}
\] 
where
\[
H(v) = -\norm{Y-Xv}_2^2
\]

\subsection{The positive result}
Similar to our previous case of interest, we study the $G_{p,k}$-Metropolis chain which samples from the (potentially) temperature misparameterized posterior distribution,
\[
\pi_\beta(v) 
\propto \exp(\beta H(v))
, \quad \forall v \in \mathcal{X}_{p,k}.
\] 
where here $\beta>0$ is the inverse temperature parameter.

We work under the high signal to noise ratio assumption for sparse regression, as described in the following (mild) assumption.
\begin{assumption}
\label{ass:sr}
Assume that
\begin{align*}
    \sigma^2\leq \frac{k}{\log p}.
\end{align*}
\end{assumption}

We now present our main positive result for sparse regression. This result holds with high probability as $k\rightarrow\infty$ (and therefore $n,p\rightarrow\infty$ as well).

\begin{theorem}[Sparse Regression: Positive Result]\label{thm:SRpos}
    Consider the Sparse Linear Regression model where \cref{ass:sr} holds. Assume that for a sufficiently large constant $C>0$,
    \[n\geq Ck\log (p/k).\]
    Then the following statements hold:
    \begin{enumerate}
        \item (Moderate $\beta$): For any $\beta = \Omega(\log(kp)/n)$, and any $T=\omega\left(k^3 p \left(\log\left(\frac{pe}{k}\right)+\beta n\right)\right)$, the Metropolis chain with any initialization after $T$ iterations is at $v^*$ with probability at least $2/3-o(1)$.
        \item (Large $\beta$): For any $\beta > 8/C$, and any $T=\omega(k^2 p)$, the Metropolis chain with any initialization after $T$ iterations is at $v^*$ with probability at least $2/3-o(1)$. Moreover, for the Randomized Greedy algorithm (i.e., $\beta = \infty$), once the chain hits $v^*$, with high probability it remains at $v^*$ for all future iterations.
    \end{enumerate}
\end{theorem}
The proof of \cref{thm:SRpos} is given in \cref{sec:PosSRproof}.
\begin{remark}(Hitting time dependence on $\beta$)
    We note that the hitting time of the chain for moderate $\beta$ has a positive dependence on $\beta n$, but once $\beta$ enters the large $\beta$ regime, this dependence disappears, and the hitting time reduces to $O(k^2 p)$.
\end{remark}
\begin{remark}(Positive result for the posterior temperature)
    We note that for a uniform prior on $v^*$, the posterior given $X$ and $Y$ as described in \eqref{eq:post}
corresponds to the case of $\pi_{\beta}$ where $\beta=\beta_{\mathrm{post}}=\frac{1}{2}$. \cref{thm:SRpos} therefore implies that the Metropolis chain can sample from the posterior of sparse regression in polynomial time whenever $C > 16$ or
\[n\geq 16k\log (p/k).\]

We remark that achieving such sampling results from the posterior of high dimensional estimation tasks is a non-trivial task which has seen recent progress in non-sparse settings via diffusion models techniques \cite{montanari2023posterior}.  Our work implies that sampling from the posterior of sparse estimation tasks is also possible in polynomial-time via simple MCMC methods almost all the way down to the LASSO threshold \cite{wainwright_lasso} $n_{\mathrm{ALG}}=\Theta(k \log (p/k))$.
\end{remark}

\subsection{The negative result}
Here we prove the tightness of our positive results for sparse regression. We show that any reversible chain on $G_{p,k}$ (in particular the Metropolis chain) sampling from $\pi_{\beta}$ fails to recover $v^*$ in polynomial time when $n=o( k\log (p/k))$.

\begin{theorem}\label{thm:negresult}[Sparse Regression: Negative Result]
   Assume that for some sufficiently small $c>0$, we have
    \begin{align*}
        n\leq ck\log (p/k)
    \end{align*}
    Assume also that $\sigma^2\leq k$, $k=o(p^{1/3-\eta})$ for some $\eta>0$, and that $\beta$ satisfies
    \begin{align*}
        \beta > \Omega\left(\frac{\log(p/k)}{n}\exp\left(\frac{2k\log(p/k)}{n}\right)\right)
    \end{align*}
   Then w.h.p. as $k\rightarrow\infty$, for any reversible Markov chain on the graph $G_{p,k}$ with stationary distribution $\pi_{\beta}$, the first time the chain hits $v^*$ is $\exp\left(\Omega(k\log (p/k))\right)$.
\end{theorem}
To prove \cref{thm:negresult} (see \cref{sec:NegSRproof}), we establish a bottleneck for the MCMC methods, by significantly generalizing an Overlap Gap Property result from \cite{10.1214/21-AOS2130}, where the authors assume $n =\Omega(k\log k)$ and $k=p^{o(1)}$. In particular, we generalize to arbitrarily small $n$ and any $k=o(p^{1/3})$. 
\begin{remark}(Negative result for the posterior temperature)
    If $\sigma^2$ is non-vanishing, then we note that the conditions of \cref{thm:negresult} are satisfied at the posterior temperature $\beta_{\mathrm{post}}=\frac{1}{2}$ when $n\geq Ck\log(p/k)/(\log(k/\sigma^2+1)=\Theta(n_{\mathrm{STATS}})$ for a large enough constant $C$. Therefore \cref{thm:negresult} implies that any reversible Markov chain on the graph $G_{p,k}$ whose stationary distribution is $\pi_{\mathrm{post}}$ has at least $\exp\left(\Omega(\log(p/k)\right)$ hitting time whenever
    \[\omega(n_{\mathrm{STATS}})=n=o(n_{\mathrm{ALG}}).\]
\end{remark}

\section{Further Comparison with Previous Work}

In this section, we discuss more how our results on sparse Gaussian additive model and sparse regression compare with existing results. 

\subsection{Comparison with \texorpdfstring{\cite{arous2023free}}{[AWZ23]}}For sparse tensor PCA, our results directly compare with \cite{arous2023free}.  \cite{arous2023free} studies the same setting exactly but only under the assumption $t=2$, i.e., in the matrix case, while we study the more general case $t \geq 2$. In this matrix case, they study the same class of low-temperature reversible MCMC methods on the Johnson graph as we consider in this work. They prove that whenever $k=o(p)$ the hitting time of the planted vector is super-polynomial for a part of the conjectured ``computationally hard'' regime that is when
\[\lambda=o(\min\{\lambda_{\mathrm{ALG}}(2,k),(kp)^{1/4}\})=o(\min\{k,(kp)^{1/4}\}).\]
The authors provide this as evidence for the conjectured computational hardness of sparse (matrix) PCA. Importantly they leave as an open question whether one of these MCMC methods succeed in polynomial-time when other polynomial-time methods are known to work i.e. when $\lambda=\omega(\lambda_{\mathrm{ALG}}(2,k))$.

Our negative results prove that when  $k\leq p^{\alpha_0}$ for $\alpha_0>0$ a small enough, a much wider failure takes place even when $t=2$. Specifically, all considered low-temperature MCMC methods in \cite{arous2023free} have super-polynomial hitting time of the planted vector whenever $\lambda=o(k^{3/2}/\sqrt{\log p})$. Moreover, we also provide a positive result saying that the hitting time becomes polynomial for an element of this class (i.e., the Metropolis chain) when $\lambda=\omega(k^{3/2}\sqrt{\log p})$. Hence we pin down the exact threshold, up to logarithmic dependencies, for these MCMC methods to succeed in polynomial-time in sparse PCA, and we highlight that potentially surprisingly it turns out to be much bigger than $\lambda_{\mathrm{ALG}}(2,k)$.

From a technical standpoint, we obtain the mentioned stronger negative MCMC results compared to \cite{arous2023free} because of a significantly improved application of the second moment method (see the proof of \cref{prop:2mm}), which is perhaps notable as already   \cite{arous2023free} used an elaborate second moment method to carry through the same exact argument (when $t=2$). The key new technical ingredient we leverage is some novel ``flatness'' ideas from random graph theory \cite{balister2019dense,gamarnik2019landscape} appropriately adapted to Gaussian tensors. We expect this idea to be broadly useful for similar sparse problems.

\subsection{Comparison with \texorpdfstring{\cite{10.1214/21-AOS2130}}{[GZ22a]}} For sparse regression, the most relevant work is \cite{10.1214/21-AOS2130} that studies the exact same setting as we do. The authors prove that for all $k=o(p)$, if $n=\omega(n_{\mathrm{ALG}})$ a greedy local-search algorithm maximizing the posterior succeeds in polynomial-time. Moreover, they show that if $n=o(n_{\mathrm{ALG}})$ a disconnectivity property which they called Overlap Gap Property for inference takes place, which in particular leads to the failure of the mentioned greedy local search procedure. Importantly, their negative result applies only in the more restricted sparsity regime $k=p^{o(1)}$.

In our work, we prove more general both positive and negative results. For the positive side, we prove that for all $k=o(p)$, a larger family of low-temperature MCMC methods succeeds whenever $n=\omega(n_{\mathrm{ALG}})$. In particular, our positive results include the randomized greedy algorithm maximizing the posterior, as opposed to the vanilla greedy approach of \cite{10.1214/21-AOS2130}. We highlight that this is a non-trivial extension, which is not always known to be true (see e.g., the relevant discussion in \cite{gheissari2023finding} for the planted clique model). On top of that, our positive results hold down to $n=\Theta(n_{\mathrm{ALG}})$ for any $\sigma^2=O( k/\log p)$ which is a significant improvement from the required assumption $\sigma^2=O( \min\{\frac{\log p}{\log \log p}, k\})$ in \cite{10.1214/21-AOS2130}, under a very mild assumption that $k=\omega( \frac{(\log p)^2}{\log \log p})$. Finally, we prove that our negative result that for all $k=o(p^{1/3})$ a large family of low-temperature MCMC methods fails when $n=o(n_{\mathrm{ALG}})$, significantly again generalizing the result of \cite{10.1214/21-AOS2130}.

\subsection{Comparison with \texorpdfstring{\cite{jordanMCMC}}{[YWJ16]} and \texorpdfstring{\cite{zhouDimensionFree}}{[ZYV+22]}}
In terms of positive results for Bayesian models of sparse regression, there are two other very relevant works on the performance of Markov chains, \cite{jordanMCMC} and \cite{zhouDimensionFree}.

In \cite{jordanMCMC}, the authors consider a $k$-sparse regression setting where the goal is to sample from the posterior, given the Zellner hierarchical $g$-prior \cite{zellnerGprior}, a deterministic measurement matrix $X$, and i.i.d. Gaussian noise $W$ of variance $\sigma^2$. Briefly, instead of considering Markov chains operating only on the exactly $k$-sparse vectors as in our work (which is the parameter space of the problem), they consider a ``lifted" state space of all support sets of size below some threshold, and add a regularizer prior to the model that encourages sparsity. They then derive the posterior and study the Metropolis chain\footnote{The graph in this case has two states connected if their Hamming distance is at most two, rather than exactly two as in our case.} to sample from it. They go on to derive polynomial mixing time bounds and hitting time guarantees of the planted vector of the Metropolis chain sampling from the posterior. Their result hold under certain structural deterministic assumptions on the feature matrix $X$.

Comparing with our results where $X$ has i.i.d. standard Gaussian entries, we show that these structural conditions on $X$ are not satisfied with high probability when $n=o(k\sigma)$, where $\sigma^2$ is the noise variance (recall $n_{\mathrm{ALG}}=\Theta(k\log (p/k))$. We present a proof of this claim in \cref{sec:structCondsJordan}. Note that when $\sigma^2=\Theta(k/\log p)$ (the maximum noise level allowed for our positive result to hold), this implies that the required sample size for the MCMC algorithm in \cite{jordanMCMC} to succeed is $n=\Omega\left(k^{1.5}/\sqrt{\log p}\right)$, which significantly larger than the conjectures algorithmic threshold $n_{\mathrm{ALG}}=\Theta(k\log(p/k))$ which is achieved by our analysis of the Metropolis process using our main positive result \cref{thm:canonical}. 

The same setting of sparse regression with deterministic $X$ and Gaussian noise $W$ as in \cite{jordanMCMC} is considered in \cite{zhouDimensionFree}. Here the authors study a different Metropolis chain but on the same state space as \cite{jordanMCMC} (that is of supports of arbitrary size below some threshold), and prove that, under assumptions, it mixes with the posterior rapidly. Of note, the final mixing time bound obtained is independent of $p$. A key feature of their Metropolis chain is the use of an \textit{informed proposal} method, where subsets of variables that have a higher posterior probability are more likely to be proposed as transitions. Such informed proposal schemes have been studied in several previous works \cite{zanellaScalable,zanellaInformed, griffinMixing}. This contrasts with the usual version of Metropolis, where proposed transitions are chosen in an uninformed manner, such as uniformly (the approach taken both in our work and in \cite{jordanMCMC}). However, informed proposals carry a significantly higher computational cost than uninformed ones, as they generally require calculating the full posterior probability of each neighbor of the current state of the chain. This computational cost can be up to $O(pk^2)$ at each iteration, rather than $O(nk)$ in our case (the cost of evaluating $H$ at a single point). The total computational complexity of the algorithm in \cite{zhouDimensionFree} therefore has a linear dependence on $p$, identical to our result.

We also note that the assumptions on $X$ made in \cite{zhouDimensionFree}, while deterministic assumptions, are slightly weaker than those in \cite{jordanMCMC}. Rather than making the same structural assumptions on the design matrix $X$, the authors show that if there exists a path (in Hamming distance) from any set of coordinates (of arbitrary sparsity) to the true set, such that the posterior probability is suitably increasing at each step, then the chain mixes quickly. Then they prove that the structural assumptions in \cite{jordanMCMC} indeed imply their required condition. We highlight that this assumption is in fact very similar to the general assumption we propose for arbitrary state space \cref{assum:paths}.

We highlight, however, some key further differences between our results and that in \cite{zhouDimensionFree}. First, as mentioned above we analyze the uninformed Metropolis process on the $k$-sparse vectors, while they do anayze the informed Metropolis process on sparse vectors of arbitrary size below some threshold. Second, the authors do not prove that their ``ascending paths" assumptions on this extended state space is satisfied even in some relatively simple specific examples such as the case of i.i.d. Gaussian $X$, instead they show that their assumptions are implied by the structural assumptions in \cite{jordanMCMC}. As argued above, these structural assumptions of \cite{jordanMCMC} can fail to hold for i.i.d. Gaussian data when $n=\Theta(n_\mathrm{ALG})$ and $\sigma^2=\Theta(k/\log p)$. On the other hand, we do establish that our \cref{assum:paths} does indeed hold down to $n=\Theta(n_\mathrm{ALG})$ when we focus on the $k$-sparse state space.

\section{Simulations} \label{sec:sim}
Here we present the results of simulations that demonstrate the theoretical results given in the previous sections. For both Sparse PCA and Sparse Linear Regression, we implement the above described Metropolis chain. We then track the proportion of the hidden vector $v^*$ discovered by the chain over time for varying levels of the signal-to-noise ratio. A phase transition is observed, in line with the theoretical results.
\subsection{Sparse Matrix PCA}

We simulate the Metropolis chain for Sparse \textit{Matrix} PCA, i.e. the case where $t=2$. We note that simulating for higher $t$ poses computational tractability issues, as even for as small as $p=1000$, the resulting $3$-tensor has a billion entries. For similar computational concerns, we focus on the relatively sparse case, where $k \leq \sqrt{p}$.

\subsubsection{Specification}
Using $t=2$ gives a model
\[
Y = \lambda v^*v^{*\top} + Z
\]
where $v^*\in\{0,1\}^p$, $\norm{v}_0=k$, and $Z\in\mathbb{R}^{p\times p}$ with i.i.d. standard Gaussian entries. The Hamiltonian for this model is thus given by
\[
H(v) = v^\top Y v.
\]
Our theoretical results give three thresholds for $\lambda$: the information-theoretic threshold is given by
\[
\lambda_{\mathrm{STATS}} = \sqrt{k},
\]
the general algorithmic threshold is predicted to be given by
\[
\lambda_{\mathrm{ALG}} = k,
\]
while the low-temperature MCMC threshold for this model shown to be
\[
\lambda_{\mathrm{MCMC}} = k^{1.5}
\]
These thresholds are affirmed by our simulation results, where the Metropolis chain fails to recover $v^*$ when $\lambda \leq \lambda_{\mathrm{MCMC}}$, but quickly succeeds when $\lambda$ surpasses this threshold.

\subsubsection{Results}

Our implementation is for $p=3500$, $k=15$, and $\beta = 1000$, while we set $\lambda = k^x$ for values of $x$ between $1$ and $2$. Both $v^*$ and the initialisation $v_0$ are chosen uniformly at random independently. The Metropolis chain is then run with these parameters. At each step, we calculate the inner product of the current state $v_t$ with $v^*$, and halt the process when $v_t=v^*$, or after $10^6$ iterations, whichever occurs first.

In \cref{fig:sparsePCA_xVsTime}, we simulate five times for each value of $x$, and plot the median time for the chain to fully recover $v^*$ (or $10^6$ if the chain does not fully recover $v^*$ before then).

We interestingly observe \emph{a clear phase transition}, where for lower values of $\lambda$, the median chain does not succeed in finding $v^*$. In fact, in $72\%$ of simulations where $\lambda < k^{1.5}$, the chain fails to recover a single relevant coordinate of $v^*$ in the first one million iterations. On the other hand, above the critical MCMC threshold of $\lambda \geq k^{1.5}$, the chain quickly succeeds. In particular, for values of $\lambda$ in this regime, $83\%$ of simulations fully recover $v^*$ in fewer than 200 thousand iterations, while $90\%$ achieve full recovery in the first million iterations.

In \cref{fig:sparsePCA_pathPlots} we present in more detail the dynamics of the Metropolis chain in terms of recovering the hidden vector. Here we plot the proportion of the hidden vector $v^*$ recovered by the chain at each time step for the first $\approx 10^5$ iterations. We observe again a clear phase transition. When $\lambda \geq k^{1.5}$, the chain quickly finds the entirety of $v^*$. In contrast, for $\lambda < k^{1.4}$, the median chain never finds even a small fraction of $v^*$.

\begin{figure}[ht]
    \centering
    \includegraphics[width=1\linewidth]{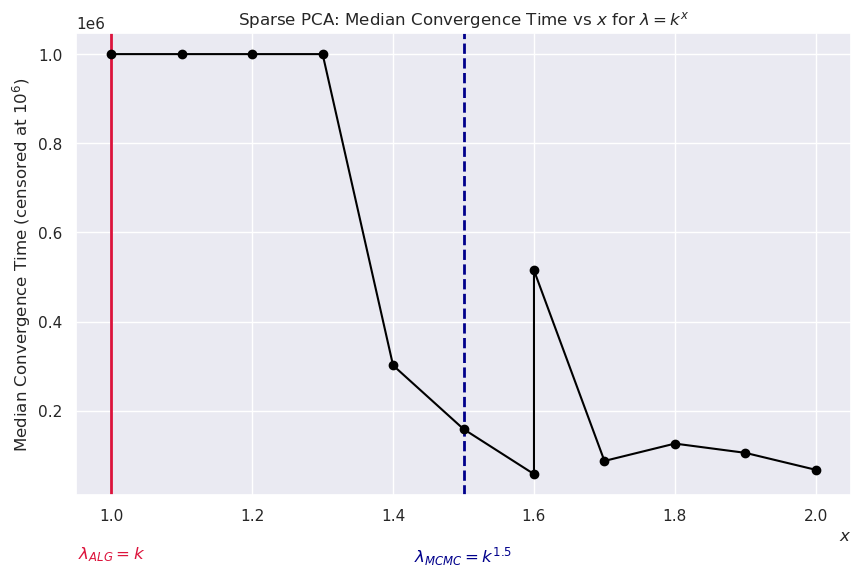}
    \caption{Sparse PCA: Median convergence time (censored at $10^6)$ vs $x$, where $\lambda=k^x$. Here we take $p=3500$, $k=15$, and $\beta = 1000$.}
    \label{fig:sparsePCA_xVsTime}
\end{figure}

\subsection{Sparse Linear Regression}
We also simulate the Metropolis chain for the sparse regression model. In line with our predictions, we also observe a transition point around the predicted algorithmic threshold, where above this point the chain quickly recovers $v^*$, and fails below it.
\begin{figure}[ht]
    \centering
    \includegraphics[width=1\linewidth]{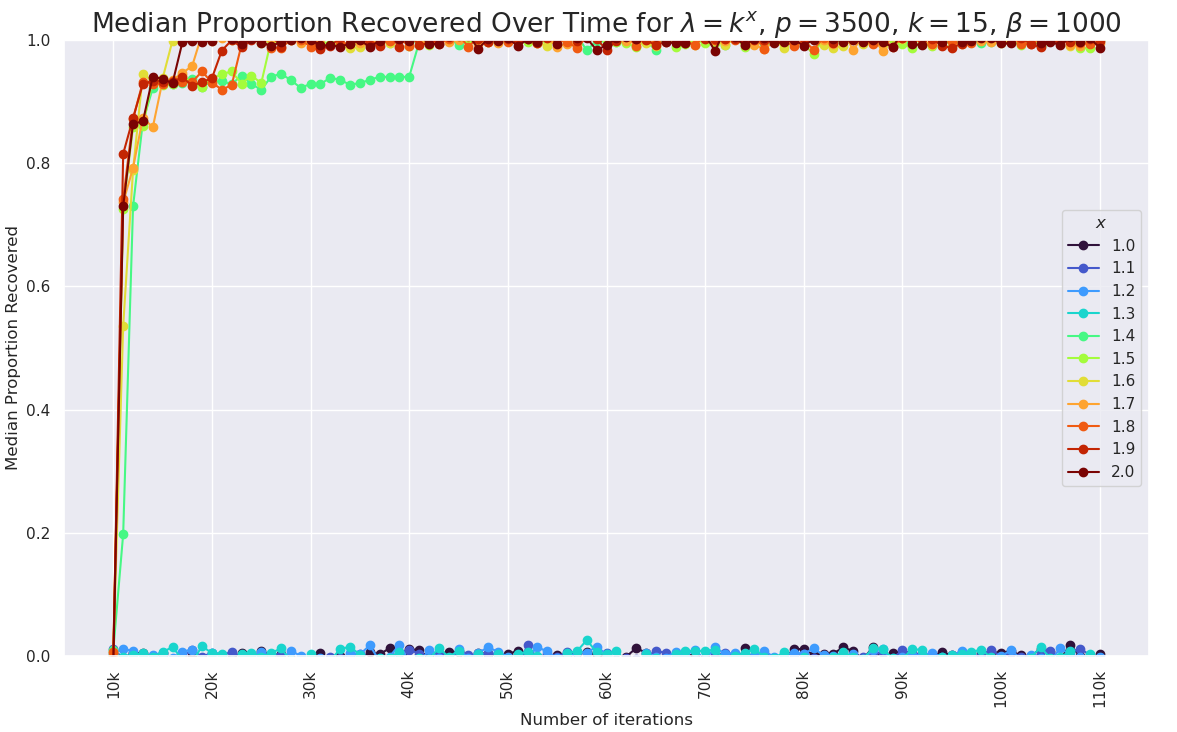}
    \caption{Sparse PCA: Median proportion of $v^*$ recovered over time for different values of $\lambda$. Here $\lambda = k^x$, $p=3500$, $k=15$, and $\beta=1000$. The predicted MCMC threshold for success is $x=1.5$.}
    \label{fig:sparsePCA_pathPlots}
\end{figure}
\subsubsection{Specification}
Our implementation here uses $p=20000$, $k=18$, $\sigma^2 = 0.05$, and $\beta = 100$. We then simulate the Metropolis chain for varying values of the sample size $n$. The information-theoretic threshold here is given by \cite{10.1214/21-AOS2130}
\begin{align*}
    n_{\mathrm{STATS}} = \frac{2k\log (p/k)}{\log\left(\frac{k}{\sigma^2}+1\right)} \approx 43
\end{align*}
Our positive result in \cref{thm:SRpos} provides the correct order for polynomial time convergence, but not the exact constant. We therefore compare to the algorithmic threshold for LASSO, where the exact threshold for success is given by \cite{wainwright_lasso}
\begin{align*}
    n_{\mathrm{LASSO}} = (2k+\sigma^2)\log (p/k) \approx 253
\end{align*}
\subsubsection{Results}
The results of these simulations are displayed in \cref{fig:sparse_regression_nVsTime}. We again observe \emph{a clear phase transition} when the sample size $n$ crosses the algorithmic threshold. When $n < n_{\mathrm{LASSO}}$, most samples time out after a million iterations having failed to find $v^*$. Above $n_{\mathrm{LASSO}}$, the chain quickly finds $v^*$. We can see these dynamics in more detail in \cref{fig:sparse_regression_pathPlots}. Here we group simulations by their value of $n$. Then at each time, we plot the median proportion of $v^*$ recovered at that time within that group. We observe that for values of $n$ above 300, we achieve very rapid convergence. Conversely, below 175 the chain does not appear to be converging at all, despite being significantly above the information-theoretic threshold $n_{\mathrm{STATS}}\approx 43$.

We recognize that the correct constant for the MCMC algorithmic threshold may not be the same as for LASSO, and when $n$ is relatively small as it is here, that can affect the results. Nevertheless, we do see a rapid speed up in convergence when the LASSO threshold is surpassed.

\begin{figure}
    \centering
    \includegraphics[width=1\linewidth]{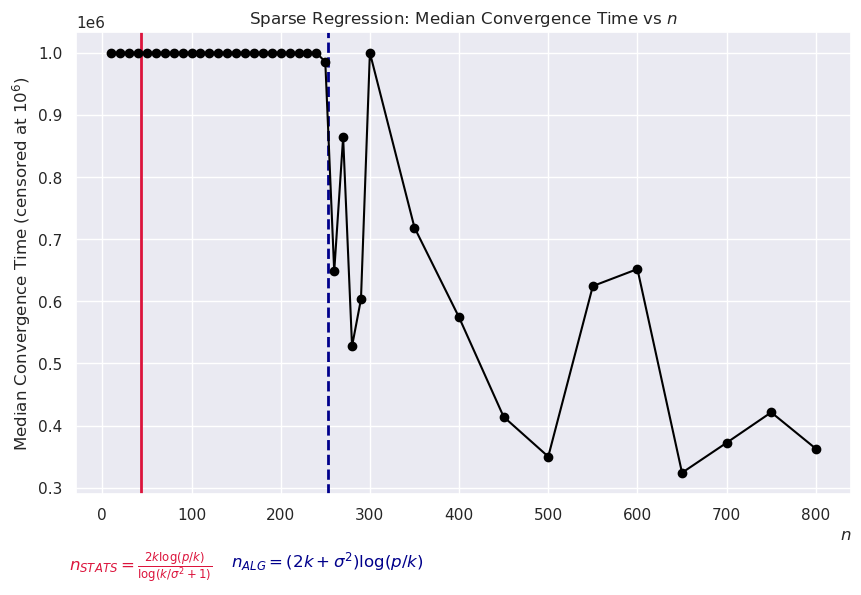}
    \caption{Sparse Regression: Median convergence time (censored at $10^6$) vs $n$, for $p=20000$, $k=18$, $\sigma^2 = 0.05$, and $\beta = 100$.}
    \label{fig:sparse_regression_nVsTime}
\end{figure}
\begin{figure}
    \centering
    \includegraphics[width=1\linewidth]{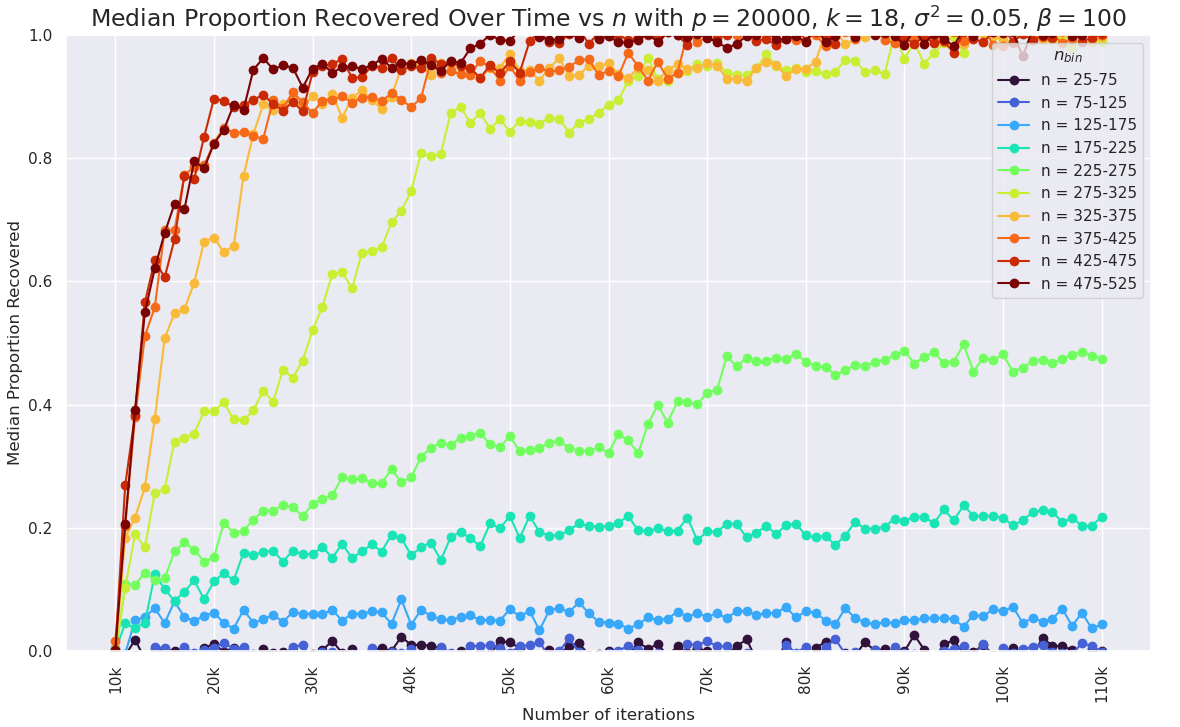}
    \caption{Sparse Regression: Median proportion of $v^*$ recovered over time for different values of $n$. Here $p=20000$, $k=18$, $\sigma^2 = 0.05$ and $\beta = 100$. The information theoretic threshold is $\approx 43$, and the predicted algorithmic threshold is $\approx 253$.}
    \label{fig:sparse_regression_pathPlots}
\end{figure}

\section{Proof of General Positive Result}\label{sec:GenPosResult}
Here we prove \cref{thm:canonical}, the general algorithmic positive result for the Metropolis chain. We highlight that the proof is rather short.
\begin{proof}[Proof of \cref{thm:canonical}] 
For $x,y \in \mathcal{X}$, we write $x \prec y$ if $x$ is an ancestor of $y$ (equivalently, $y$ is a descendant of $x$). 
For part (1), we observe that
\begin{align*}
\frac{1}{\pi_\beta(v^*)}
= \sum_{x \in \mathcal{X}} e^{\beta(H(x)-H(v^*))} 
\leq \sum_{i=0}^\infty \Delta^i e^{-\beta \delta i} 
\le \sum_{i=0}^\infty e^{-2i} \le \frac{4}{3}.
\end{align*}
We thus focus on parts (2) and (3).

    For part (2), we use the standard canonical paths argument for mixing (for example, see chapter 5 of \cite{jerrum2003counting}).
	For any $x,y \in \mathcal{X}$ we define $\gamma_{x,y}$ to be the unique path connecting them in $T$. Note that $|\gamma_{x,y}| \le D$.
    Now for each edge $(a,b)$ in $T$, we upper bound the conductance $t(a,b)$ of it, defined as
    \begin{align*}
        t(a,b) = \frac{1}{\pi(a)P(a,b)} \sum_{(x,y):\, (a,b) \in \gamma_{x,y}}  \pi(x)\pi(y)|\gamma_{x,y}|.
    \end{align*}
    Without loss of generality, we assume $b \prec a$ (this is because $\pi(a)P(a,b) = \pi(b)P(b,a)$ by reversibility and the proof is symmetric for paths containing $(a,b)$ or $(b,a)$). 
    Then we have 
    \begin{align*}
    P(a,b) = \frac{1}{|\Gamma_G(a)|} \ge \frac{1}{\Delta}
    \end{align*}
    since $H(b) > H(a)$. It follows that
    \begin{align*} 
    t(a,b) &= \frac{1}{\pi(a)P(a,b)} \sum_{(x,y):\, (a,b) \in \gamma_{x,y}}  \pi(x)\pi(y)|\gamma_{x,y}|\\
    & \leq D\Delta \sum_{(x,y):\, (a,b) \in \gamma_{x,y}} \frac{\pi(x)\pi(y)}{\pi(a)}\\
    & \leq D\Delta \sum_{x:\, x \succeq a} \frac{\pi(x)}{\pi(a)}
    \end{align*}
    where we observe that for any path $\gamma_{x,y}$ containing $(a,b)$, the starting point $x$ must be a descendant of $a$ or $x = a$.
Hence, it follows that
\begin{align*} 
    t(a,b)  
    \leq D\Delta \sum_{x:\, x \succeq a} e^{\beta(H(x)-H(a))}
    \leq D\Delta \sum_{i=0}^\infty \Delta^i e^{-\beta \delta i} 
    \le D\Delta \sum_{i=0}^\infty e^{-2i}
    =O\left( D\Delta \right).
\end{align*}
Thus, the conductance of the Metropolis chain on $G$ is at most $O(D\Delta)$
and therefore the mixing time is at most $O(D\Delta \left(\log |\mathcal{X}| + \beta \mathcal{R}_H\right))$ using standard results, see \cite{DS91,Sinclair92} or \cite[Theorem 5.2]{jerrum2003counting}. Now, given the mixing time upper bound, part (2) follows directly from part (1). 

Lastly, we consider the case where $\beta \ge \frac{2}{\delta}(\log |\mathcal{X}| + \log(\frac{40\mathcal{R}_H}{\delta}))$ and prove part (3).
We argue that in such a low-temperature regime, the behavior of the Metropolis chain is very similar as the Randomized Greedy algorithm with high probability.

We call a step of the Metropolis chain to be ``good'' if, at current state $X_{t-1} = x \neq v^*$, it selects $y \in \Gamma_G(x)$ which is the parent of $x$ in $T$; note that since $H(y) > H(x)$ by \cref{assum:paths}, the chain will then move to $y$, i.e., $X_t = y$. For completeness, if the current state is $X_{t-1} = v^*$ then we always call this step good no matter which next state it picks or whether it moves.
Thus, each step has a probability at least $1/\Delta$ to be good.
A simple application of the Chernoff bounds yields that, if we run the Metropolis chain for $T=\omega(\frac{\Delta \mathcal{R}_H}{\delta})$ steps then -- since for large enough $n$ it must hold $T \geq \frac{20 \Delta \mathcal{R}_H}{\delta}$ -- we have that at least $\frac{4 \mathcal{R}_H}{\delta}$ of the first $T$ steps of the Metropolis chain are good with probability at least $0.9$. 

Let $m'$ be the number of good steps.
Suppose the first $m = \min\left\{\frac{4 \mathcal{R}_H}{\delta}, m'\right\}$ good steps happen at time $1\le t_1 < t_2 < \cdots < t_m \le T$; we also set $t_0 = 0$.
Since the transition matrix, conditional on no move between time $t_{i-1}$ and $t_i$ is good, is still time-reversible, it holds that
\begin{align*}
    \Pr\left( H(X_{t_i - 1}) \le H(X_{t_{i-1}}) - \frac{\delta}{2} \;\bigg\vert\; X_{t_{i-1}} = x  \right)
    &= \sum_{y:\, H(y) \le H(x) - \frac{\delta}{2}}
    \Pr\left( X_{t_i - 1} = y \mid X_{t_{i-1}} = x \right) \\
    &\le \sum_{y:\, H(y) \le H(x) - \frac{\delta}{2}} \exp\left( \beta H(y) - \beta H(x) \right) \\
    &\le |\mathcal{X}| e^{-\frac{\beta \delta}{2}} \le \frac{\delta}{40 \mathcal{R}_H}.
\end{align*}
Therefore, with probability at least $1-\frac{4 \mathcal{R}_H}{\delta} \cdot \frac{\delta}{40 \mathcal{R}_H} = 0.9$, it holds that $H(X_{t_i - 1}) > H(X_{t_{i-1}}) - \frac{\delta}{2}$ for all $1 \le i \le m$.

Now, by the union bound, with probability at least $0.8$, we have that (1) there are at least $m' \ge \frac{4 \mathcal{R}_H}{\delta}$ good steps; (2) for all $1 \le i \le m=\frac{4 \mathcal{R}_H}{\delta}$ it holds $H(X_{t_i - 1}) > H(X_{t_{i-1}}) - \frac{\delta}{2}$.
We claim that under (1) and (2), it must hold that $X_{t_i - 1} = v^*$ for some $1 \le i \le m$, and the theorem then follows.
Suppose for sake of contradiction that $X_{t_i - 1} \neq v^*$ for all $i$.
Then, by \cref{assum:paths} we have $H(X_{t_i}) \ge H(X_{t_i - 1}) + \delta > H(X_{t_{i-1}}) + \frac{\delta}{2}$.
It follows that $H(X_{t_m}) > H(X_{t_0}) + \frac{\delta m}{2} = H(X_0) + 2\mathcal{R}_H$, contradicting $H(X_{t_m}) - H(X_0) \le \mathcal{R}_H$ by definition. 
\end{proof}

\section{Proofs of the Positive Results in Applications}

\subsection{Proof of Positive Results for the Gaussian Additive Model}\label{sec:PosGAMproof}
Here we apply \cref{thm:canonical} to prove \cref{thm:alg-sparse-dense} under \cref{ass:f}. The first step is to show that \cref{ass:f} implies \cref{assum:paths}.

\subsubsection{Verifying \texorpdfstring{\cref{assum:paths}}{assum:paths}.}
Write $G = G_{p,k}$ for brevity.
For each $x \in \mathcal{X}_{p,k}$ let $\Gamma_{G,v^*}(x) = \{y \in \Gamma_G(x): |y \cap v^*| = |x \cap v^*| + 1 \}$.
That is, $\Gamma_{G,v^*}(x)$ contains all $k$-subsets of $[p]$ obtained by removing an element from $x \setminus v^*$ and adding another element from $v^* \setminus x$.  

In the following lemma, we verify \cref{assum:paths} under \cref{ass:f}.

\begin{lemma}\label{lem:key}
Assume that \cref{ass:f} holds with constants $\eps \in (0,1)$ and $t \ge 1$.
For all positive integers $p,k,\ell_{\min} \in \N^+$ such that $2/\eps \le k \le p$ and $\ell_{\min} \le \eps k$, if
\[
\lambda \ge 
\Omega\left( \frac{k^{t-\frac{1}{2}} }{(\ell_{\min}+1)^{t-1}} \sqrt{\log p} \right),
\]
then, with high probability, for every $x \in \mathcal{X}_{p,k}$ whose overlap $\ell = |x \cap v^*|$ satisfies $\ell_{\min} \le \ell \le k-1$, there exist at least $\frac{1}{2}(k-\ell)^2$ vectors $y \in \Gamma_{G,v^*}(x)$ such that
\[
H(y) - H(x) 
= \Omega\left( \frac{\lambda (\ell_{\min}+1)^{t-1}}{k^t} \right).
\]
\end{lemma}

\begin{proof}
Fix $x \in \mathcal{X}_{p,k}$, and suppose $|x \cap v^*| = \ell$ for some integer $\ell \le k-1$. 
Hence, we have $|y \cap v^*| = \ell+1$ for all $y \in \Gamma_{G, v^*}(x)$.
Observe from \eqref{eq:H-GAM} (the definition of $H$), that 
\begin{align*}
H(y) - H(x)
&= \lambda \left( f\left( \frac{\ip{y}{v^*}}{k} \right) - f\left( \frac{\ip{x}{v^*}}{k} \right) \right) + \ip{F\left( \frac{y}{\sqrt{k}} \right) - F\left( \frac{x}{\sqrt{k}} \right)}{W} \\
&= \lambda \left( f\left(\frac{\ell+1}{k}\right) - f\left(\frac{\ell}{k}\right) \right) + \ip{F\left( \frac{y}{\sqrt{k}} \right) - F\left( \frac{x}{\sqrt{k}} \right)}{W}.
\end{align*}
We deduce from \cref{ass:f} that
\begin{align*}
    f\left(\frac{\ell+1}{k}\right) - f\left(\frac{\ell}{k}\right)
    &\ge f\left(\frac{\ell_{\min} + 1}{k}\right) - f\left(\frac{\ell_{\min}}{k}\right) \\
    &= \int_{\frac{\ell_{\min}}{k}}^{\frac{\ell_{\min}+1}{k}} f'(s) \dif s \\
    &= \Omega\left( \int_{\frac{\ell_{\min}}{k}}^{\frac{\ell_{\min}+1}{k}} s^{t-1} \dif s \right) \\
    &= \Omega\left( \frac{(\ell_{\min}+1)^t - \ell_{\min}^t}{t k^t} \right) \\
    &= \Omega\left( \frac{(\ell_{\min}+1)^{t-1}}{k^t} \right),
\end{align*}
and hence
\begin{align*}
    \lambda \left( f\left(\frac{\ell+1}{k}\right) - f\left(\frac{\ell}{k}\right) \right)
    = \Omega\left( \frac{\lambda (\ell_{\min}+1)^{t-1}}{k^t} \right).
\end{align*}

It suffices to show that there are at least $\frac{1}{2}(k-\ell)^2$ vectors $y \in \Gamma_{G,v^*}(x)$ satisfying
\begin{align}\label{eq:second-term}
    \left| \ip{F\left( \frac{y}{\sqrt{k}} \right) - F\left( \frac{x}{\sqrt{k}} \right)}{W} \right| = O\left( \sqrt{ \frac{\log p}{k} } \right),
\end{align}
and the lemma then follows from the assumption $\lambda \ge C \frac{k^{t-\frac{1}{2}} }{(\ell_{\min}+1)^{t-1}} \sqrt{\log p} $ by taking a sufficiently large constant $C$.

Given $x \in \mathcal{X}_{p,k}$, 
for each $y \in \Gamma_{G,v^*}(x)$ let us define
\[
Z^x_y := \ip{F\left( \frac{y}{\sqrt{k}} \right) - F\left( \frac{x}{\sqrt{k}} \right)}{W}
\qquad\text{and}\qquad
Z^x := (Z^x_y)_{y \in \Gamma_{G,v^*}(x)}.
\]
We want to show that there exists at least half of $y \in \Gamma_{G, v^*}(x)$ that satisfies $|Z^x_y| = O\left( \sqrt{ (\log p)/k } \right)$.

Observe that each $Z^x_y$ is a Gaussian with mean zero and variance given by
\begin{align*}
    \Var\left( Z^x_y \right) = \E\left[ (Z^x_y)^2 \right] &= \norm{F\left( \frac{y}{\sqrt{k}} \right) - F\left( \frac{x}{\sqrt{k}} \right)}_2^2 \\
    &= f\left(\frac{\ip{y}{y}}{k}\right) - 2 f\left(\frac{\ip{x}{y}}{k}\right) + f\left(\frac{\ip{x}{x}}{k}\right) \\
    &= 2\left( f(1) - f\left( 1-\frac{1}{k} \right) \right) = O\left( \frac{1}{k} \right),
\end{align*}
where we use the convexity of $f$ and item (b) of \cref{ass:f}.
Thus, we have
\begin{align}\label{eq:z^x-norm-exp}
\E\left[ \norm{Z^x}_2^2 \right]
= \sum_{y \in \Gamma_{G,v^*}(x)} \E\left[ (Z^x_y)^2 \right] = O\left( \frac{(k-\ell)^2}{k} \right).
\end{align}
Furthermore, for any $y,y' \in \Gamma_{G,v^*}(x)$, the covariance of $Z^x_y$ and $Z^x_{y'}$ is given by
\begin{align*}
\Cov\left( Z^x_y, Z^x_{y'} \right)
&= \ip{F\left( \frac{y}{\sqrt{k}} \right) - F\left( \frac{x}{\sqrt{k}} \right)}{F\left( \frac{y'}{\sqrt{k}} \right) - F\left( \frac{x}{\sqrt{k}} \right)} \\
&= f\left(\frac{\ip{y}{y'}}{k}\right) - f\left(\frac{\ip{x}{y}}{k}\right) - f\left(\frac{\ip{x}{y'}}{k}\right) + f\left(\frac{\ip{x}{x}}{k}\right) \\
&= f\left(\frac{\ip{y}{y'}}{k}\right) - 2 f\left( 1-\frac{1}{k} \right) + f(1).
\end{align*}

Now fix some $y \in \Gamma_{G,v^*}(x)$. 
The number of $y' \in \Gamma_{G,v^*}(x)$ such that $\ip{y}{y'} = k-2$ is $(k-\ell-1)^2$, and for such $y'$ it holds
\begin{align*}
\left| \Cov\left( Z^x_y, Z^x_{y'} \right) \right|
&= f\left( 1-\frac{2}{k} \right) - 2 f\left( 1-\frac{1}{k} \right) + f(1) \\
&\le -\frac{1}{k} \cdot f'\left( 1-\frac{2}{k} \right) + \frac{1}{k} \cdot f'(1) \\
&= O\left( \frac{1}{k^2} \right),
\end{align*}
where we apply the convexity of $f$ and item (c) of \cref{ass:f}. 
Meanwhile, the number of $y' \in \Gamma_{G,v^*}(x)$ such that $\ip{y}{y'} = k-1$ is $2(k-\ell-1)$, and for such $w$ it holds
\begin{align*}
\left| \Cov\left( Z^x_y, Z^x_{y'} \right) \right|
= f(1) - f\left( 1-\frac{1}{k} \right)
\le \frac{1}{k} \cdot f'(1)
= O\left( \frac{1}{k} \right),
\end{align*}
by the convexity of $f$ and item (b) of \cref{ass:f}.
Finally, if $y'=y$ then $\ip{y}{y'} = k$ and
\begin{align*}
\left| \Cov\left( Z^x_y, Z^x_{y'} \right) \right|
= 2 f(1) - 2 f\left( 1-\frac{1}{k} \right) 
= O\left( \frac{1}{k} \right).
\end{align*}
This allows us to bound the operator and Frobenius norms of the covariance matrix $K^x := \Cov(Z^x)$ as follows: 
for the operator norm, 
\begin{align*}
\norm{K^x}_2
\le \norm{K^x}_\infty
&= \max_{y \in \Gamma_{G,v^*}(x)} \sum_{y' \in \Gamma_{G,v^*}(x)} \left| \Cov\left( Z^x_y, Z^x_{y'} \right) \right|\\
&= (k-\ell-1)^2 \cdot O\left( \frac{1}{k^2} \right) 
+ 2(k-\ell-1) \cdot O\left( \frac{1}{k} \right) 
+ O\left( \frac{1}{k} \right) \\
&= O \left( \frac{k-\ell}{k} \right);
\end{align*}
and for the Frobenius norm,
\begin{align*}
    \norm{K^x}_{\mathrm{F}}^2
    &= \sum_{y \in \Gamma_{G,v^*}(x)} \sum_{y' \in \Gamma_{G,v^*}(x)} \left| \Cov\left( Z^x_y, Z^x_{y'} \right) \right|^2 \\
    &= (k-\ell)^2 \left( (k-\ell-1)^2 \cdot O\left( \frac{1}{k^4} \right) + 2(k-\ell-1) \cdot O\left( \frac{1}{k^2} \right) + O\left( \frac{1}{k^2} \right) \right) \\
    &= O\left( \frac{(k-\ell)^3}{k^2} \right).
\end{align*}

It follows from the Hanson--Wright inequality \cref{lem:HW} below that, for any constant $a > 0$, we have
\begin{align*}
& \Pr\left( \norm{Z^x}_2^2 \ge \E\left[ \norm{Z^x}_2^2 \right] + \frac{a(k-\ell)^2 \log p}{k} \right) \\
\le{}&
\exp\left( - \Omega\left( \min\left\{ \frac{a^2(k-\ell)^4 \log^2 p}{k^2 \norm{K^x}_{\mathrm{F}}^2}, \frac{a(k-\ell)^2 \log p}{k \norm{K^x}_2} \right\} \right) \right) \\
={}& \exp\left( - \Omega\left( a (k-\ell) \log p \right) \right).
\end{align*} 
Notice that the total number of vectors in $\mathcal{X}_{p,k}$ with overlap size $\ell$ with $v^*$ is given by
\[
\binom{k}{\ell} \binom{p-k}{k-\ell}
\le k^{k-\ell} p^{k-\ell}
\le \exp(2(k-\ell) \log p).
\]
Therefore, by choosing a constant $a>0$ large enough and recalling \cref{eq:z^x-norm-exp}, it holds with high probability that for all $x \in \mathcal{X}_{p,k}$, 
\[
\norm{Z^x}_2^2 = \sum_{y \in \Gamma_{G,v^*}(x)} (Z^x_y)^2 = O\left( \frac{(k-\ell)^2 \log p}{k} \right).
\]

Assuming this event holds. 
Then for any $x \in \mathcal{X}_{p,k}$, 
there are at most half of $y \in \Gamma_{G,v^*}(x)$ (i.e., $\le (k-\ell)^2/2$ of them) satisfying
\[
(Z^x_y)^2 > \frac{2\norm{Z^x}_2^2}{(k-\ell)^2};
\]
in other words,
at least half of $y \in \Gamma_{G,v^*}(x)$ satisfies
\[
(Z^x_y)^2 \le \frac{2\norm{Z^x}_2^2}{(k-\ell)^2} = O\left( \frac{\log p}{k} \right).
\]
This establishes \cref{eq:second-term}, and the lemma then follows.
\end{proof}

\begin{lemma}[Hanson--Wright Inequality {\cite[Theorem 1.1]{RV13}}]
\label{lem:HW}
Let $Z \sim \mathcal{N}(0,I_n)$ be a standard normal random vector. Let $A \in \R^{m\times n}$ and $X = AZ \in \R^m$. 
Let $K = A A^\trans$ be the covariance matrix of $X$. 
Then for any $a > 0$ it holds
\begin{align*}
\Pr\left( \norm{X}_2^2 \ge \E\left[ \norm{X}_2^2 \right] + a \right) 
\le
\exp\left( - \Omega \left( \min\left\{ \frac{a^2}{\norm{K}_{\mathrm{F}}^2}, \frac{a}{\norm{K}_2} \right\} \right) \right).
\end{align*}
\end{lemma}

\subsubsection{Proof of \texorpdfstring{\cref{thm:alg-sparse-dense}}{thm:alg-sparse-dense}} 
We are now ready to prove \cref{thm:alg-sparse-dense}. 
We first consider the sparse case where $k=O(\sqrt{p \log p})$, where we directly apply \cref{thm:canonical} and \cref{lem:key}.
We need to show that \cref{assum:paths} is satisfied with high probability, and then the results follow from \cref{thm:canonical}.
By \cref{lem:key} with $\ell_{\min} = 0$, for every $x \in \mathcal{X}_{p,k}$ we pick an arbitrary $y \in \Gamma_{G,v^*}(x)$ such that
\begin{align*}
H(y) - H(x) 
= \Omega\left( \frac{\lambda}{k^t} \right),
\end{align*}
and set the parent of $x$ to be $y$; this gives a spanning tree with high probability. 
Note that this automatically establishes the second part of \cref{assum:paths} with $\delta = \Omega( \lambda/k^t )$. 
Observe that the diameter of the tree is at most $2k$, and the maximum degree of the graph $G_{p,k}$ is at most $kp$. 
Hence, \cref{assum:paths} follows.

To upper bound the running time, notice that
\begin{align*}
    \log |\mathcal{X}_{p,k}| = \log \binom{p}{k} = O(k \log(p/k)).
\end{align*}
Also recalling the definition of $H$, the range of $H$ is at most
\begin{align*}
    \mathcal{R}_H \le \lambda f\left( 1 \right) + 2 \max_{x \in \mathcal{X}_{p,k}} \left| \ip{F\left( \frac{x}{\sqrt{k}} \right)}{W} \right|.
\end{align*}
Note that for each $x \in \mathcal{X}_{p,k}$, the random variable $Z_x := \ip{F\left( \frac{x}{\sqrt{k}} \right)}{W}$ is a centered Gaussian with variance $f(1) = O(1)$. 
Thus, $\max_{x \in \mathcal{X}_{p,k}} |Z_x| = O(\sqrt{\log |\mathcal{X}_{p,k}|}) = O(\sqrt{k \log(p/k)})$ with high probability.
It follows that
\begin{align*}
    \mathcal{R}_H = O\left( \lambda + \sqrt{k \log(p/k)} \right) = O(\lambda).
\end{align*}
Therefore, we deduce from \cref{thm:canonical} that when $\beta \lambda = \Omega(k^t \log p)$, the number of iterations required is at most
\begin{align*}
    O\left( 2k \cdot kp \cdot (k \log(p/k) + \beta \lambda) \right) = O\left( \beta \lambda k^2 p \right)
\end{align*}
with high probability, and when $\beta \lambda = \Omega(k^{t+1} \log(p/k))$, it is at most
\begin{align*}
    O\left( kp \cdot \lambda \cdot \frac{k^t}{\lambda} \right) = O\left( k^{t+1} p \right)
\end{align*}
with high probability.

For the final statement of \cref{thm:alg-sparse-dense}, we note that by \cref{lem:key}, for any neighbor $y$ of $v^*$, it holds with high probability that $H(v^*)-H(y)=\Omega\left(\frac{\lambda}{k^t}\right)$. Therefore, once the Randomized Greedy algorithm (i.e. when $\beta=\infty$) hits $v^*$, it remains there for all future iterations with high probability. 

Next we consider the case where $k = \omega(\sqrt{p \log p})$.
The following lemma allows us to couple the evolution of the Metropolis chain together with the restricted Metropolis chain only on $x$'s whose overlap $\ip{x}{v^*} \ge \frac{k^2}{10p}$. 
Hence, any polynomial mixing time in the latter implies the same polynomial mixing time in the former.

\begin{lemma}\label{lem:escape}
    Under the assumptions of \cref{thm:alg-sparse-dense}, and suppose $k = \omega(\sqrt{p \log p})$, 
    if we initialize the Metropolis chain with a uniformly random $X_0 \in \mathcal{X}_{p,k}$, then with high probability it holds $|X_t \cap v^*| \geq \frac{k^2}{10p}$ for all $t \leq e^{O(k^2/p)}$.
\end{lemma}

\begin{proof}
Let $\ell_t = |X_t \cap v^*|$ for all $t\ge 0$.
Note that $\ell_0 = |X_0 \cap v^*| \ge \frac{9k^2}{10p}$ with high probability since $X_0$ is chosen uniformly at random.
Also observe that $\ell_{t+1} - \ell_t \in \{-1,0,1\}$ 
for any $t \ge 0$. We have from the update rules of the Metropolis chain that
\[\mathbb{P}(\ell_{t+1}=\ell_t-1) \leq \frac{\ell_t}{k},\]
and using \cref{lem:key},
\[\mathbb{P}\left( \ell_{t+1}=\ell_t+1 \;\Big\vert\; \ell_t \ge \frac{k^2}{10p} \right) \geq \frac{(k-\ell_t)^2}{2kp}.\] 
This means that on the interval $\ell_t \in [\frac{k^2}{10p}, \frac{k^2}{8p}]$, 
it holds
\begin{align*}
    \mathbb{P}\left( \ell_{t+1}=\ell_t-1 \;\Big\vert\; \frac{k^2}{10p} \le \ell_t \le \frac{k^2}{8p} \right) &\leq \frac{k}{8p}, \\
    \mathbb{P}\left( \ell_{t+1}=\ell_t+1 \;\Big\vert\; \frac{k^2}{10p} \le \ell_t \le \frac{k^2}{8p} \right) &\geq \frac{k}{4p}.
\end{align*}
Namely, $\ell_t$ stochastically dominates a lazy random walk with $\pm 1$ steps and upward drift $\Omega(k/p)$. For such a lazy random walk, the probability of it becoming $< \frac{k^2}{10p}$ when starting at $\frac{k^2}{8p}$ is exponentially small $e^{-\Omega(k^2/p)}$ when run for $e^{O(k^2/p)}$ steps. The lemma then follows.
\end{proof}

Suppose $k = \omega(\sqrt{p \log p})$.
Let $\mathcal{X}' \subseteq \mathcal{X}_{p,k}$ be those vectors whose overlap with $v^*$ is at least $\frac{k^2}{10p}$. 
We define the restricted Metropolis chain to be the Metropolis chain that does not leave $\mathcal{X}'$, i.e., when a vector from $\mathcal{X}_{p,k} \setminus \mathcal{X}'$ is proposed, the chain always rejects it (one can think of it as setting the Hamiltonian to be $-\infty$ outside $\mathcal{X}'$). 
If we could establish polynomial mixing for the restricted Metropolis chain, then the theorem would follow from the equivalence between the restricted Metropolis chain and the standard one as stated in \cref{lem:escape}.
To prove for the restricted chain, we apply \cref{lem:key} with $\ell_{\min} = \frac{k^2}{10p}$ and construct a spanning tree similarly as in the previous case $k = O(\sqrt{p \log p})$; in particular, we achieve
\begin{align*}
    \delta = \Omega\left( \frac{\lambda k^{t-2}}{p^{t-1}} \right)
\end{align*}
for the second part of \cref{assum:paths}.
Thus, by \cref{thm:canonical}, for $\beta \lambda = \Omega(\frac{p^{t-1}}{k^{t-2}} \log p)$ the number of iterations required is at most $O(\beta \lambda k^2 p)$

with high probability, and for $\beta \lambda = \Omega(\frac{p^{t-1}}{k^{t-3}} \log(p/k))$ it is at most
\begin{align*}
    O\left( kp \cdot \lambda \cdot \frac{p^{t-1}}{\lambda k^{t-2}} \right) = O\left( \frac{p^t}{k^{t-3}} \right)
\end{align*}
with high probability. Again, for the final statement of \cref{thm:alg-sparse-dense}, we note that for any neighbor $y$ of $v^*$, it holds with high probability that $H(v^*)-H(y)=\Omega\left(\frac{\lambda k^{t-2}}{p^{t-1}}\right)$. Therefore, once the Randomized Greedy algorithm (i.e. when $\beta=\infty$) hits $v^*$, it remains there for all future iterations with high probability. This completes the proof of \cref{thm:alg-sparse-dense}.

\subsection{Proof of Positive Results for Sparse Linear Regression}\label{sec:PosSRproof}
Here we prove \cref{thm:SRpos} under \cref{ass:sr}. First we show that \cref{ass:sr} implies \cref{assum:paths}.

\subsubsection{Verifying \texorpdfstring{\cref{assum:paths}}{assum:paths}}
In order to verify \cref{assum:paths} we present a Binary - Local Search Algorithm that is a binary version of the Local Search Algorithm in \cite{10.1214/21-AOS2130}. The proof of its convergence is essentially the same as for the Local Search Algorithm, except we only need to consider support recovery. We use $e_1,e_2,\dots$ to refer to the standard basis vectors in $\mathbb{R}^p$. We write $X_i$ to refer to the $i^{th}$ column of $X$.

\begin{algorithm}[H]
\caption{Binary - Local Search Algorithm (B-LSA)}
\KwIn{An exactly $k$-sparse binary vector $v$ with support $S$.}
\BlankLine

\Repeat{termination}{
    \For{$i \in S$ \textbf{and} $j \in [p]$}{
        $\mathrm{err}_{i}(j) \gets \|Y - Xv + X_i - X_j\|^2_2$\;
    }

    $(i_1, j_1) \gets \mathrm{argmin}_{i \in S, j \in [p]} \mathrm{err}_{i}(j)$\;

    \If{$\|Y - Xv + X_{i_1} - X_{j_1}\|_2^2 < \|Y - Xv\|_2^2$}{
        $v \gets v - e_{i_1} + e_{j_1}$\;
        $S\gets \mathrm{supp}(v)$\;
    }
    \Else{
        Terminate and output $v$\;
    }
}

\end{algorithm}

B-LSA looks at all the possible coordinate ``flips" we could make in one step, and chooses the one that minimises the squared error. It repeats until no local improvement is possible.
In the following lemma, we verify the second part of \cref{assum:paths} under \cref{ass:sr}, with $\delta = \frac{n}{4}$ by showing that B-LSA always reduces the squared error by a factor of at least order $n$ at each step.
\begin{lemma}\label{lem:SR_ver_assump2}
    Assume that \cref{ass:sr} holds. If $v \neq v^*$, then for any $v'$ obtained from $v$ in one iteration of B-LSA, with high probability it holds that 

    \[\lVert Y-Xv'\rVert^2_2 \leq \lVert Y-Xv\rVert^2_2-\frac{1}{4}n\]
\end{lemma}
The proof of \cref{lem:SR_ver_assump2} is deferred to \cref{sec:proofSR_ver_assump2}. In the next lemma, we verify the first part of \cref{assum:paths} with $D = O(k)$. 
\begin{lemma}\label{lem:SR_ver_assump1}
    Assume that \cref{ass:sr} holds. With high probability as $k\rightarrow\infty$, B-LSA with uniform random initialisation $v_0$ terminates in at most $O(k)$ iterations and outputs $v^*$.
    \begin{proof}
        By \cref{lem:SR_ver_assump2}, the only possible termination point is $v^*$, and since we decrease the error by at least $\frac{n}{4}$ at each step, we can take at most $\frac{4}{n}\lVert Y-Xv_0 \rVert_2^2$ total steps.
        
        It remains to show that with high probability $\frac{4}{n}\lVert Y-Xv_0 \rVert_2^2=O(k)$. Note that if $v_0$ satisfies $\ip{v}{v^*}=\ell$, then since
        \begin{align*}
            \lVert Y-Xv_0\rVert_2^2 &= \lVert Xv^*+W-Xv_0\rVert_2^2
        \end{align*} and for each $i$, $1\leq i\leq n$, we have 
        \begin{align*}
            \left(Xv^*+W-Xv_0\right)_i \sim N(0, 2(k-\ell)+\sigma^2)
        \end{align*}
        and so 
        \begin{align*}
            \lVert Y-Xv_0\rVert_2^2 \sim (2(k-\ell)+\sigma^2)\chi^2(n)
        \end{align*}
        We will now use the standard $\chi^2$ bound (for instance, page 1325 of \cite{LaurentMassart}) that\newline$\mathbb{P}(\chi^2(n)\geq n + 2\sqrt{nx}+2x)\leq e^{-x}$ for all $x>0$ for $x=n$ to have that the probability that 
        \begin{align*}
            \mathbb{P}\left(\frac{4}{n}\lVert Y-Xv_0\rVert_2^2 \geq 20\left(2(k-\ell)+\sigma^2\right)\right) \leq \exp(-n)=o(1).
        \end{align*}

        Note that since $\sigma^2 = o(k)$ by \cref{ass:sr}, we have that $20\left(2(k-\ell)+\sigma^2\right)\leq 50k$ for large enough $k$, and therefore
        \begin{align*}
            \mathbb{P}\left(\frac{4}{n}\lVert Y-Xv_0\rVert_2^2 \geq 50k\right)=o(1)
        \end{align*}
        as $n\rightarrow \infty$. Thus $\frac{4}{n}\lVert Y-Xv_0\rVert_2^2=O(k)$ with high probability as $k\rightarrow\infty$.
    \end{proof}
\end{lemma}

\begin{lemma}\label{lem:SR_Pi_min}
    We assume \cref{ass:sr} and $n \geq C k\log(p/k)$ for some $C>0$. Then with high probability as $k\rightarrow\infty$, it holds that 
    \[
    \mathcal{R}_H = O(kn)
    \]
    \begin{proof}
        As $H(v) \leq 0$ for all $v$, we can only focus on bounding $\min_v H(v)$, the lowest value of the Hamiltonian (i.e. the worst squared error). Each $H(v)$ can be directly checked to follow a $(2(k-\ip{v}{v^*} + \sigma^2)\chi^2_n$ distribution. We use a union bound to see that for $x = 2k\log\left(\frac{pe}{k}\right)$ it holds that for some $U \sim \chi^2(n)$,
        \begin{align*}
    \mathbb{P}\left(\max_v \lVert Y-Xv\rVert_2^2 \geq \left(n+2\sqrt{nx}+2x\right)(2k+\sigma^2)\right)
    &\leq \binom{p}{k}\mathbb{P}\left(U \geq n+2\sqrt{nx}+2x \right)\\
    &\leq \left(\frac{pe}{k}\right)^k e^{-x}=o(1).
\end{align*}
and thus we get that with high probability (and recalling that $\sigma^2 = o(k)$),
\begin{align*}
    \max_v \lVert Y-Xv\rVert_2^2 \leq \left(n+2\sqrt{2nk\log\left(\frac{pe}{k}\right)}+4k\log\left(\frac{pe}{k}\right)\right)(2k+\sigma^2) 
\end{align*}
Note that since $n=\Omega(k\log(p/k)$ and $\sigma^2 = o(k)$, all terms are $O(kn)$, and the result follows.
    \end{proof}
\end{lemma}
Therefore \cref{assum:paths} is satisfied with $D = O(k)$ and $\delta = \frac{n}{4}$, and $\mathcal{R}_H=O(kn)$.
We now have everything we need to prove \cref{thm:SRpos}.
\begin{proof}[Proof of \cref{thm:SRpos}]
    We note that in this case that $\Delta$, the maximum degree of the graph, is given by $\Delta = \max_{v \in \mathcal{X}_{p,k}} |\Gamma_G(v)| = k(p-k)$, as each $v$ has $k$ 1's and $(p-k)$ 0's. By \cref{lem:SR_ver_assump2}, \cref{lem:SR_ver_assump1}, and \cref{lem:SR_Pi_min}, we have that \cref{assum:paths} holds with $D = O(k)$, $\delta = \frac{n}{4}$, and $\mathcal{R}_H=O(kn)$. We also see that the cardinality of the state space $\mathcal{X}$ is given by $\binom{p}{k}$. We therefore apply the first two parts of \cref{thm:canonical} to see that if
    \[
    \beta \geq \frac{1}{\delta}(\log\Delta+2) = \frac{4}{n}\left(\log(k(p-k))+2\right) = \Omega\left(\frac{\log(kp)}{n}\right)
    \]
    then the chain mixes in at most
    \begin{align*}
        &O\left(D\Delta\left(\log \left|\mathcal{X}\right|+\beta\mathcal{R}_H\right)\right)\\
         = &O\left(k(k)(p-k)\left(\log\binom{p}{k}+\beta kn\right)\right)\\
         =&O\left(k^3 p\left(\log \left(\frac{pe}{k}\right)+\beta n\right)\right)
    \end{align*}
    iterations, where we use $\binom{p}{k}\leq \left(\frac{pe}{k}\right)^k$. This proves the moderate $\beta$ statement.
    For the large $\beta$ statement, we apply the final part of \cref{thm:canonical}. This part requires
    \begin{align*}
    \beta &\geq \frac{2}{\delta}\left(\log\left|\mathcal{X}\right|+\log\left(\frac{40\mathcal{R}_H}{\delta}\right)\right)\\
    &=\frac{8}{n}\left(k\log\left(\frac{pe}{k}\right)+\log\left(\frac{160\mathcal{{R}_H}}{n}\right)\right)
    \end{align*}
    Note that since $n\geq Ck\log(\frac{p}{k})$, $\frac{8}{n}k\log\left(\frac{pe}{k}\right)$ is at most $8/C$ plus an $o(1)$ term. For the second term: $\mathcal{R}_H=O(kn)$, so $\frac{8}{n}\log\left(\frac{160\mathcal{R}_H}{n}\right)=O\left(\frac{\log k}{n}\right)$, which is dominated by the first term since $n=\omega(k)$. Therefore the condition on $\beta$ is satisfied whenever $\beta$ is larger than $8/C$. It then follows by \cref{thm:canonical} that the Metropolis chain after 
    \begin{align*}
        O\left(\frac{\Delta \mathcal{R}_H}{\delta}\right) &= O\left(\frac{k(p-k)(kn)}{n}\right)=O\left(k^2 p\right)
    \end{align*}
    iterations is at $v^*$ with probability at least $2/3-o(1)$. Finally, we note that by \cref{lem:SR_ver_assump2} that for any neighbor $v$ of $v^*$, it holds with high probability that $H(v^*)-H(v)\geq n/4$. Therefore, with high probability, the Randomized Greedy algorithm (i.e. $\beta=\infty$) never leaves $v^*$ once it hits it.
\end{proof}

\section{Proofs of Negative Results for the Gaussian Additive Model}\label{sec:NegGAMproof}
Here we show the tightness of our positive results for the Gaussian Additive Model by proving \cref{thm:lb}, a superpolynomial lower bound on the mixing time for Sparse PCA. The general approach we follow for both this lower bound, and the lower bound in sparse regression that follows, is a careful implementation of the ``Overlap Gap Property for inference'' framework from \cite{10.1214/21-AOS2130}, which is inspired by statistical physics. This Overlap Gap Property framework has been shown in previous works (e.g., in \cite{gamarnik2021overlap, gamarnik2019landscape,arous2023free}) to directly imply mixing time lower bounds.

In this work though, to keep the presentation simple and focused on Markov chains, we avoid defining exactly what the Overlap Gap Property is, but follow only the proof steps required to derive directly the desired mixing lower bound. We direct the reader to \cite{arous2023free} for the more details on the underlying connections.

\subsection{Auxiliary Lemmas}

For the proof we first need some auxiliary lemmas.

We start with two folklore inequalities, see e.g., \cite{savage1962mills}.

\begin{lemma}\label{lem:mill}
Let $Z,Z_{\rho} \sim N(0,1)$ with covariance $\rho$. For large $x>0$ we have
\begin{align}\label{eq:mill_1}
    \mathbb{P}(Z \geq x) = \left( 1 + o_x(1) \right) \frac{1}{\sqrt{2 \pi} x} e^{-x^2/2},
\end{align}and also 
and
\begin{align}\label{eq:mill_2}
    \mathbb{P}(Z \ge x \land Z_{\rho}\geq x) \leq \frac{(1+\rho)^2}{2\pi \sqrt{1-\rho^2} x^2} e^{-x^2/(1+\rho)}.
\end{align}
\end{lemma}

We also need the following lemma.

\begin{lemma}\label{lem:binom}
    For all $ 1 \leq k = o(p)$ and $k^2/p \ll \ell \ll k$, it holds for all large enough $p$ that
    \begin{align*}
    \log\left( \binom{k}{\ell}\binom{p-k}{k-\ell} \right) - \log \binom{p}{k}
    \leq - \frac{\ell}{2}.
    \end{align*}
\end{lemma}

\begin{proof}
Let $q = \floor{k^2/p}$ denote the ``typical'' overlap of a random $k$-sparse binary vector $x$ with the hidden signal $v^*$.
Note that $q = 0$ if $k < \sqrt{p}$, and $q = o(k)$ since $k = o(p)$.
We will show that
\begin{align*}
    \log\left( \frac{ \binom{k}{\ell}\binom{p-k}{k-\ell}}{\binom{k}{q} \binom{p-k}{k-q}} \right) \leq - \frac{\ell}{2}.
\end{align*}
Since $\binom{p}{k} \ge \binom{k}{q} \binom{p-k}{k-q}$, the lemma then follows.
Observe that
\begin{align*}
    \log\left( \frac{ \binom{k}{\ell}\binom{p-k}{k-\ell}}{ \binom{k}{q} \binom{p-k}{k-q}} \right) 
    = \sum_{m=q}^{\ell-1} \log \left( \frac{ \binom{k}{m+1} \binom{p-k}{k-m-1} }{ \binom{k}{m} \binom{p-k}{k-m} } \right).
\end{align*}
For large enough $p$, for all $q \leq m \leq \ell-1$,

    \begin{align*}
        \frac{ \binom{k}{m+1} \binom{p-k}{k-m-1} }{ \binom{k}{m} \binom{p-k}{k-m} } =\frac{(k-m)^2}{(m+1)(p-2k+m+1)} 
        \le \frac{2k^2}{(m+1)p} \le \frac{2(q+1)}{m+1},
    \end{align*}where in the last inequality we used that $k \leq p/4$ for large enough $p$.
Hence,
\begin{align*}
    \log \left( \frac{ \binom{k}{m+1} \binom{p-k}{k-m-1} }{ \binom{k}{m} \binom{p-k}{k-m} } \right)
    \le \log\left( \frac{2(q+1)}{m+1} \right) \le \frac{2(q+1)}{m+1} - 1.
\end{align*}
It follows that
\begin{align*}
    \log\left( \frac{ \binom{k}{\ell}\binom{p-k}{k-\ell}}{ \binom{k}{q} \binom{p-k}{k-q}} \right) 
    &\le \sum_{m=q}^{\ell-1} \left( \frac{2(q+1)}{m+1} - 1 \right) \\
    &= - (\ell-q) + 2(q+1) \sum_{m=q}^{\ell-1} \frac{1}{m+1} \\
    &\le - \ell + q + 4(q+1) \log(\ell/(q+1)) \le - \frac{\ell}{2},
\end{align*}
where the last inequality follows from $q \ll \ell$ and thus $\log(\ell/(q+1)) \ll \ell/(q+1)$.
\end{proof}

\subsection{Proof Strategy}
We now explain the proof strategy we follow. For simplicity, we will use the notation of \eqref{eq:stpca} in this proof section, i.e. with
\begin{align}
    Y=\lambda (v^*)^{\otimes t}/k^{t/2}+W
\end{align}
where $W$ has i.i.d. $N(0,1)$ entries.

Let us consider the following key random functions (analogues of the key quantities defined in \cite{10.1214/21-AOS2130}) which control the maximum log-likelihood value that can be achieved by any binary $k$-sparse $x$ of overlap $l=0,1,\ldots,k$ with $v^*$:
\begin{align} 
\Gamma_\ell &:= \frac{1}{k^{t/2}} \max_{\langle x,v^*\rangle=\ell} \langle Y, x^{\otimes t} \rangle 
= \lambda \frac{\ell^t}{k^t} + \frac{1}{k^{t/2}} \max_{\langle x,v^*\rangle=\ell}\langle W, x^{\otimes t} \rangle; \\
\Gamma_{[0,\ell]} &:= \max_{0\le m \le \ell} \Gamma_{m}
\ge \frac{1}{k^{t/2}} \max_{0 \le \langle x,v^*\rangle \le \ell}\langle W, x^{\otimes t} \rangle.
\end{align}

In \cite{10.1214/21-AOS2130} the authors use the non-monotonicity of $\Gamma_\ell$ to obtain the MCMC lower bound. To achieve a similar result we need to understand the concentration properties of $\Gamma_\ell$.

The following upper bound holds for $\Gamma_{\ell}$ with high probability.
\begin{proposition}\label{prop:1mm} 
Suppose $k=o(p)$.
Then for any $\ell=o(k)$ and any sequence $A_p =\omega(1)$ it holds w.h.p.~that
    \begin{align}\label{eq:1mm} 
    \Gamma_\ell \leq 
    \lambda \frac{\ell^t}{k^t} + \sqrt{2 \log \left( \binom{k}{\ell} \binom{p-k}{k-\ell} \right) - \log \left(k\log(p/k)\right) +A_p}.
    \end{align}
\end{proposition}

\begin{proof}
    For any $x$, observe that $k^{-t/2} \langle W, x^{\otimes t} \rangle$ is a standard Gaussian.
    By a union bound and \eqref{eq:mill_1}, the probability that \eqref{eq:1mm} does not hold is at most 
    \begin{align*}
    & \frac{1+o(1)}{\sqrt{2\pi}} \frac{\exp\left( \frac{1}{2} \log \left(k\log(p/k)\right) - \frac{1}{2} A_p \right)}{\sqrt{2 \log \left( \binom{k}{\ell} \binom{p-k}{k-\ell} \right) - \log \left(k\log(p/k)\right) + A_p}} \\
={}& \frac{1+o(1)}{\sqrt{2\pi}} \sqrt{ \frac{k\log(p/k)}{2 \log \left( \binom{k}{\ell} \binom{p-k}{k-\ell} \right) - \log \left(k\log(p/k)\right) + A_p} } \exp\left( -\frac{1}{2} A_p \right)\\
={}& O\left( \exp\left( -\frac{1}{2} A_p \right) \right)
= o(1).
    \end{align*} For the second to last equality we used that since $\ell=o(k)$, it holds
    \[\log \left( \binom{k}{\ell}\binom{p-k}{k-\ell} \right) \geq \log \binom{p-k}{k-\ell} \geq (1-o(1)) k\log (p/k)\]
    This shows \cref{eq:1mm}.
\end{proof}

The following lower bound also holds with high probability.

\begin{proposition} \label{prop:2mm} 
Assume either that
\begin{enumerate}[(1)]
    \item $t=2$ and $k = o(p^{\alpha})$ where $\alpha \in (0,1)$ is some sufficiently small absolute constant;

    \item $t \ge 3$ and $k = o\left( p^{\frac{t-2}{t+2}} \right)$.
\end{enumerate}  
Then for any $\omega(k^2/p) = \ell = o(k)$ and any sequence $A_p =\omega(1)$ it holds w.h.p.~that
\begin{align*}
    \Gamma_{[0,\ell]} \geq \sqrt{2 \log \binom{p}{k} - \log \left(k\log(p/k)\right) - A_p}.
\end{align*}
\end{proposition}

Using \cref{prop:1mm,prop:2mm} we conclude that for any $q \ll \ell \ll k$, it holds w.h.p.~for any $1 \ll A_p \ll \ell$ that
    \begin{align*}
        & \Gamma_\ell - \Gamma_{[0,\ell]} - \lambda \frac{\ell^t}{k^t}\\
        \leq{}&  \sqrt{2 \log \left( \binom{k}{\ell} \binom{p-k}{k-\ell} \right) - \log \left(k\log(p/k)\right) +A_p} \\
        & - \sqrt{2 \log \binom{p}{k} - \log \left(k\log(p/k)\right) -A_p} \\
        \le{}& \frac{ 2 \log \left( \binom{k}{\ell} \binom{p-k}{k-\ell} \right) - 2 \log \binom{p}{k} + 2 A_p}{\sqrt{2 \log \left( \binom{k}{\ell} \binom{p-k}{k-\ell} \right) - \log \left(k\log(p/k)\right) +A_p} + \sqrt{2 \log \binom{p}{k} - \log \left(k\log(p/k)\right) -A_p}}\\
        \le{}& - \frac{\ell}{16\sqrt{k \log(p/k)}},
    \end{align*}where we have used \cref{lem:binom} for the last inequality.

We choose
\begin{align*}
    \ell=\ell^*:=\floor{\left( \frac{k^{t-\frac{1}{2}}}{32 \lambda \sqrt{\log(p/k)}} \right)^{\frac{1}{t-1}}}
\end{align*}
which satisfies that $k^2/p \ll \ell^* \ll k$, and conclude by direct inspection that
\begin{align}\label{eq:gap}
    \Gamma_{\ell^*}-\Gamma_{[0,\ell^*]} = -\Omega\left( \left( \frac{k}{\lambda^{2/t} \log(n/k)} \right)^{\frac{t}{2(t-1)}} \right),
\end{align}and in particular $\Gamma_\ell$ must be \emph{strictly decreasing} for some $\ell \in [0,\ell^*]$. We remark this is the key property which leads to the presence of the Ovelap Gap Property as in \cite{10.1214/21-AOS2130}.

As we discussed, we now bypass the details of Overlap Gap Property, and focus directly on the implied bottleneck. For $\ell=\ell^* $ let $\mathcal{X}_{p,k}^{[0,\ell]}$ denote the set of $p$-dimensional $k$-sparse binary vector $x$ with $\langle x,v^* \rangle  \leq \ell$, and $\mathcal{X}_{p,k}^\ell$ be the set of those with $\langle x,v^* \rangle = \ell$ which is the inner boundary of $\mathcal{X}_{p,k}^{[0,\ell]}$. 
It is easy to check that 
\[
\frac{\pi_{\beta} \left( \mathcal{X}_{p,k}^\ell \right)}{\pi_{\beta} \left( \mathcal{X}_{p,k}^{[0,\ell]} \right)} 
\leq \binom{p}{k} \exp\left( \beta\left( \Gamma_\ell - \Gamma_{[0,\ell]} \right) \right),
\]
which using \eqref{eq:gap} implies 
\[
\frac{\pi_{\beta} \left( \mathcal{X}_{p,k}^\ell \right)}{\pi_{\beta} \left( \mathcal{X}_{p,k}^{[0,\ell]} \right)} = 
\exp\left( k\log (e p/k) - \Omega\left( \beta \left( \frac{k}{\lambda^{2/t} \log(p/k)} \right)^{\frac{t}{2(t-1)}} \right) \right).
\]
\cref{thm:lb} then follows from \cite[Claim 2.1]{MWW09} (see also \cite[Proposition 2.2]{WellsCOLT20}).

The rest of this section aims to prove \cref{prop:2mm}. The proof argument follows from two separate delicate applications of the second moment method of independent interest.
In \cref{sec:lb-t>=3}, we present calculations for the first and second moments and establish \cref{prop:2mm} when $t \ge 3$.
In \cref{sec:lb-t=2} we consider the case $t=2$ where a more sophisticated conditional second moment is required.

\subsection{The case \texorpdfstring{$t \ge 3$}{t >= 3}}
\label{sec:lb-t>=3}

We now proceed with proving the required lower bounds on $\Gamma_\ell$. 

\subsubsection{First and second moments calculation}
Fix some $v^* \in \mathcal{X}_{p,k}$.
Recall that $\mathcal{X}_{p,k}^{[0,\ell]}$ is the set of all $p$-dimensional $k$-sparse binary vectors with overlap at most $\ell$ with $v^*$.
We are interested in the number of those $x$'s such that
    \begin{align}\label{eq:goal_1}
   \frac{1}{k^{t/2}} \langle W,x^{\otimes t} \rangle \geq \sqrt{2 \log \binom{p}{k} - \log \left(k\log(p/k)\right) - A_p}.
    \end{align}
Define the random subset (determined by the Gaussian noise $W$)
\begin{align*}
    \mathcal{E} = \left\{x \in \mathcal{X}_{p,k}: \text{$x$ satisfies \eqref{eq:goal_1}} \right\},
\end{align*}
and let $Z = \left| \mathcal{E} \cap \mathcal{X}_{p,k}^{[0,\ell]} \right|$.

\begin{lemma}\label{lem:2nd-moment}
Suppose $x,y$ are chosen uniformly at random from $\mathcal{X}_{p,k}$, and let $\mathbb{P}$ denote the joint distribution of $x,y$ and the Gaussian noise $W$.
For any $\omega(k^2/p)=\ell=o(k)$, 
we have
\begin{align}\label{eq:1st-moment}
    \mathbb{E}[Z] = (1 - o(1)) \binom{p}{k} \mathbb{P}(x \in \mathcal{E}),
\end{align}
and
\begin{align}\label{eq:2nd-moment}
    \frac{\mathbb{E}[Z^2]}{(\mathbb{E}[Z])^2} 
    &= (1\pm o(1)) \frac{\mathbb{P}(x \in \mathcal{E} \land y \in \mathcal{E})}{\mathbb{P}(x \in \mathcal{E})^2}.
\end{align}
\end{lemma}

\begin{proof}
Instead of fixing some $v^* \in \mathcal{X}_{p,k}$ beforehand, we think of the joint measure $\mathbb{P}$ as first generating the Gaussian noise $W$ and independently two binary vectors $x,y \in \mathcal{X}_{p,k}$ uniformly at random, and next generating $v^* \in \mathcal{X}_{p,k}$ independently and uniformly at random.
By linearity of expectation we deduce that
\begin{align*}
    \E[Z] &= \sum_{x \in \mathcal{X}_{p,k}^{[0,\ell]}} \Pr\left(x \in \mathcal{E}\right) \\
    &= |\mathcal{X}_{p,k}| \; \mathbb{P}\left( \ip{x}{v^*} \le \ell \land x \in \mathcal{E}\right) \\
    &= |\mathcal{X}_{p,k}| \; \mathbb{P}\left(\ip{x}{v^*} \le \ell\right) \; \mathbb{P}\left(x \in \mathcal{E}\right) \\
    &= (1-o(1)) |\mathcal{X}_{p,k}| \; \mathbb{P}\left( x \in \mathcal{E}\right),
\end{align*}
where the second to last equality is because the two events $\ip{x}{v^*} \le \ell$ and $x$ satisfying \eqref{eq:goal_1} are independent, and the last is due to $\mathbb{P}(\ip{x}{v^*} > \ell) = o(1)$ whenever $\ell \gg k^2/p$.

Similarly, by linearity of expectation we have for the second moment that
\begin{align*}
    \E[Z^2] &= \sum_{x,y \in \mathcal{X}_{p,k}^{[0,\ell]}} \Pr\left( x \in \mathcal{E} \land y \in \mathcal{E} \right) \\
    &= |\mathcal{X}_{p,k}|^2 \; \mathbb{P}\left( \ip{x}{v^*} \le \ell \land \ip{y}{v^*} \le \ell \land x \in \mathcal{E} \land y \in \mathcal{E} \right) \\
    &= |\mathcal{X}_{p,k}|^2 \; \mathbb{P}\left( \ip{x}{v^*} \le \ell \land \ip{y}{v^*} \le \ell \mid x \in \mathcal{E} \land y \in \mathcal{E} \right) \; \mathbb{P}\left( x \in \mathcal{E} \land y \in \mathcal{E} \right).
\end{align*}
Observe that
\begin{align*}
    &\mathbb{P}\left( \ip{x}{v^*} > \ell \lor \ip{y}{v^*} > \ell \mid x \in \mathcal{E} \land y \in \mathcal{E} \right)\\
    \le &\mathbb{P}\left( \ip{x}{v^*} > \ell \mid x \in \mathcal{E} \land y \in \mathcal{E} \right)
    + \mathbb{P}\left( \ip{y}{v^*} > \ell \mid x \in \mathcal{E} \land y \in \mathcal{E} \right) \\
    = &\mathbb{P}\left( \ip{x}{v^*} > \ell \right)
    + \mathbb{P}\left( \ip{y}{v^*} > \ell \right)
    = o(1),
\end{align*}
where the equality in the second line is similarly because the two events $\ip{x}{v^*} > \ell$ and both $x,y$ satisfying \eqref{eq:goal_1} are independent, and the same for $y$. 
Therefore, we conclude that
\begin{align*}
    \E[Z^2] = (1-o(1)) |\mathcal{X}_{p,k}|^2 \; \mathbb{P}\left( x \in \mathcal{E} \land y \in \mathcal{E} \right).
\end{align*}
The lemma then follows.
\end{proof}

\subsubsection{Two technical lemmas}
For any $t\geq 3$ and $m=0,1,\ldots,k-1$ we define
\begin{align}\label{eq:km}
K_m
:= \sqrt{ \frac{(1+m^t/k^t)^3}{1-m^{t}/k^{t}} } \exp\left( -\frac{m^t}{k^t+m^t} \log k \right).
\end{align}
The following technical lemma holds.
\begin{lemma}\label{lem:low_ov-km}
Let $t\geq 2$ and $K_m$ defined in \eqref{eq:km}. Then 
for all sufficiently large $k$ we have
\begin{align*}
    K_{m} &\le 1, \quad \text{for $0\le m \le \floor{k/2}$}; \\
    K_{m} &= O(1), \quad \text{for $\floor{k/2}+1 \le m \le k-1$}.
\end{align*}
\end{lemma}

\begin{proof}
Suppose $m = \gamma k$ where $0 \le \gamma \le 1-1/k$ and $\gamma k \in \N$.
Observe that
\begin{align*}
    K_m^2 = \frac{(1+m^t/k^t)^3}{1-m^{t}/k^{t}} \exp\left( -\frac{2m^t}{k^t+m^t} \log k \right)
    = \frac{(1+\gamma^t)^3}{1-\gamma^t} \exp\left( -\frac{2\gamma^t}{1+\gamma^t} \log k \right).
\end{align*}

Consider first the case $0\le \gamma \le 1/2$. 
Notice that
\begin{align*}
    \frac{(1+\gamma^t)^3}{1-\gamma^t} 
    \le \exp\left( 3\gamma^t + \frac{\gamma^t}{1-\gamma^t} \right)
    \le \exp\left( 5\gamma^t \right),
\end{align*}
where we use $1+x \le e^x$ for all $x \ge 0$ and $\frac{1}{1-x} = 1+\frac{x}{1-x} \le \exp(\frac{x}{1-x})$ for all $x \in [0,1]$.
Furthermore, we also have
\begin{align*}
    \exp\left( -\frac{2\gamma^t}{1+\gamma^t} \log k \right)
    \le \exp\left( - \gamma^t \log k \right).
\end{align*}
Therefore, we deduce that
\begin{align*}
    K_m^2 \le \exp\left( - \gamma^t (\log k - 5) \right) \le 1
\end{align*}
when $k$ is sufficiently large.

Next, we consider the case where $\frac{1}{2} < \gamma \le 1-1/k$.
We observe that
\begin{align*}
    K_m^2 
    \le \frac{8}{1-\gamma^t} \exp\left( -\frac{2\gamma^t}{1+\gamma^t} \log k \right)
    = 8 \exp\left( \log\left( \frac{1}{1-\gamma^t} \right) -\frac{2\gamma^t}{1+\gamma^t} \log k \right).
\end{align*}
Let $y = k(1 - \gamma^t)$. Since we have
\begin{align*}
    1 - \gamma^t \ge 1 - \left( 1-\frac{1}{k} \right)^t 
    \ge 1 - e^{-t/k}
    \ge \frac{t}{k+t} \ge \frac{1}{k},
\end{align*}
we know that $1 \le y \le k$.
Rewriting in terms of $y$, we obtain
\begin{align*}
    \log\left( \frac{1}{1-\gamma^t} \right) -\frac{2\gamma^t}{1+\gamma^t} \log k
    &= \log\left( \frac{k}{y} \right) -\frac{2k-2y}{2k-y} \log k \\
    &= \frac{y}{2k-y} \log k - \log y \\
    &\le \frac{y}{k} \log k - \log y.
\end{align*}
If $1 \le y \le 3$, then $\frac{y}{k} \log k - \log y \le 3(\log k)/k \le 1$ for $k$ large enough. 
If $3 < y \le k$, since the function $(\log y)/y$ is decreasing in this regime, it holds $\frac{y}{k} \log k - \log y \le 0$.
Therefore, we conclude with $K_m^2 \le 8e$, which shows the lemma.
\end{proof}

For any $t\geq 2$ and $m=0,1,\ldots,k-1$ we define
\begin{align}\label{eq:lm}
L_m
&:= \frac{\binom{k}{m}\binom{p-k}{k-m}}{\binom{p}{k}} \exp\left( \frac{2m^t}{k^t+m^t} \log \binom{p}{k} \right) \nonumber\\
&= \binom{k}{m}\binom{p-k}{k-m} \exp\left( - \frac{k^t-m^t}{k^t+m^t} \log \binom{p}{k} \right).
\end{align}
The following technical lemma holds.
\begin{lemma}\label{lem:low_ov-lm}
Let $t\geq 2$ and $L_m$ defined in \eqref{eq:lm}. 
Then for all sufficiently large $p$ and $k$ we have:
\begin{itemize}
    \item If $k = o\left( p^{\frac{t-1}{t}} \right)$, then 
    \begin{align*}
        \sum_{m=0}^{\floor{k/2}} L_{m} \le 1+o(1)
        \quad\text{and}\quad
        \sum_{m=\floor{k/2}+1}^{\ceil{\left( 1-\frac{1}{t^2} \right)k}-1} L_{m} = o(1).
    \end{align*}
    \item If furthermore $t \ge 3$ and $k = o\left( p^{\frac{t-2}{t+2}} \right)$, then 
    \begin{align*}
        \sum_{m=\floor{k/2}+1}^{k-1} L_{m}=o(1).
    \end{align*}
\end{itemize}
\end{lemma}

\begin{proof}

We define $r = \floor{ \frac{2 k^{2-\frac{1}{t}}}{p^{1-\frac{1}{t}}} }$.
Note that $r = o(k)$ since $k = o(p)$, and $r = 0$ if $2 k^{2-1/t} < p^{1-1/t}$.
Pick the largest $\eps > 0$ such that $\eps k \in \N$ and $8t^2 \eps^{t-1} \le 1$.
Let $\xi > 0$ be the smallest such that $\xi k \in \N$ and $1-\xi \le \frac{1}{t^2}$.
We consider four cases.

\bigskip
\textbf{Case 1: $0\le m \le r$.}
Note that for sufficiently large $p$ we have from $k = o(p)$ that
\begin{align*}
    \log \binom{p}{k} \le k \log(ep/k) \le 2k \log(p/k).
\end{align*}
For all $0 \le m \le r$, we deduce that
\begin{align*}
    \exp\left( \frac{2m^t}{k^t+m^t} \log \binom{p}{k} \right) 
    &\le \exp\left( \frac{2r^t}{k^t} \cdot 2k \log(p/k) \right) \\
    &\le \exp\left( \frac{2^{t+2} k^{t}}{p^{t-1}} \log(p/k) \right)
    =1+o(1).
\end{align*}
Hence we conclude, 
\begin{align*}
    \sum_{m = 0}^r L_m 
    \leq (1+o(1)) \sum_{m = 0}^r \frac{\binom{k}{m}\binom{p-k}{k-m}}{\binom{p}{k}}
    \le 1+o(1).
\end{align*}

\bigskip
\textbf{Case 2: $r+1\le m \le \eps k$.}
First notice the following relations.
\begin{align*}
    \frac{L_{m+1}}{L_m} = \frac{\binom{k}{m+1}}{\binom{k}{m}} \frac{\binom{p-k}{k-m-1}}{\binom{p-k}{k-m}} 
    \exp\left( \left( \frac{k^t-m^t}{k^t+m^t} - \frac{k^t-(m+1)^t}{k^t+(m+1)^t} \right) \log \binom{p}{k} \right).
\end{align*}
For all $r \le m \le \eps k - 1$, using that $k=o(p)$ and $m+1 \ge r+1 \ge \frac{2 k^{2-1/t}}{p^{1-1/t}}$, we have
\begin{align*}
    \frac{\binom{k}{m+1}}{\binom{k}{m}} \frac{\binom{p-k}{k-m-1}}{\binom{p-k}{k-m}} = \frac{(k-m)^2}{(m+1)(p-2k+m+1)}
    \le \frac{2k^2}{(m+1)p}
    \le \frac{k^{1/t}}{p^{1/t}}.
\end{align*}
We also have, using $m+1 \le \eps k$ and $8t^2 \eps^{t-1} \le 1$ that,
\begin{align*}
    \frac{k^t-m^t}{k^t+m^t} - \frac{k^t-(m+1)^t}{k^t+(m+1)^t}
    &= \frac{2k^t\left( (m+1)^t - m^t \right)}{(k^t+m^t)(k^t+(m+1)^t)} \\
    &\le \frac{2 k^t \cdot t (m+1)^{t-1}}{k^{2t}}
    \le \frac{2t \eps^{t-1}}{k} \le \frac{1}{4tk}.
\end{align*}
Combining everything above, we deduce that for all $r \le m \le \eps k - 1$,
\begin{align*}
    \frac{L_{m+1}}{L_m} \le \frac{k^{1/t}}{p^{1/t}} \exp\left( \frac{1}{4tk} \cdot 2k \log(p/k) \right)
    = \frac{k^{\frac{1}{2t}}}{p^{\frac{1}{2t}}}. 
\end{align*}
Since $k = o(p)$, it follows that
\begin{align*}
    \sum_{m = r+1}^{\eps k} L_m \le \sum_{m = r+1}^{\eps k} L_r \left( \frac{k}{p} \right)^{\frac{m-r}{2t}}
    = O\left( L_r \left( \frac{k}{p} \right)^{\frac{1}{2t}} \right)
    = o(L_r) = o(1),
\end{align*}
where we use $L_r = O(1)$ which was shown in Case 1.

\bigskip
\textbf{Case 3: $\eps k + 1 \le m \le \xi k - 1$.}
Assume that $m=\gamma k$ for some arbitrary $\gamma \in (\eps, \xi)$ (so that $\gamma k \in \mathbb{N}$). 
Recall that $(p/k)^k \le \binom{p}{k} \le (pe/k)^k$, and hence we have
\begin{align}
    \binom{k}{m} &= \binom{k}{k-m} \le \left( \frac{ek}{k-m} \right)^{k-m} = \left( \frac{e}{1-\gamma} \right)^{k-m}; \nonumber\\
    \binom{p-k}{k-m} &\le \left( \frac{e(p-k)}{k-m} \right)^{k-m} \le \left( \frac{e p}{(1-\gamma)k} \right)^{k-m}. \label{eq:two-binom}
\end{align}
Then we have by direct expansion that
\begin{align}
    L_m &= \binom{k}{m} \binom{p-k}{k-m} \binom{p}{k}^{- \frac{k^t-m^t}{k^t+m^t}} \nonumber\\
    &\le \left( \frac{e}{1-\gamma} \right)^{k-m} \left( \frac{e p}{(1-\gamma)k} \right)^{k-m} \left( \frac{p}{k} \right)^{-k \cdot \frac{1-\gamma^t}{1+\gamma^t}} \nonumber\\
    &= \left( \left( \frac{e}{1-\gamma} \right)^2 \left( \frac{p}{k} \right)^{1-\frac{1-\gamma^t}{(1-\gamma)(1+\gamma^t)}} \right)^{k-m}. \label{eq:case-large-ov}
\end{align}
Notice that $\frac{e}{1-\gamma} = O(1)$ since $\gamma < \xi$. 
Furthermore, we have
\begin{align*}
    \frac{1-\gamma^t}{(1-\gamma)(1+\gamma^t)}
    = \frac{1+\gamma+ \cdots + \gamma^{t-1}}{1+\gamma^t}
    \ge \frac{1+\gamma}{1+\gamma^t}
    = 1 + \frac{\gamma(1-\gamma^{t-1})}{1+\gamma^t}
    = 1 + \Omega(\eps)
\end{align*}
where the second to last inequality follows from $\gamma^t \le \gamma^{t-1} \le \xi^{t-1} \le (1-\frac{1}{2t^2})^{t-1}$ and $\gamma \ge \eps$.
Therefore, we deduce from \cref{eq:case-large-ov} that
\begin{align*}
    L_m = \left( O(1) \cdot \left( \frac{p}{k} \right)^{-\Omega(\eps)} \right)^{k-m},
\end{align*}
and since $k=o(p)$,
\begin{align*}
    \sum_{m = \eps k+1}^{\xi k-1} L_m 
    = \sum_{m = \eps k+1}^{\xi k-1} \left( O(1) \cdot \left( \frac{k}{p} \right)^{\Omega(\eps)}  \right)^{k-m}
    \le O(1) \cdot \left( \frac{k}{p} \right)^{\Omega(\eps)} 
    = o(1).
\end{align*}

\bigskip
\textbf{Case 4: $\xi k \le m \le k-1$.}
Again we assume $m=\gamma k$, so that $\xi \le \gamma \le 1- \frac{1}{k}$ and hence
\begin{align}\label{eq:gamma-ineq0}
\frac{e}{1-\gamma} \le ek.
\end{align}
We claim that for all $\gamma \in [\xi,1]$ it holds
\begin{align}\label{eq:gamma-ineq}
    \frac{1-\gamma^t}{(1-\gamma)(1+\gamma^t)} \ge \frac{t}{2}.
\end{align}
By the AM--GM inequality
\begin{align*}
    \frac{1-\gamma^t}{1-\gamma}
    = \sum_{i=0}^{t-1} \gamma^i
    \ge t \gamma^{\frac{t-1}{2}}.
\end{align*}
Thus, to establish \cref{eq:gamma-ineq} it suffices to show
\begin{align*}
    2\gamma^{\frac{t-1}{2}} \ge 1+\gamma^t.
\end{align*}
Write $x = 1-\gamma \in [0,1/t^2]$, and we have
\begin{align*}
    1+\gamma^t &= 1+(1-x)^t \le 1 + (1 - tx + t^2 x^2) = 2 - tx + t^2 x^2 \le 2 - (t-1)x;\\
    2\gamma^{\frac{t-1}{2}} &= 2(1-x)^{\frac{t-1}{2}} \ge 2\left( 1 - \frac{t-1}{2}x \right) = 2 - (t-1)x.
\end{align*}
\cref{eq:gamma-ineq} then follows.
Combining \cref{eq:gamma-ineq0,eq:gamma-ineq} with \cref{eq:case-large-ov},
we deduce that
\begin{align*}
    L_m \le \left( e^2 k^2 \cdot \left( \frac{k}{p} \right)^{\frac{t}{2}-1} \right)^{k-m}
    \le \left( e^2 \frac{k^{\frac{t}{2}+1}}{p^{\frac{t}{2}-1}} \right)^{k-m}.
\end{align*}
Since $k = o(p^{\frac{t-2}{t+2}})$ by the assumption, it follows that
\begin{align*}
    \sum_{m = \xi k}^{k-1} L_m 
    \le \sum_{m = \xi k}^{k-1} \left( e^2 \frac{k^{\frac{t}{2}+1}}{p^{\frac{t}{2}-1}} \right)^{k-m}
    = O\left( \frac{k^{\frac{t}{2}+1}}{p^{\frac{t}{2}-1}} \right)
    = o(1).
\end{align*}
Since clearly $\eps < 1/2 < \xi$, this completes the proof of the lemma.
\end{proof}

\subsubsection{Proof of \texorpdfstring{\cref{prop:2mm}}{Proposition~\ref{prop:2mm}} when \texorpdfstring{$t \geq 3$}{t >= 3}}

By \cref{lem:2nd-moment,eq:mill_1}, we deduce that
\begin{align*}
\mathbb{E} [Z] 
\geq{} & \frac{1+o(1)}{\sqrt{2\pi}} \frac{\exp\left( \frac{1}{2} \log \left(k\log(p/k)\right) + \frac{1}{2} A_p \right)}{\sqrt{2 \log \binom{p}{k} - \log \left(k\log(p/k)\right) - A_p}} \\
={}& \frac{1+o(1)}{\sqrt{2\pi}} \sqrt{ \frac{k\log(p/k)}{2 \log \binom{p}{k} - \log \left(k\log(p/k)\right) - A_p} } \exp\left( \frac{1}{2} A_p \right)\\
={}& \Omega\left( \exp\left( \frac{1}{2} A_p \right) \right)
= \omega(1).
\end{align*}

For the second moment, following from \cref{lem:2nd-moment,eq:mill_1,eq:mill_2} we have
\begin{align*}
    & (1-o(1)) \left( \frac{\mathbb{E} [Z^2]}{(\mathbb{E}[Z])^2} - \frac{1}{\mathbb{E}[Z]} \right) \\
    \le{}& \sum_{m=0}^{k-1} \mathbb{P}(\ip{x}{y} = m) \, \frac{\mathbb{P}\left( x \in \mathcal{E} \land y \in \mathcal{E} \mid \ip{x}{y} = m \right)}{\mathbb{P}\left( x \in \mathcal{E} \right)^2} \\
    \le{}& \sum_{m=0}^{k-1} 
    \frac{\binom{k}{m}\binom{p-k}{k-m}}{\binom{p}{k}} 
    \frac{(1+m^t/k^t)^2}{\sqrt{1-m^{2t}/k^{2t}}} 
    \exp\left( \frac{m^t}{k^t+m^t} \left( 2 \log\binom{p}{k}- \log \left(k\log(p/k)\right) - A_p \right) \right) \\
    \le{}& \sum_{m=0}^{k-1} \frac{\binom{k}{m}\binom{p-k}{k-m}}{\binom{p}{k}} \exp\left( \frac{2m^t}{k^t+m^t} \log\binom{p}{k} \right)
    \cdot \sqrt{ \frac{(1+m^t/k^t)^3}{1-m^{t}/k^{t}} } \exp\left( - \frac{m^t}{k^t+m^t} \log k \right) \\
    ={}& \sum_{m=0}^{k-1} L_m K_m \le 1+o(1),
\end{align*}
where the last inequality follows from \cref{lem:low_ov-lm,lem:low_ov-km}.
Since $\E[Z] = \omega(1)$ and $\mathbb{E} [Z^2] \ge (\mathbb{E}[Z])^2$, we conclude that
\begin{align*}
    \frac{\mathbb{E} [Z^2]}{(\mathbb{E}[Z])^2} = 1+o(1).
\end{align*}
\cref{prop:2mm} then follows from an immediate application of the second moment method (e.g., the Paley--Zygmund inequality).

\subsection{The case \texorpdfstring{$t=2$}{t = 2}}
\label{sec:lb-t=2}

We next consider the case $t=2$ where a specialized argument is needed.
Recall that $\mathcal{X}_{p,k} = \{x \in \{0,1\}^p: \norm{x}_0 = k\}$ is the set of $k$-sparse binary vectors.

We first need a definition.

\begin{definition}
    We call a $k$-sparse binary $x \in \mathcal{X}_{p,k}$ to be \emph{flat} if for all $\ell=0,1,\ldots,k $ and any $\ell$-subset $s$ of $x$, i.e., such that $\langle s,x \rangle=\|s\|_0=\ell$, it holds 
    \begin{align}\label{eq:flat}
    \left| s^\trans Ws- \frac{\ell^2}{k^2} x^\trans Wx \right| \leq 2 \sqrt{(k^2-\ell^2) \log \binom{k}{\ell}}.
    \end{align}
\end{definition} 

It turns out for any specific $x \in \{0,1\}^p$ with $\|x\|_0=k$, the vector $x$ is flat with high probability; furthermore, the flatness of $x$ is independent of the value of $x^\trans Wx$.

\begin{lemma}\label{lem:flat}
    Suppose $k = \omega(1)$. 
    For any $x \in \{0,1\}^p$ with $\|x\|_0=k$, $x$ is flat with high probability; furthermore, whether $x$ is flat is independent of the random variable $x^\trans Wx$.
\end{lemma}

\begin{proof}
    Fix any $x \in \{0,1\}^p, \|x\|_0=k$ and for any $\ell =1,2,\ldots,k-1$ fix any $s \in \{0,1\}^p$ with $\langle s,x \rangle=\|s\|_0=\ell$. Note that $( \frac{1}{\ell} s^\trans Ws, \frac{1}{k} x^\trans Wx)$ is a bivariate pair of standard normal distributions with correlation $\rho=\ell/k$. Hence, there exists a standard normal $Z_s$ which is independent of $x^\trans Wx$ such that 
    \[\frac{1}{\ell} s^\trans Ws = \frac{\rho}{k} x^\trans Wx + \sqrt{1-\rho^2}Z_s\]
    and therefore
    \begin{align*}
        s^\trans Ws - \frac{\ell^2}{k^2} x^\trans Wx = \frac{\ell}{k} \sqrt{k^2-\ell^2}Z_s.
    \end{align*}

    Hence for any $T>0$ and any $r \in \R$,
    \begin{align*}
        & \mathbb{P}\left( \left| s^\trans Ws - \frac{\ell^2}{k^2} x^\trans Wx \right| \geq \frac{\ell}{k} \sqrt{k^2-\ell^2} T \;\bigg\vert\; x^\trans Wx = r \right) \\
        ={}& \mathbb{P}\left(|Z_s| \geq T \right)
        \leq \frac{1+o(1)}{\sqrt{2\pi} T} \exp\left( -\frac{T^2}{2} \right).
    \end{align*} 
    Letting $T=2 \sqrt{\log\binom{k}{\ell}}$ and taking a union bound over all the possible $\binom{k}{\ell}$ choices of $s$ gives that 
    \begin{align*}
        \mathbb{P}\left( \forall s: \left| s^\trans Ws - \frac{\ell^2}{k^2} x^\trans Wx \right| \geq \frac{2\ell}{k} \sqrt{(k^2-\ell^2) \log \binom{k}{\ell}} \;\Bigg\vert\; x^\trans Wx = r \right)
        \le 
        \frac{1}{\binom{k}{\ell}}.
    \end{align*}
    Since $\binom{k}{\ell} = k$ for $\ell=1,k-1$ and $\binom{k}{\ell} \ge \frac{k(k-1)}{2}$ for $2\le \ell \le k-2$, a union bound over $\ell$ yields that with probability $1 - O(1/k) = 1 - o(1)$ independent of $x^\trans Wx$, for all $s$ it holds 
\[ \left| s^\trans Ws - \frac{\ell^2}{k^2} x^\trans Wx \right| \le \frac{2\ell}{k} \sqrt{(k^2-\ell^2) \log \binom{k}{\ell}}
\le 2 \sqrt{(k^2-\ell^2) \log \binom{k}{\ell}}\]
which establishes the flatness of $x$ as wanted.
\end{proof}

We call now $Z$ the number of $x$ such that
\begin{align}\label{eq:goal_1_t_2}
   \frac{1}{k} x^\trans Wx  \geq \sqrt{2 \log \binom{p}{k} - \log \left(k\log(p/k)\right) - A_p}.
    \end{align} 
Moreover, let $Z_{\textsf{flat}}$ be the number of flat $x$ satisfying \eqref{eq:goal_1_t_2}.
Define
\begin{align*}
    \mathcal{E} = \left\{x \in \mathcal{X}_{p,k}: \text{$x$ satisfies \eqref{eq:goal_1_t_2}} \right\}
\quad \text{and} \quad
    \mathcal{F} = \left\{x \in \mathcal{X}_{p,k}: \text{$x$ is flat} \right\}.
\end{align*}
Hence, $Z = |\mathcal{E}|$ and $Z_{\textsf{flat}} = |\mathcal{E} \cap \mathcal{F}|$ by definitions.
By the linearity of expectation and \eqref{eq:mill_1}, as before
\begin{align*}
\mathbb{E} [Z] 
\geq{} & \frac{1+o(1)}{\sqrt{2\pi}} \frac{\exp\left( \frac{1}{2} \log \left(k\log(p/k)\right) + \frac{1}{2} A_p \right)}{\sqrt{2 \log \binom{p}{k} - \log \left(k\log(p/k)\right) - A_p}} \\
={}& \frac{1+o(1)}{\sqrt{2\pi}} \sqrt{ \frac{k\log(p/k)}{2 \log \binom{p}{k} - \log \left(k\log(p/k)\right) - A_p} } \exp\left( \frac{1}{2} A_p \right)\\
={}& \Omega\left( \exp\left( \frac{1}{2} A_p \right) \right)
= \omega(1).
    \end{align*} 
Therefore, by \cref{lem:flat} and linearity of expectation we directly have
    \begin{align*}
    \mathbb{E} [Z_{\textsf{flat}}] = (1-o(1)) \mathbb{E} [Z]=\omega(1).
    \end{align*}

Now expanding the second moment to squared first moment ratio, we have similarly as \cref{lem:2nd-moment} that
    \begin{align*}
        \frac{\mathbb{E}\left[Z_{\textsf{flat}}^2\right]}{\left( \mathbb{E}[Z_{\textsf{flat}}] \right)^2} - \frac{1}{\mathbb{E}[Z_{\textsf{flat}}]} 
        \leq (1+o(1)) \sum_{m=0}^{k-1} \mathbb{P}(\ip{x}{y} = m) \, \frac{\mathbb{P}\left( x,y \in \mathcal{E} \cap \mathcal{F} \mid \ip{x}{y} = m \right)}{\mathbb{P}\left( x \in \mathcal{E} \cap \mathcal{F} \right)^2}.
   \end{align*}
We focus on proving the right-hand side of the last inequality is $1+o(1)$ by splitting into two cases.

\bigskip
\textbf{Case 1: $m \le \frac{3}{4}k$.} 
For this case of $m$ we ``ignore'' the flatness condition and use
\begin{align*}
    \mathbb{P}\left( x,y \in \mathcal{E} \cap \mathcal{F} \mid \ip{x}{y} = m \right) \leq \mathbb{P}(x,y \in \mathcal{E} \mid \ip{x}{y} = m),
\end{align*}
and by \cref{lem:flat}
\begin{align*}
    \mathbb{P}\left( x \in \mathcal{E} \cap \mathcal{F} \right) = (1-o(1)) \mathbb{P}\left( x \in \mathcal{E} \right).
\end{align*}
Similar as the proof for the case $t\ge 3$, we deduce from \cref{lem:low_ov-lm,lem:low_ov-km} that
\begin{align*}
& \sum_{m=0}^{\frac{3}{4}k} \mathbb{P}(\ip{x}{y} = m) \, \frac{\mathbb{P}\left( x,y \in \mathcal{E} \cap \mathcal{F} \mid \ip{x}{y} = m \right)}{\mathbb{P}\left( x \in \mathcal{E} \cap \mathcal{F} \right)^2} \\
\le{}& (1+o(1)) \sum_{m=0}^{\frac{3}{4}k} \mathbb{P}(\ip{x}{y} = m) \, \frac{\mathbb{P}(x,y \in \mathcal{E} \mid \ip{x}{y} = m)}{\mathbb{P}\left( x \in \mathcal{E} \right)^2} \\
\le{}& (1+o(1)) \sum_{m=0}^{\frac{3}{4}k} L_m K_m \le 1+o(1).
\end{align*}

\bigskip
\textbf{Case 2: $m \geq \frac{3}{4}k$.} 
Now, we focus on ``high'' overlap values. 
Again by \cref{lem:flat} we have
\begin{align*}
    \mathbb{P}\left( x \in \mathcal{E} \cap \mathcal{F} \right) = (1-o(1)) \mathbb{P}\left( x \in \mathcal{E} \right).
\end{align*}
We focus on upper bounding
\begin{align*}
    \mathbb{P}\left( x,y \in \mathcal{E} \cap \mathcal{F} \mid \ip{x}{y} = m \right).
\end{align*}

Fix an $m \geq 2$ and any $x,y \in \{0,1\}^p$ with $\langle x,y \rangle =m$, and we aim to upper bound $\mathbb{P}'\left( x,y \in \mathcal{E} \cap \mathcal{F} \right)$ where $\mathbb{P}'(\cdot) = \mathbb{P}(\cdot \mid x,y)$ denotes the conditional measure with the given $x,y$. 
Denote by $s$ the indicator vector of the intersection of the support of $x$ and the support of $y$. 
Now, notice that the following three random variables are independent Gaussian random variables:
\begin{align*}
    Z_0 &:=s^\trans W s \sim N(0,m^2);\\
    Z_1 & :=x^\trans Wx - s^\trans W s \sim N(0,k^2-m^2);\\
    Z_2 &:=y^\trans Wy - s^\trans W s \sim N(0,k^2-m^2).
\end{align*}
Finally, let us define 
\begin{align}\label{dfn:e}
    q := \sqrt{2 \log \binom{p}{k} - \log \left(k\log(p/k)\right) - A_p} 
    \quad\text{and}\quad
    E := 2 \sqrt{(k^2-m^2) \log \binom{k}{m}}.
\end{align}
Then for the given pair of vectors $x,y$,
the event $\{ x \in \mathcal{E} \}$ implies $Z_0+Z_1 \ge kq$, and 
the event $\{ y \in \mathcal{E} \cap \mathcal{F} \}$ implies, via a simple application of the triangle inequality, that
\begin{align*}
    \begin{cases}
        \frac{1}{k}(Z_0+Z_2) \ge q;\\[6pt]
        \left| Z_0 - \frac{m^2}{k^2} (Z_0 + Z_2) \right| \le E.
    \end{cases}
    &\Longleftrightarrow \quad
    \begin{cases}
        Z_0+Z_2 \ge kq;\\[6pt]
        \left| Z_2 - \frac{k^2-m^2}{k^2} (Z_0 + Z_2) \right| \le E.
    \end{cases}\\
    &\Longrightarrow \quad
    Z_2 \ge \frac{k^2-m^2}{k}q - E.
\end{align*}
Therefore, we obtain from the independence of $Z_0,Z_1,Z_2$ that
\begin{align*}
    \mathbb{P}'\left( x,y \in \mathcal{E} \cap \mathcal{F} \right)
    \le \mathbb{P}'\left( Z_0+Z_1 \ge kq \right)
        \mathbb{P}'\left( Z_2 \ge \frac{k^2-m^2}{k}q - E \right).
\end{align*}
Let $\overline{\Phi}(x) = \Pr(Z \ge x)$, $x \in \R$ where $Z \sim N(0,1)$ denote the complementary CDF of the standard Gaussian.
Rewriting the right-hand side with $\overline{\Phi}$ and taking expectation over $x,y$ yields
\begin{align*}
    \mathbb{P}\left( x,y \in \mathcal{E} \cap \mathcal{F} \mid \ip{x}{y} = m \right)
    \le \overline{\Phi}(q) \cdot \overline{\Phi}\left( \frac{\sqrt{k^2-m^2}}{k}q - \frac{E}{\sqrt{k^2-m^2}} \right).
\end{align*}
We also have from \cref{eq:mill_1} that
\begin{align*}
    \mathbb{P}\left( x \in \mathcal{E} \cap \mathcal{F} \right) = (1-o(1)) \mathbb{P}\left( x \in \mathcal{E} \right)
    = (1-o(1)) \overline{\Phi}(q).
\end{align*}
Therefore, it follows that
\begin{align}\label{eq:Z012-tail}
    \frac{\mathbb{P}\left( x,y \in \mathcal{E} \cap \mathcal{F} \mid \ip{x}{y} = m \right)}{\mathbb{P}\left( x \in \mathcal{E} \cap \mathcal{F} \right)^2}
    &\le (1+o(1)) \frac{\overline{\Phi}\left( \frac{\sqrt{k^2-m^2}}{k}q - \frac{E}{\sqrt{k^2-m^2}} \right)}{\overline{\Phi}(q)}
\end{align}

We need the following bounds on $q$ and $E$.
\begin{lemma}\label{lem:q-E-bounds}
    Suppose $\omega(1) = A_p = o(k \log(p/k))$.
    Then for sufficiently large $p$ we have
    \begin{align*}
        \sqrt{k \log(p/k)} \le q \le \sqrt{2k\log p}
        \quad\text{and}\quad
        E 
        \le 2(k-m) \sqrt{(k+m)\log k}.
    \end{align*}
    Furthermore, if $k \le p^{1/17}$ then
    \begin{align*}
        \frac{k^2-m^2}{k}q \ge 2 E.
    \end{align*}
    Furthermore, if $k \le p^{1/64}$ then
    \begin{align*}
        \frac{Eq}{k} \le \frac{1}{2}(k-m) \log p.
    \end{align*}
\end{lemma}

\begin{proof}[Proof of \cref{lem:q-E-bounds}]
    For $q$, we use $\binom{p}{k} \le p^k$ and get
    \begin{align*}
        q \le \sqrt{2 \log \binom{p}{k}} \le \sqrt{2 k \log p}.
    \end{align*}
    Meanwhile, since $\binom{p}{k} \ge (p/k)^k$, we have that
    \begin{align*}
        q \ge \sqrt{\log \binom{p}{k}}
        \ge \sqrt{k \log(p/k)}.
    \end{align*}

    For $E$, since we have
    \begin{align*}
        \binom{k}{m} = \binom{k}{k-m} \le k^{k-m},
    \end{align*}
    it follows that
    \begin{align*}
        E \le 2 \sqrt{(k-m)(k+m) \cdot (k-m)\log k}
        = 2(k-m) \sqrt{(k+m)\log k}.
    \end{align*}

    Next, $k \le p^{1/17}$ implies that $16\log k \le \log(p/k)$, and therefore
    \begin{align*}
        \frac{k^2-m^2}{k} q \ge (k-m) \sqrt{\frac{k+m}{k}} \cdot \sqrt{k \log(p/k)}
        \ge 4(k-m) \sqrt{(k+m)\log k} \ge 2E,
    \end{align*}
    as claimed.

    Finally, if $k \le p^{1/64}$ then $\log k \le \frac{1}{64} \log p$, and combining the upper bounds on $q,E$ we just showed yields
    \begin{align*}
        \frac{Eq}{k} \le \frac{\sqrt{2k\log p} \cdot 2\sqrt{2} (k-m) \sqrt{k\log k}}{k}
        = 4 (k-m) \sqrt{\log p \log k}
        \le \frac{1}{2} (k-m)\log p,
    \end{align*}
    as claimed.
\end{proof}

For $m\le k-1$ we deduce from \cref{lem:q-E-bounds} that 
\begin{align}\label{eq:comp-bound}
    \frac{\sqrt{k^2-m^2}}{k}q - \frac{E}{\sqrt{k^2-m^2}}
    \ge \frac{\sqrt{k^2-m^2}}{2k}q \ge \frac{q}{2\sqrt{k}} \to \infty \text{ as $p \to \infty$}.
\end{align}
Thus, we can apply the Gaussian tail bound \cref{eq:mill_1} to \cref{eq:Z012-tail} and obtain
\begin{align*}
    & \frac{\mathbb{P}\left( x,y \in \mathcal{E} \cap \mathcal{F} \mid \ip{x}{y} = m \right)}{\mathbb{P}\left( x \in \mathcal{E} \cap \mathcal{F} \right)^2} \\
    \le{}& (1+o(1)) \frac{q}{\frac{\sqrt{k^2-m^2}}{k}q - \frac{E}{\sqrt{k^2-m^2}}} \exp\left( - \frac{1}{2}\left( \frac{\sqrt{k^2-m^2}}{k}q - \frac{E}{\sqrt{k^2-m^2}} \right)^2 + \frac{1}{2} q^2 \right) \\
    \le{}& (1+o(1)) \cdot 2 \sqrt{k} \exp\left( \frac{m^2}{2k^2} q^2 + \frac{Eq}{k} - \frac{E^2}{2(k^2-m^2)} \right) \tag{by \cref{eq:comp-bound}}\\
    \le{}& 3\sqrt{k} \exp\left( \frac{m^2}{2k^2} q^2 + \frac{Eq}{k} \right).
\end{align*}
Applying the upper bounds $q \le \sqrt{2 \log \binom{p}{k}}$ by definition and $\frac{Eq}{k} \le \frac{1}{2}(k-m) \log p$ from \cref{lem:q-E-bounds}, we thus obtain that
\begin{align*}
    \frac{\mathbb{P}\left( x,y \in \mathcal{E} \cap \mathcal{F} \mid \ip{x}{y} = m \right)}{\mathbb{P}\left( x \in \mathcal{E} \cap \mathcal{F} \right)^2}
    \le{}& 3\sqrt{k} \exp\left( \frac{m^2}{k^2} \log \binom{p}{k} + \frac{1}{2}(k-m) \log p \right) \\
    ={}& 3\sqrt{k} \cdot \binom{p}{k}^{\frac{m^2}{k^2}} \cdot \sqrt{p}^{k-m}.
\end{align*}

Therefore, we have that
\begin{align*}
& \sum_{m=\frac{3}{4}k}^{k-1} \mathbb{P}(\ip{x}{y} = m) \, \frac{\mathbb{P}\left( x,y \in \mathcal{E} \cap \mathcal{F} \mid \ip{x}{y} = m \right)}{\mathbb{P}\left( x \in \mathcal{E} \cap \mathcal{F} \right)^2} \\
\le{}& \sum_{m=\frac{3}{4}k}^{k-1} \frac{\binom{k}{m}\binom{p-k}{k-m}}{\binom{p}{k}} \cdot 3\sqrt{k} \cdot \binom{p}{k}^{\frac{m^2}{k^2}} \cdot \sqrt{p}^{k-m} \\
={}& 3\sqrt{k} \sum_{m=\frac{3}{4}k}^{k-1} \binom{k}{m}\binom{p-k}{k-m} \binom{p}{k}^{-\frac{k^2-m^2}{k^2}} \cdot \sqrt{p}^{k-m}.
\end{align*}
Using standard bounds on binomial coefficients, we deduce that for $m \ge 3k/4$,
\begin{align*}
    \binom{k}{m}\binom{p-k}{k-m} \binom{p}{k}^{-\frac{k^2-m^2}{k^2}}
    &\le k^{k-m} p^{k-m} \left( \frac{p}{k} \right)^{- \frac{k^2-m^2}{k}}
    = \left( kp \left( \frac{p}{k} \right)^{-\frac{k+m}{k}} \right)^{k-m}\\
    &\le \left( kp \left( \frac{p}{k} \right)^{-7/4} \right)^{k-m}
    \le \left( \frac{k^{11/4}}{p^{3/4}} \right)^{k-m}.
\end{align*}
It follows that
\begin{align*}
    \sum_{m=\frac{3}{4}k}^{k-1} \mathbb{P}(\ip{x}{y} = m) \, \frac{\mathbb{P}\left( x,y \in \mathcal{E} \cap \mathcal{F} \mid \ip{x}{y} = m \right)}{\mathbb{P}\left( x \in \mathcal{E} \cap \mathcal{F} \right)^2} 
    &\le 3\sqrt{k} \sum_{m=\frac{3}{4}k}^{k-1} \left( \frac{k^{11/4}}{p^{1/4}} \right)^{k-m} \\
    &\le 3\sqrt{k} \cdot \frac{2k^{11/4}}{p^{1/4}} 
    = \frac{6k^{13/4}}{p^{1/4}} = o(1),
\end{align*}
when $k = o(p^{1/13})$.

\medskip
Combining the two cases we finally have
\begin{align*}
    \frac{\mathbb{E}\left[Z_{\textsf{flat}}^2\right]}{\left( \mathbb{E}[Z_{\textsf{flat}}] \right)^2} - \frac{1}{\mathbb{E}[Z_{\textsf{flat}}]} = 1+o(1).
\end{align*}
The rest of the proof follows the same way as the case $t \ge 3$.

\section{Proof of Negative Results for Sparse Regression}\label{sec:NegSRproof}
In this section, we prove \cref{thm:negresult}, the lower bound on the mixing time for sparse regression. This shows the tightness of our positive results for sparse regression. We do this by a similar but significantly improved approach to prove the Overlap Gap Property result in \cite{10.1214/21-AOS2130} which then implies the mixing time lower bound. As in the previous section, we omit the Overlap Gap Property specifics as much as possible, and focus on the desired mixing time lower bound.

\subsection{Improving Previous Results}
Towards proving our negative result, we first need to generalize some key results from \cite{10.1214/21-AOS2130}. We consider the following two constrained optimization problems:
\begin{align*}
    (\Phi_2)\text{ } \min n^{-\frac{1}{2}}\lVert Y-Xv\rVert_2\\
    \text{s.t. }v\in\{0,1\}^p,\\
    \lVert v\rVert_0 = k
\end{align*}
This problem chooses the exactly $k$-sparse binary $v$ that minimises the squared error. Denote its optimal value $\phi_2$. We will also need this restricted version:
\begin{align*}
    (\Phi_2)(\ell)\text{ } \min n^{-\frac{1}{2}}\lVert Y-Xv\rVert_2\\
    \text{s.t. }v\in\{0,1\}^p,\\
    \lVert v\rVert_0 = k, \text{ } \lVert v-v^*\rVert_0 = 2\ell
\end{align*}
where $\ell\in\{0,1,2,\dots,k\}$. This minimizes the squared error over the $k$-sparse binary vectors $v$ that have overlap $k-\ell$ with $v^*$. Denote the optimal value of this problem by $\phi_2(\ell)$.

We now present strengthened versions of two theorems from \cite{10.1214/21-AOS2130}. Both of these have been strengthened to assume $k=o(p^{1/3-\eta})$ for some $\eta>0$ and no assumption on $n$, instead of the more restrictive $k\log k = O(n)$ in \cite{10.1214/21-AOS2130}, which requires $k=p^{o(1)}$ in the relevant regime $n=o(k \log (p/k))$.

\cref{thm:2.1new}(a) gives us a lower bound on the squared error of the best $v$ at each overlap level with $v^*$. Part (b) tells us that below the algorithmic threshold $k\log (p/k)$, there are many $v$'s with zero overlap with $v^*$ that achieve within a small factor of this lower bound.

\begin{theorem}\label{thm:2.1new}[Improved version of Theorem 2.1 of \cite{10.1214/21-AOS2130}]
    Suppose for some $\eta>0$, $k=o(p^{1/3-\eta})$. Then:
    \begin{enumerate}[(a)]
        \item With high probability as $k$ increases,
        \begin{align*}
            \phi_2(\ell)\geq e^{-\frac{3}{2}}\sqrt{2\ell+\sigma^2}\exp\left(-\frac{\ell \log (p/k)}{n}\right)
        \end{align*}
        for all $0\leq \ell\leq k$.
        \item Suppose further that $\sigma^2 \leq 2k$. Then for every sufficiently large constant $D_0$, if $n\leq \alpha k\log (p/k)/(3\log D_0)$, then w.h.p as $k$ increases, the cardinality of the set
    \begin{align*}
        \scalebox{0.98}{$\left\{ v \in \{0,1\}^p : \lVert v \rVert_0 = k, \lVert v - v^* \rVert_0 = 2k, n^{-\frac{1}{2}} \lVert Y - Xv \rVert_2 \leq R \sqrt{2k + \sigma^2} \exp\left(-\frac{\mu k \log (p/k)}{n}\right) \right\}$}
    \end{align*}
    is at least $R^{n/5}$, where $R=\exp\left(\alpha k\log(p/k)/n\right)$, $\alpha = (1+\eta)/(2+\eta)$, and $\mu =1-o(1)$ as $k\rightarrow\infty$.
    \end{enumerate}
\end{theorem}
The proof of \cref{thm:2.1new}(a) follows identically to the proof in \cite{GZ22_supp}, but with $k \log (p/k)$ in place of $k\log p$ inside the exponent. The proof of \cref{thm:2.1new}(b) uses a second moment method (which is the bulk of the proof), which we defer to \cref{sec:10.1proof}.

\subsection{Proof of the main negative result}
In this section, we will apply \cref{thm:2.1new} to prove \cref{thm:negresult}. The crux of the argument can be described by decomposing the parameter space into a ``low'' overlap with $v^*$ part $A_0$, a ``medium'' overlap $A_1$ and a high overlap part $A_2$. We then show that the probability mass placed on the low overlap set $A_0$ is much higher than that placed on the medium overlap set $A_1$. Since if we start from $A_0$ with a local Markov chain (such as the Metropolis chain on the Johnson graph) we have to go through the medium overlap set to get to the high overlap set $A_2$, this creates a bottleneck effect where the chain remains in the low overlap set for an exponential number of iterations.

The sets of low, medium, and high overlap with $v^*$ we fix throughout some arbitrary constants $0<\zeta_1<\zeta_2<0.9-\alpha$, where $\alpha$ is defined in \cref{thm:2.1new}. Note $\zeta_2<\mu-\alpha-0.05$, since $\mu=1-o(1)$, as $k$ grows. Then define, 
\begin{align}
	&A_0 = \{v: \frac{1}{2k}\lVert v-v^*\rVert_0 >\zeta_{2}\} \text{ (low overlap with $v^*$)} \\
	&A_1 = \{v: \zeta_{1}\leq\frac{1}{2k}\lVert v-v^*\rVert_0 \leq\zeta_{2}\} \text{ (medium overlap)} \\
	&A_2 = \{v: \frac{1}{2k}\lVert v-v^*\rVert_0 <\zeta_{1}\} \text{ (high overlap)}
\end{align}

For $\zeta\in [0,1]$, we define the function $\Gamma$:
$$\Gamma(\zeta)=\left(2\zeta k+\sigma^2\right)^{1/2}\exp\left(-\frac{\zeta k\log (p/k)}{n}\right)$$

Then define $\zeta^*$ as
\begin{align*}
    \zeta^*:=\text{argmin}_{\zeta\in[\zeta_{1},\zeta_{2}]}\Gamma(\zeta)
\end{align*}
\cref{lem:A_1LB} give us a lower bound on the squared error for any $v$ in the medium overlap set $A_1$.
\begin{lemma}\label{lem:A_1LB}
    For all $v$ in $A_1$, it holds that $\lVert Y-Xv \rVert_2^2\geq e^{-3}n\Gamma^2(\zeta^*)$ with high probability.
\begin{proof}
	By \cref{thm:2.1new}(a) if $\lVert v-v^*\rVert_0 = 2\ell$, then
    \begin{align*}
        \lVert Y-Xv \rVert_2^2 \geq e^{-3}n(2\ell+\sigma^2)\exp\left(-\frac{2\ell\log (p/k)}{n}\right)
    \end{align*}
	Note that by definition of $A_1$, if $v\in A_1$, then $\lVert v-v^*\rVert_0$ lies in the interval $[2k\zeta_1,2k\zeta_2]$. It is therefore sufficient to show that for all $\ell\in [2k\zeta_1,2k\zeta_2]$, we have
    \begin{align*}
        e^{-3}n(2\ell+\sigma^2)\exp\left(-\frac{2\ell\log (p/k)}{n}\right)\geq e^{-3}n\Gamma^2(\zeta^*)
    \end{align*}
	or equivalently, that
    \begin{align*}
        e^{-3}n\Gamma^2\left(\frac{\ell}{k}\right)\geq e^{-3}n\Gamma^2(\zeta^*)
    \end{align*}
	But $\zeta^*$ is defined as $$\zeta^*=\text{argmin}_{\zeta\in[\zeta_{1,k,n},\zeta_{2,k,n}]}\Gamma(\zeta)$$
	which completes the proof.
\end{proof}
\end{lemma}
\begin{lemma}\label{lem:zeta*LB} It holds that

\begin{align*}
    \Gamma^2(\zeta^*)\geq \zeta^*\exp\left(2(1-\zeta^*)\frac{k\log (p/k)}{n}\right)\Gamma^2(1)
\end{align*}
\begin{proof}
        Writing out $\Gamma(\zeta^*)$ and $\Gamma(1)$ explicitly, we see that
        \begin{align*}
            \frac{\Gamma(\zeta^*)}{\Gamma(1)} = \frac{(2\zeta^*k+\sigma^2)^\frac{1}{2}}{(2k+\sigma^2)^\frac{1}{2}}\exp\left((1-\zeta^*)\frac{k\log (p/k)}{n}\right)
        \end{align*}
    First we note that since $\sigma^2 \leq k$ and $\zeta^* < 1$, the ratio $\frac{2\zeta^*k+\sigma^2}{2k+\sigma^2}$ is at least $\zeta^*$, which implies that 
    \begin{align*}
        \frac{\Gamma(\zeta^*)}{\Gamma(1)} \geq \sqrt{\zeta^*}\exp\left((1-\zeta^*)\frac{k\log (p/k)}{n}\right)
    \end{align*}
    We then multiply across by $\Gamma(1)$ and square both sides to get the desired result.
\end{proof}
\end{lemma}
We now apply these two lemmas to bound $\pi_\beta(A_1)/\pi_\beta(A_0)$, the ratio between the probability placed on the medium overlap set and the probability placed on the low overlap set.
\begin{lemma}\label{lem:A1A0ratio}
With high probability, it holds that
\begin{align*}
     \frac{\pi_\beta(A_1)}{\pi_\beta(A_0)} \leq \exp\left(k\log \frac{pe}{k}-\frac{\alpha}{5}k\log\frac{p}{k} - \beta nM\right)
    \end{align*}
    where $M=\left[e^{-3}\zeta^*(2k+\sigma^2)\exp\left(-\frac{2\zeta^*k\log (p/k)}{n}\right)-(2\mu k+\sigma^2)\exp\left(-\frac{2(\mu-\alpha)k\log (p/k)}{n}\right)\right]$
\begin{proof}
	Let $v \in A_1$. Then by \cref{lem:A_1LB} and \cref{lem:zeta*LB}, with high probability we have
	\begin{align}
		\lVert Y-Xv \rVert_2^2 &\geq e^{-3}n\Gamma^2(\zeta^*) \text{ by \cref{lem:A_1LB} } \\
		&\geq e^{-3}n\zeta^*\exp\left(2(1-\zeta^*)\frac{k\log (p/k)}{n}\right)\Gamma^2(1) \text{ by \cref{lem:zeta*LB}}
	\end{align}
    Next, note that $\Gamma^2(1)=(2k+\sigma^2)\exp\left(-\frac{2k\log (p/k)}{n}\right)$. This means that 
    \begin{align*}
        &e^{-3}n\zeta^*\exp\left(2(1-\zeta^*)\frac{k\log (p/k)}{n}\right)\Gamma^2(1)\\
        = &e^{-3}n\zeta^*\exp\left(2(1-\zeta^*)\frac{k\log (p/k)}{n}\right)(2k+\sigma^2)\exp\left(-\frac{2k\log (p/k)}{n}\right)\\
        = &e^{-3}n\zeta^*(2k+\sigma^2)\exp\left(-\frac{2\zeta^*k\log (p/k)}{n}\right)\\
    \end{align*}
	This gives us that
	\begin{align*}
        \pi_\beta(v) &= \frac{1}{Z}\exp(-\beta \lVert Y-Xv \rVert^2_2)\\
        &\leq \frac{1}{Z}\exp\left(-\beta e^{-3}n\zeta^*(2k+\sigma^2)\exp\left(-\frac{2\zeta^*k\log (p/k)}{n}\right)\right)
    \end{align*}
	Notice that $A_1$ has at most $\binom{p}{k}$ elements, so we can use a union bound to upper bound $\pi_\beta(A_1)$ by
	\begin{align}
		\pi_\beta(A_1) &\leq \frac{\binom{p}{k}}{Z}\exp\left(-\beta e^{-3}n\zeta^*(2k+\sigma^2)\exp\left(-\frac{2\zeta^*k\log (p/k)}{n}\right)\right) \\
		&\leq \frac{\left(\frac{pe}{k}\right)^k}{Z}\exp\left(-\beta e^{-3}n\zeta^*(2k+\sigma^2)\exp\left(-\frac{2\zeta^*k\log (p/k)}{n}\right)\right)\\
		&= \frac{1}{Z}\exp\left(k\log \frac{pe}{k}-\beta e^{-3}n\zeta^*(2k+\sigma^2)\exp\left(-\frac{2\zeta^*k\log (p/k)}{n}\right)\right)
	\end{align}
    Now we claim that $A_0$ contains at least $R^{n/5}$ elements for which
	\begin{align*}
	    \lVert Y-Xv\rVert ^2 < nR^2 \Gamma^2(\mu)
	\end{align*}
    To see this, apply \cref{thm:2.1new}(b). This allow us to lower bound the probability mass placed on $A_0$:
	\begin{align}
		\pi_\beta(A_0) &\geq \frac{R^{n/5}}{Z}\exp\left(-\beta nR^2\Gamma^2(\mu)\right)\\
		&= \frac{1}{Z}\exp\left(\frac{n}{5}\log R-\beta nR^2\Gamma^2(\mu)\right)
	\end{align}
    Then apply the definition of $\Gamma$ to get that
    \begin{align*}
        \beta n R^2\Gamma^2(\mu) = \beta n R^2 (2\mu k+\sigma^2)\exp\left(-\frac{2\mu k\log (p/k)}{n}\right)
    \end{align*}
    Recall that $R = \exp\left(\frac{\alpha k\log (p/k)}{n}\right)$. This means that $R^2 = \exp\left(\frac{2\alpha k\log (p/k)}{n}\right)$. Apply this fact to see that
    \begin{align*}
        \beta n R^2\Gamma^2(\mu) &= \beta n \exp\left(\frac{2\alpha k\log (p/k)}{n}\right)\left(2\mu k+\sigma^2\right)\exp\left(-\frac{2\mu k\log (p/k)}{n}\right)\\
        &=\beta n \left(2\mu k+\sigma^2\right)\exp\left(-\frac{2(\mu-\alpha)k\log (p/k)}{n}\right)\\
    \end{align*}
    This tells us that 
    \begin{align*}
        \pi_\beta(A_0) \geq \frac{1}{Z}\exp\left(\frac{n}{5}\log R-\beta n \left(2\mu k+\sigma^2\right)\exp\left(-\frac{2(\mu-\alpha)k\log (p/k)}{n}\right)\right)
    \end{align*}
    Apply $\log R = \frac{\alpha k\log (p/k)}{n}$ to get a lower bound of
    \begin{align*}
        \pi_\beta(A_0) \geq \frac{1}{Z}\exp\left(\frac{\alpha}{5}k\log(p/k)-\beta n \left(2\mu k+\sigma^2\right)\exp\left(-\frac{2(\mu-\alpha)k\log (p/k)}{n}\right)\right)
    \end{align*}
	We therefore combine the upper bound on $\pi_\beta(A_1)$ with the lower bound on $\pi_\beta(A_0)$ to see that
    \begin{align*}
     \frac{\pi_\beta(A_1)}{\pi_\beta(A_0)} \leq \exp\left(k\log \frac{pe}{k}-\frac{\alpha}{5}k\log\frac{p}{k} - \beta nM\right)
    \end{align*}
    for $M=\left[e^{-3}\zeta^*(2k+\sigma^2)\exp\left(-\frac{2\zeta^*k\log (p/k)}{n}\right)-(2\mu k+\sigma^2)\exp\left(-\frac{2(\mu-\alpha)k\log (p/k)}{n}\right)\right]$
\end{proof}
\end{lemma}
We now show that if the inverse temperature parameter $\beta$ is sufficiently large, then with high probability the ratio between $\pi_\beta(A_1)$ and $\pi_\beta(A_0)$ is exponentially small.
\begin{theorem}\label{thm:A1A0ratiobound}
    Whenever $\beta$ satisfies
    \begin{align*}
        \beta > \frac{2}{c_2n}\log(p/k)\exp\left(\frac{2k\log(p/k)}{n}\right)
    \end{align*}
    for a small positive constant $c_2$, then with high probability, 
    \begin{align*}
        \frac{\pi_\beta(A_1)}{\pi_\beta(A_0)}\leq \exp(-\Omega(k\log (p/k)))
    \end{align*}
    \begin{proof}
    By \cref{lem:A1A0ratio}, it suffices to show that the following expression is at least $\Omega(k\log (p/k))$:
        \begin{align*}
            \frac{\alpha}{5}k\log\frac{p}{k} + \beta nM-k\log \frac{pe}{k}
        \end{align*}
        for $M=\left[e^{-3}\zeta^*(2k+\sigma^2)\exp\left(-\frac{2\zeta^*k\log (p/k)}{n}\right)-(2\mu k+\sigma^2)\exp\left(-\frac{2(\mu-\alpha)k\log (p/k)}{n}\right)\right]$.
        
        First, define $h$ as
        \begin{align*}
            h = e^{-3}\zeta^*\left(2k+\sigma^2\right)-\frac{2\mu k+\sigma^2}{D_0^{\frac{6}{\alpha}(\mu-\alpha-\zeta^*)}}
        \end{align*}
        Then we claim that $h$ is lower bounded by $c_2 k$ for a small positive constant $c_2$ when $D_0$ is large enough. To see this, note that $e^{-3}\zeta^*\left(2k+\sigma^2\right)$ is lower bounded by $2e^{-3}\zeta^*k$, and we can upper bound the second summand of $h$ since $\sigma^2 \leq 2k$ as follows,
        \begin{align*}
            \frac{2\mu k+\sigma^2}{D_0^{\frac{6}{\alpha}(\mu-\alpha-\zeta^*)}}\leq \frac{4}{D_0^{\frac{6}{\alpha}(\mu-\alpha-\zeta^*)}}k
        \end{align*}
        Recall $\mu-\alpha-\zeta^*>0.05$ by construction, so for large enough $D_0$, $h$ is at least $c_2k$ for some small $c_2>0$. 

        Next, we claim that for large enough $D_0$, it holds that
        \begin{align*}
            M\geq h\exp\left(-\frac{2\zeta^*k\log (p/k)}{n}\right)
        \end{align*}
        To prove this, observe that rearranging the definition of $h$, we get
        \begin{align*}
            \frac{2\mu k+\sigma^2}{D_0^{\frac{6}{\alpha}(\mu-\alpha-\zeta^*)}} = e^{-3}\zeta^*\left(2k+\sigma^2\right)-h
        \end{align*}
        and therefore,
        \begin{align*}
            \frac{2\mu k+\sigma^2}{e^{-3}\zeta^*\left(2k+\sigma^2\right)-h} = D_0^{\frac{6}{\alpha}(\mu-\alpha-\zeta^*)}
        \end{align*}
        Taking the logarithm of both sides, we see that
        \begin{align*}
            \log\left[\frac{2\mu k+\sigma^2}{e^{-3}\zeta^*\left(2k+\sigma^2\right)-h}\right] = \frac{6}{\alpha}(\mu-\alpha-\zeta^*)\log D_0
        \end{align*}
        and then by dividing across by $2(\mu-\alpha-\zeta^*)$ and inverting both sides we get
        \begin{align*}
            \frac{2(\mu-\alpha-\zeta^*)}{\log\left[\frac{2\mu k+\sigma^2}{e^{-3}\zeta^*\left(2k+\sigma^2\right)-h}\right] } = \frac{\alpha}{3\log D_0}
        \end{align*}
        Since by assumption we have $n\leq \alpha k\log(p/k)/(3\log D_0)$, it therefore holds that
        \begin{align*}
            n\leq \frac{2(\mu-\alpha-\zeta^*) k\log(p/k)}{\log\left[\frac{2\mu k+\sigma^2}{e^{-3}\zeta^*\left(2k+\sigma^2\right)-h}\right]}
        \end{align*}
        Rearranging, we get
        \begin{align*}
        \log\left[\frac{e^{-3}\zeta^*\left(2k+\sigma^2\right)-h}{2\mu k+\sigma^2}\right]&\geq -\frac{2(\mu-\alpha-\zeta^*) k\log(p/k)}{n}\\
        \end{align*}
        This gives us
        \begin{align*}
            \log\left[e^{-3}\zeta^*\left(2k+\sigma^2\right)-h\right]-\frac{2\zeta^*k\log(p/k)}{n}&\geq \log\left[2\mu k+\sigma^2\right]-\frac{2(\mu-\alpha)k\log(p/k)}{n}
        \end{align*}
        Take the exponential of both sides to get
        \begin{align*}
            \left[e^{-3}\zeta^*\left(2k+\sigma^2\right)-h\right]\exp\left(-\frac{2\zeta^*k\log(p/k)}{n}\right)&\geq \left[2\mu k+\sigma^2\right]\exp\left(-\frac{2(\mu-\alpha)k\log(p/k)}{n}\right)
        \end{align*}
        and finally, rearrange once again to get
        \begin{align*}
            \left[e^{-3}\zeta^*\left(2k+\sigma^2\right)\right]\exp\left(-\frac{2\zeta^*k\log(p/k)}{n}\right) - \left[2\mu k+\sigma^2\right]\exp\left(-\frac{2(\mu-\alpha)k\log(p/k)}{n}\right)\\
            \geq h\exp\left(-\frac{2\zeta^*k\log(p/k)}{n}\right)
        \end{align*}
        which proves the lower bound on $M$. Combining this with the fact that $h\geq c_2 k$, it therefore suffices for \cref{thm:A1A0ratiobound} to show that 
        \begin{align}\label{thm6boundB}
            \beta c_2nk\exp\left(-\frac{2\zeta^*k\log (p/k)}{n}\right)-k\log\frac{pe}{k} = \Omega(k\log (p/k))
        \end{align}
        We assume that
        \begin{align*}
            \beta > \frac{2}{c_2n}\log(p/k)\exp\left(\frac{2k\log(p/k)}{n}\right)
        \end{align*}
        which means that (since $\zeta^*<1$)
        \begin{align*}
            \beta c_2nk\exp\left(-\frac{2\zeta^*k\log (p/k)}{n}\right)-k\log\frac{pe}{k} >2k\log(p/k)-k\log\frac{pe}{k} = \Omega(k\log(p/k)
        \end{align*}
    \end{proof}
\end{theorem}
Using \cref{thm:A1A0ratiobound}, \cref{thm:negresult} follows from \cite[Claim 2.1]{MWW09} (see also \cite[Proposition 2.2]{WellsCOLT20}).

\section{Computational hardness of sparse tensor PCA}\label{sec:reduction}

In this short section, we describe how a reduction argument provides evidence of the computational hardness of the sparse tensor PCA model as defined in \eqref{eq:stpca} when $\lambda=o(\lambda_{ALG}(t,k))$ for some 
\[\lambda_{ALG}(t,k)=\tilde{\Theta}(\min\{k^{t/2},p^{(t-1)/2}/k^{t/2-1}\})\]
The reduction is based on a similar reduction described in \cite[Theorem 10]{brennan2020reducibility} for where $v^*\in\{-1,1\}^n$ without sparsity, which we appropriately extend to the case of a $k$-sparse binary $v^*$.
Our modification of the argument given here is essentially borrowed from a more general argument described in the proof of \cite[Theorem 16]{luo2022tensor}. Yet, as we are only interested in a specific implication of the proof in \cite{luo2022tensor}, not formally covered by any of the theorems in \cite{luo2022tensor} or \cite{brennan2020reducibility}, we describe it here for reasons of completeness and the reader's convenience.

\subsection{The hypergraph planted dense subgraph problem}

The key idea is to construct an average-case reduction from the hard regime of another recovery task, known as hypergraph planted dense subgraph problem (HPDS).

In HPDS, given $n,k=k_p \in \mathbb{N}, k\leq p$, some $\rho=\rho_p \in (0,1)$ and some fixed $t \geq 2$, one observes an instance of a random $t$-hypergraph on $p$ vertices, which we denote by $G$ and is constructed as follows. Firstly, one chooses $k$ vertices out of $n$ uniformly at random, resulting in a $k$-subset $\mathcal{PC}$. Then, given $\mathcal{PC}$, $G$ is a random $t$-hypergraph which contains all $t$-hyperedges between vertices of $\mathcal{PC}$ independently with probability $1/2+\rho$, and all other $t$-hyperedges independently and with probability exactly $1/2$. The goal of the statistician is to recover $\mathcal{PC}$ from observing $G$ with high probability as $p \rightarrow +\infty$.

Of course the larger the ``gap parameter'' $\rho \in (0,1)$ is, the easier the recovery of $\mathcal{PC}$ from $G$ is. In terms of computational limits for HPDS, a direct application of \cite[Theorem 20]{luo2022tensor} implies that for some
\[\rho_{\mathrm{ALG}}=\tilde{\Theta}(\min\{1,p^{(t-1)/2}/k^{t-1}\})\]
if $\rho=o(\rho_{\mathrm{ALG}})$ then all $(\log p)^2$-degree polynomials fail to recover $\mathcal{PC}$ from $G$. This result provides strong evidence of computational hardness in that regime of HPDS, based on the low-degree framework for estimation problems \cite{schramm2022computational,luo2022tensor}.

\subsection{Connecting back to sparse tensor PCA}
Importantly to us, as long as $\rho=p^{-O(t)}$ holds, a simple average-case polynomial-time reduction can map any instance of HPDS with parameter $\rho$ to an instance of sparse tensor PCA with parameter $\lambda=  \tilde{\Theta}(k^{t/2} \log (1+2\rho))$ and $v=\mathbf{1}_{\mathcal{PC}}$ the indicator vector of $\mathcal{PC}$. In particular, given this result, the computationally hard regime of HPDS $\rho=o(\rho_{\mathrm{ALG}})$ maps directly to the regime of sparse tensor PCA where $\lambda=o(\lambda_{ALG}(t,k))$ for some 
\[\lambda_{ALG}(t,k)=\tilde{\Theta}(k^{t/2}\log(1+2 \rho_{\mathrm{ALG}}))=\tilde{\Theta}(k^{t/2} \rho_{\mathrm{ALG}})= \tilde{\Theta}(\min\{k^{t/2},p^{(t-1)/2}/k^{t/2-1}\})\]
giving us the desired result.

The simple reduction works by employing the rejection kernel technique \cite{brennan2018reducibility}. Indeed by using the technique as described in \cite[Lemma 15]{luo2022tensor} for probabilities $1/2+\rho, 1/2$ and some $\xi=\tilde{\Theta}(\log (1+2\rho)/\sqrt{t})$, combined with the tensorization of the total variation distance \cite[Lemma 13]{luo2022tensor}, we directly get that the total variation between an instance $G$ from HPDS with parameter $\rho$ and an instance $Y$ from sparse tensor PCA for $\lambda= k^{t/2} \xi$ is $o(1)$, which completes the reduction.

\section{Acknowledgments}
I.Z. is thankful to Alex Wein for multiple discussions regarding the ``geometric rule'' during a past project.

\newpage

\bibliographystyle{emss}
\bibliography{references,geom}

\newpage
\appendix
\section{Deferred Proofs}\label{sec:DeferredProofs}

\subsection{Proof of \texorpdfstring{\cref{lem:SR_ver_assump2}}{lem:SR ver assump2}}\label{sec:proofSR_ver_assump2}
Here we prove \cref{lem:SR_ver_assump2}, that the Binary - Local Search Algorithm decreases the squared error by at least $\frac{1}{4}n$ at every step when $n\geq Ck\log (p/k)$ and $\sigma^2 \leq k/\log p$. The proof is based on an appropriately strengthened version of the proof of Theorem 2.12 in \cite{10.1214/21-AOS2130}, when applied to the case of binary coefficients. We will need a couple of results from \cite{GZ22_supp}.

\begin{definition}
     Let $n,k,p\in\mathbb{N}$ with $k\leq p$. We say that a matrix $Z\in\mathbb{R}^{n\times p}$ satisfies the $k$-Restricted  Isometry Property ($k$-RIP) with restricted isometric constant $\delta\in(0,1)$ if for every vector $\beta\in\mathbb{R}^p$ which is at most $k$-sparse, it holds that 
    \begin{align*}
        (1-\delta)\lVert\beta\rVert_2^2 n\leq \lVert Z\beta\rVert_2^2\leq (1+\delta)\lVert\beta\rVert_2^2 n
    \end{align*}
\end{definition}
We will need this standard result.
\begin{theorem}\label{thm:k-RIP}[Theorem 5.2 of \cite{Baraniuk2008}]
    Let $n,k,p\in\mathbb{N}$ with $k\leq p$. Suppose $Z\in\mathbb{R}^{n\times p}$ has i.i.d. standard Gaussian entries. Then for every $\delta >0$, there exists a constant $C=C_\delta>0$ such that if $n\geq Ck\log \frac{p}{k}$, then $Z$ satisfies the $k$-RIP with restricted isometric constant $\delta$ w.h.p. as $k\rightarrow\infty$.
\end{theorem}

We will also need the concept of a Deviating Local Minimum.
\begin{definition}
    Let $d\in\mathbb{N}$. We call a set $\emptyset\neq S\subseteq [d]$ a \textbf{super-support} of a vector $x\in\mathbb{R}^d$ if $\text{supp}(x)\subseteq S$.
\end{definition}
\begin{definition}
    Let $n,p\in\mathbb{N}$, $\alpha\in(0,1)$, $X\in\mathbb{R}^{n\times p}$, and $\emptyset\neq S_1,S_2,S_3\subseteq[p]$. A triplet of vectors $(a,b,c)$ with $a,b,c\in\mathbb{R}^p$ is called an $\alpha$-\textbf{deviating local minimum} ($\alpha$-\textbf{D.L.M.}) with respect to $S_1,S_2,S_3$ and to the matrix $X$ if the following are satisfied:
    \begin{enumerate}
        \item The sets $S_1,S_2,S_3$ are pairwise disjoint and the vectors $a,b,c$ have super-supports $S_1,S_2,S_3$ respectively.
        \item $|S_1|=|S_2|$ and $|S_1|+|S_2|+|S_3|\leq 3k$.
        \item For all $i\in S_1$ and $j\in S_2$,
        \begin{align*}
            \lVert \left(Xa-a_iX_i\right)+\left(Xb-b_jX_j\right)+Xc\rVert_2^2\geq \lVert Xa+Xb+Xc\rVert_2^2-\alpha \left(\frac{\lVert a\rVert_2^2}{|S_1|}+\frac{\lVert b\rVert_2^2}{|S_2|}\right)n
        \end{align*}
    \end{enumerate}
\end{definition}

Finally, we will need two results from \cite{GZ22_supp} concerning D.L.M.s
\begin{proposition}\label{prop:F7}[Proposition F.7 of \cite{GZ22_supp}]
    Let $n,p,k\in\mathbb{N}$ with $k\leq \frac{1}{3}p$. Suppose that $X\in\mathbb{R}^{n\times p}$ satisfies the $3k$-RIP with restricted isometry constant $\delta_{3k}<\frac{1}{12}$. Then there is no $\frac{1}{4}$-D.L.M. triplet $(a,b,c)$ with respect to the matrix $X$ with $\lVert a\rVert_2^2 + \lVert b\rVert_2^2 \geq \frac{1}{4}\lVert c\rVert_2^2$.
\end{proposition}
\begin{lemma}\label{lem:Hanson_app}[Proposition F.11 of \cite{GZ22_supp}]
There exists a universal constant $c_0>0$ such that for any fixed triplet $(a,b,c)$  with $a \neq0$,
\[\mathbb{P}\left( (a,b,c) \text{ is a } \frac{1}{4}\text{-D.L.M. triplet} \right) \leq 2\exp\left(-c_0 n\min\left\{1,\frac{\|a\|_2^2+\|b\|_2^2}{\|c\|_2^2}\right\}\right),\]
where for the case $c=0$ we abuse the notation by defining $\frac{1}{0}:=+\infty$.
\end{lemma}
The proof of \cref{lem:Hanson_app} is identical to the proof of Lemma F.11 of \cite{GZ22_supp}, but with $\frac{1}{4}$ in place of $\frac{1}{2}$.

The next proposition is the main technical contribution of this section.
\begin{proposition}\label{prop:F8}[(Improved) Proposition F.8 of \cite{GZ22_supp}]
    Let $n,p,k\in\mathbb{N}$ with $k\leq \frac{1}{3}p$. Suppose that $X\in\mathbb{R}^{n\times p}$ has i.i.d. $\mathcal{N}(0,1)$ entries. There exists a constant $C_1>0$ such that if $n\geq C_1 k\log (p/k)$ then w.h.p., there is no $\frac{1}{4}$-D.L.M. triplet $(a,b,c)$ with respect to some sets $\emptyset\neq S_1,S_2,S_3\subset [p]$ and the matrix $X$ such that the following conditions are satisfied.
    \begin{enumerate}
        \item $a,b,c, \in \{-1,0,1\}^p$
        \item $S_1\cup S_3 = [k]\cup \{p\}$, $p\in S_3$, and $S_1=\text{supp}(a)$.
        \item $\lVert a\rVert_2^2+\lVert b\rVert_2^2+\lVert c\rVert_2^2\leq 1.25k/\log p$
    \end{enumerate}
\end{proposition}

\begin{proof}[Proof of \cref{prop:F8}]

We first choose $C_1>0$ large enough based on \cref{thm:k-RIP} so that $n \geq C_1 k \log(p/k)$ implies that $X$ satisfies the $3k$-RIP with $\delta_{3k}<\frac{1}{16}$ w.h.p. In particular, all the probability calculations below will be conditioned on this high-probability event. We define the following sets parametrized by $r>0$ and $\alpha \in (0,1)$
\[B_{r}:=\{ (a,b,c) \big{|} a,b,c \in \{-1,0,1\}^p \setminus \{(0,0,0\},  \|a\|_0+\|b\|_0+\|c\|_0 \leq \min\left\{2k+1,r^2\right\} \}\]
and
\begin{align*}
    &D_{\alpha,r}:= \{(a,b,c) \in B_{r} \big{|} (a,b,c) \text{ is } \alpha\text{-D.L.M. with corresponding super-supports satisfying}\\
    & \text{assumption (2) of \cref{prop:F8}. } \}
\end{align*}
We call a triplet of sets $\emptyset \neq S_1,S_2,S_3 \subseteq [p]$ \textit{good} if \begin{itemize}
    \item $S_1,S_2,S_3$ are pair-wise disjoint
    \item $|S_1|=|S_2|$, $p \in S_3$ and $S_1 \cup S_3=[k] \cup \{p\}$
\end{itemize}
For $\alpha \in \mathbb{R}$ and $S \subseteq \mathbb{R}$ we define the set
\[S-\alpha:=\{s-\alpha|s \in S\}.\]
For $i=1,2,3$ we set $P_i:=\{(i-1)p+1,(i-1)p+2,\ldots,ip\}$. Notice that the sets $P_1,P_2,P_3$ partition $[3p]$. We define the following family of subsets of $[3p]$,
\[\mathcal{T}:=\{T \subset [3p] | \text{ the triplet } T \cap P_1,T\cap P_2-p,T \cap P_3-2p \text{ is good}\}.\]
It is easy to see that $\mathcal{T} \subset \{T \subset [3p] | |T| \leq 2k+1\}$. Furthermore for any $T \in \mathcal{T}$ we define
\[B_{r}(T):=\{ (a,b,c) \in B_{r}\big{|}  \mathrm{Support}\left( (a,b,c) \right) \subseteq T, T \cap P_1=\mathrm{Support}\left(a\right)\}\]
and
\begin{align*}
    D_{\alpha,r}(T):=\{ (a,b,c) \in B_{r}(T) \big{|} (a,b,c) \text{ is } \alpha\text{-D.L.M. with respect to } \\
    T \cap P_1,T\cap P_2-p,T \cap P_3-2p   \}.
\end{align*}
We claim that
\begin{align}\label{eq:Tunion}
    D_{\frac{1}{4},r}=\bigcup_{T \in \mathcal{T}} D_{\frac{1}{4},r} \left(T\right).
\end{align}
For the one direction, if $A=(a,b,c) \in D_{\frac{1}{4},r} \left(T\right)$ for some $T \in \mathcal{T}$ then $(a,b,c)$ is $\alpha$-DLM with corresponding super-supports $T \cap P_1,T\cap P_2-p,T \cap P_3-2p$ which can be easily checked to satisfy assumption (2) of \cref{prop:F8} based on our assumptions. For the other direction, if $A\in D_{\frac{1}{4},r}$ is an $\alpha$-DLM with respect to $S_1,S_2,S_3$ satisfying assumption (2) of \cref{prop:F8}, it can be easily verified that for the set $T=S_1 \cup \left(S_2+p \right) \cup \left(S_3+2p\right)$ it holds $T \in \mathcal{T}$ and furthermore $A \in D_{\frac{1}{4},r} \left(T\right)$.

Now to prove the proposition it suffices to prove that there exists a constant $C_1>0$ such that if $n \geq C_1k \log(p/k)$ and $r^2\leq 1.25 k/\log p$ then
\[\lim_{k \rightarrow +\infty} \mathbb{P}\left(D_{\frac{1}{4},r} \neq \emptyset \right) =0.\] Using \eqref{eq:Tunion} with $\alpha=\frac{1}{4}$ and a union bound, it suffices to be shown that for some $C_1>0$ if $n \geq C_1 k \log (p/k)$ and $r^2 \leq 1.25 k/\log p$ then
\[\lim_{k \rightarrow +\infty}\sum_{T \in \mathcal{T}}   \mathbb{P}\left(D_{\frac{1}{4},r} \left(T\right) \neq \emptyset \right)=0.\]
But by Markov's inequality for all such $T \in \mathcal{T}$,
\begin{align}\label{eq:TunionMarkov}
    \mathbb{P}\left(| D_{\frac{1}{4},r}| \geq 1\right) \leq \E{|D_{\frac{1}{4},r}|}.
\end{align}
Furthermore, for all $T \in \mathcal{T}$, $1 \leq |T \cap P_2| \leq k$. By \eqref{eq:TunionMarkov} and summing over the possible values of $|T\cap P_2|$ for $T \in \mathcal{T}$, it therefore suffices to show that for some $C_1>0$, if $n \geq C_1 k \log (p/k)$ and $r^2 \leq 1.25 k/\log p$ then,
\begin{align}\label{eq:Target}
    \lim_{k \rightarrow + \infty} \sum_{m=1}^{k} \sum_{T \in \mathcal{T}, |T \cap P_2|=m} \mathbb{E}\left(| D_{\frac{1}{4},r}(T)| \right) = 0
\end{align}
Fix $m \in [k]$ and a set $T \in \mathcal{T}$ with $|T\cap P_2|=m$. Then for any $A=(a,b,c) \in  D_{\frac{1}{4},r}(T)$, since $D_{\frac{1}{4},r}(T) \subseteq B_{r}(T)$, we have $\|a\|_2^2+\|b\|_2^2+\|c\|_2^2 \leq r^2$. Based on the definition of $D_{\frac{1}{4},r}(T)$, we also have $|\mathrm{Support}(a)|=|S_1|=|S_2|=|T \cap P_2|=m$. Hence, $\|a\|_2^2 = m$ (as $a\in\{-1,0,+1\}^p$) and  $\|c\|_2^2 \leq \|a\|_2^2+\|b\|_2^2+\|c\|_2^2 \leq r^2$. By \cref{lem:Hanson_app} we know that for any triplet $A=(a,b,c)$,
\[\mathbb{P}\left( A \in D_{\frac{1}{4},r}(T) \right) \leq  \exp \left(-c_0n \min \left\{1,\frac{\|a\|_2^2+\|b\|_2^2}{\|c\|_2^2} \right\} \right)\] 
for some universal constant $c_0>0$. Hence using the above inequalities we can conclude that for any such $A=(a,b,c) $ it holds that
\begin{align}\label{eq:app1}
    \mathbb{P}\left( A \in D_{\frac{1}{4},r}(T) \right) \leq 2\exp \left(-c_0n \min \left\{1,\frac{m}{r^2} \right\} \right)
\end{align}
Linearity of expectation, the above bound, and an easy bound on the cardinality of the vectors in $\{-1,0,1\}^p$ with non-zero values only in $T$ and sparsity level at most $r^2$ imply that
\begin{align}\label{eq:app2}
& \E{| D_{\frac{1}{4},r}(T)|} \leq 3^{r^2}\binom{2k+1}{r^2} \exp \left(-c_0n \min \left\{1,\frac{m}{r^2} \right\} \right).
\end{align}
We now count the number of possible $T \in \mathcal{T}$ with $ |T\cap P_2|=m$. Recall that any $T \subseteq [3p]$ satisfies $T \in \mathcal{T}$ if and only if the triplet of sets $T \cap P_1,T\cap P_2-p,T \cap P_3-2p$ is a \textit{good} triplet. That is if and only if
\begin{itemize}
    \item[(1)] $T \cap P_1,T\cap P_2-p,T \cap P_3-2p$ are pairwise disjoint sets and $|T \cap P_1|=|T \cap P_2-p|=|T \cap P_2|=m$
    \item[(2)] $p \in T \cap P_3-2p $ 
    \item[(3)] $\left( T\cap P_1 \right) \cup \left( T \cap P_3-2p \right)=[k] \cup \{p\}$
\end{itemize}
Since a set $T \subseteq [3p]$ is completely characterized by the intersections with $P_1,P_2,P_3$, it suffices to count the number of triplets of sets $T \cap P_i$, $i=1,2,3$ satisfying the three above conditions. Now conditions (1),(3) imply that $T\cap P_3$ is completely characterized by $T \cap P_1$. Furthermore by checking conditions $(1),(2),(3)$ we know that $T \cap P_1$ is an arbitrary subset of $[k]$ of size $m$. Hence we have $\binom{k}{m}$ choices for both the sets $T \cap P_1$ and $T \cap P_3$. Finally for the set $T \cap P_2$ we only have that it needs to satisfy $|T \cap P_2|=m$. Hence for $T \cap P_2$ we have $\binom{p}{m}$ choices, giving in total that the number of sets $T \in \mathcal{T}$ with $ |T \cap P_2|=m$  equals to $\binom{k}{m} \binom{p}{m}$. Hence,
\[\sum_{T \in \mathcal{T}, |T\cap P_2|=m} \mathbb{E}\left(| D_{\frac{1}{4},r}(T)| \right) \leq \binom{k}{m}\binom{p}{m} 3^{r^2}\binom{2k+1}{r^2} \exp \left(-c_0n \min \left\{1,\frac{m}{r^2} \right\} \right).\]
Summing over all $m=1,2,\ldots,k$ and using the bounds $\binom{k}{m} \leq \binom{p}{m} \leq (pe/m)^m, 3^{r^2}\binom{2k+1}{r^2} \leq (10k)^{r^2}$ we conclude that
\begin{align*}
&\sum_{m=1}^k\sum_{T \in \mathcal{T}, |T\cap P|=m} \mathbb{E}\left(| D_{\frac{1}{4},r}(T)| \right)
\end{align*}
is at most  
\[ k\max_{m=1,\ldots,k} \left[ (pe/m)^{2m} (10k)^{r^2}\exp \left(-c_0n \min \left\{1,\frac{m}{r^2} \right\} \right)\right].\]
Therefore it suffices to show that for some $C_1>0$ if $n \geq C_1 k \log (p/k)$, $r^2 \leq 1.25k/\log p$  then
\begin{align}
    \lim_{k \rightarrow  \infty} k\max_{m=1,\ldots,k} \left[ (pe/m)^{2m} (10k)^{r^2}\exp \left(-c_0n \min \left\{1,\frac{m}{r^2} \right\} \right)\right] = 0.
\end{align}
Since this is a decreasing quantity in $n$ we plug in $n=C_1 k \log (p/k)$. After taking logarithms it suffices to be proven that for $C_1$ large enough but constant, if $r^2 \leq 1.25 k/\log p$  then
\[\max_{m=1,\ldots,k} \left[ 2m \log (pe/m)+ r^2\log \left(10k\right)-c_0C_1k \log (p/k) \min \left\{1,\frac{m}{r^2} \right\} \right] +\log k \rightarrow -\infty.\]

We consider the two cases: when $m \leq  r^2$ and when $m \geq r^2$. Suppose $m \geq  r^2$, that is $\min \left\{1,\frac{m}{r^2} \right\}=1$. Since $r^2 \leq 1.25 k/\log (p) $ it therefore holds for large enough $k$ that
\begin{align*}
    &\max_{k \geq m \geq  r^2} \left[2m \log (pe/m)+ r^2\log \left(10k\right)-c_0C_1k \log (p/k) \min \left\{1,\frac{m}{r^2} \right\} \right]+ \log k\\
    & \leq 2k \log (pe/k)+ 2k-c_0C_1k \log (p/k) +\log k \\
    & \leq -(c_0C_1-5)k \log (p/k),
\end{align*} which if $C_1>6/c_0$ clearly diverges to $-\infty$ as $k \rightarrow +\infty$.

Now suppose $m \leq  r^2$, that is when $\min \left\{1,\frac{m}{r^2} \right\}=\frac{m}{r^2}$. We have 
\begin{align*}
&\max_{1 \leq m \leq  r^2} \left[2m \log (pe/m)+ r^2\log \left(10k\right)-c_0C_1k \log (p/k) \min \left\{1,\frac{m}{r^2}\right\} \right]+ \log k\\
& \leq \max_{1 \leq m \leq  r^2} \left[2m \log (pe)+ r^2\log \left(10k\right)-c_0C_1k \log p \left(\frac{m}{r^2}\right)\right] + \log k. 
\end{align*} We write 
\begin{align*}
&2m \log (pe)+ r^2\log \left(10k\right)-c_0C_1k \log (p/k) \frac{m}{r^2}\\
&=2m \log (pe) -\frac{c_0C_1}{2}k \log (p/k) \cdot \frac{m}{r^2}+r^2\log \left(10k\right)- \frac{c_0C_1}{2}k \log (p/k) \cdot  \frac{m}{r^2}.
\end{align*} 
Since for large enough $k$, $r^2 \leq 2k/(\log (pe))$, we therefore see that
\begin{align}\label{eq:step1}
2m \log (pe) -\frac{c_0C_1}{2}k \log (p/k) \cdot \frac{m}{r^2} \leq \left(2-\frac{c_0C_1}{4}\log(p/k)\right)m \log(pe) \leq - 2\log p 
\end{align} for $C_1 \geq 16/(c_0\log(p/k))$ (note that this is at most a constant because $k\leq p/3$). Now we will bound the second summand.
Again assuming $C_1>6/c_0$ and using that $m \geq 1$ we have 
 \begin{align}\label{eq:step2}
    r^2\log \left(10k\right)- \frac{c_0C_1}{2}k \log (p/k) \cdot  \frac{m}{r^2} &\leq r^2( 2\log p-c_0C_1 k/r^2)/2\\
    &\leq -(c_0C_1-2)\log p/2.
\end{align}  Now combining \eqref{eq:step1} and \eqref{eq:step2} we conclude that for large enough $C_1>2/c_0$ that 
\begin{align*}
    &\max_{1 \leq m \leq  r^2} \left[2m \log (pe)+ r^2\log \left(10k\right)-c_0C_1k \log p \frac{m}{r^2}\right] + \log k\\
    & \leq -2 \log p-(c_0C_1-2)\log p/2+\log k \\
    & \leq -\log p \rightarrow -\infty, \text{ as } n,p,k \rightarrow +\infty
\end{align*} which completes the proof of \cref{prop:F8}.
\end{proof}
Now we have everything we need to prove \cref{lem:SR_ver_assump2}.
\begin{proof}[Proof of \cref{lem:SR_ver_assump2}]
    Define $X'\in \mathbb{R}^{n\times(p+1)}$ to be $X'=[X,\frac{1}{\sigma}W]$, so that $X'$ has i.i.d. standard normal entries, and 
        \begin{align*}
            Y =Xv^*+W=X' \begin{bmatrix}
           v^* \\
           \sigma
         \end{bmatrix}
        \end{align*}
        Choose $C>0$ large enough so that $X'$ satisfies $3k$-RIP (Restricted Isometry Property) with $\delta_{3k}<\frac{1}{12}$ (we know we can do this by \cref{thm:k-RIP}). Let $T$ denote the support of $v$, and without loss of generality, assume that the support of $v^*$ is $[k]$. Denote $m = |T\setminus [k]|=|[k]\setminus T|$.
        
        Suppose $v\neq v^*$, and $v'$ is obtained from $v$ in one iteration of B-LSA, but
        \begin{align}
            \lVert Y-Xv'\rVert^2_2 \geq \lVert Y-Xv\rVert^2_2-\frac{1}{4}n
        \end{align}
        We therefore know that for all $i\in[p]$, for all $j\in T$, it holds that
        \begin{align*}
            \|Y-Xv+X_j-X_i\|_2^2\geq\|Y-Xv\|_2^2-\frac{1}{4}n
        \end{align*}
        or equivalently,
        \begin{align}
            \lVert Xv^* + W -Xv +X_j - X_i \rVert_2^2 \geq \lVert Xv^* + W -Xv \rVert^2_2 - \frac{1}{4}n\label{eq:L1ineq1}
        \end{align}
        Now let 
        \begin{align*}
            a &= \begin{pmatrix}
                v^*_{[k]\setminus T} \\
                0
            \end{pmatrix} \in \mathbb{R}^{p+1} \\
            b & = \begin{pmatrix}
                -v_{T\setminus [k]} \\
                0
            \end{pmatrix} \in \mathbb{R}^{p+1} \\
            c & = \begin{pmatrix}
                \mathbf{0}_{p\times 1} \\
                \sigma
            \end{pmatrix} \in \mathbb{R}^{p+1} \\
        \end{align*}
        where here for any $S\subseteq[p]$ we write $x_S$ for the vector $x$ with all coordinates outside of $S$ set to $0$.
        Note that $\text{supp}(a)=[k]\setminus T$, $\text{supp}(b)=T\setminus [k]$, and $|[k]\setminus T| = |T\setminus [k]|$, and $\text{supp}(c) =\{p+1\}$. Note also that $|[k]\setminus T|+|T\setminus [k]|+|\{p+1\}|\leq 3k$. We also have that $([k]\setminus T)\cap(T\setminus[k])=\varnothing$, $([k]\setminus T)\cap \{p+1\}=\varnothing$, and $(T\setminus [k])\cap \{p+1\}=\varnothing$. Finally note that $\lVert a\rVert_2^2 = |\text{supp}(a)|$ and $\lVert b\rVert_2^2 = |\text{supp}(b)|$, so
        \begin{align*}
            \frac{\lVert a\rVert_2^2}{|\text{supp}(a)|}+\frac{\lVert b\rVert_2^2}{|\text{supp}(b)|} = 2
        \end{align*}
        Observe that we can rewrite the inequality \eqref{eq:L1ineq1} above in terms of $X'$ to get that for all $i\in [p]$ and for all $j\in T$, we have
        \begin{align}
            \lVert X'a+X'b+X'c-a_i X'_i-b_jX'_j\rVert^2_2 \geq \lVert X'(a+b+c)\rVert_2^2 -\frac{1}{4}n
        \end{align}
        This therefore gives that $(a,b,c)$ is $\frac{1}{4}$-DLM with respect to $[k]\setminus T$, $T\setminus [k]$, $\{p+1\}$, and the matrix X'. Therefore since $X'$ satisfies $3k$-RIP with $\delta_{3k}<\frac{1}{12}$, by \cref{prop:F7}, we have $\lVert a\rVert_2^2 + \lVert b\rVert_2^2 \leq \frac{1}{4} \lVert c \rVert^2_2$, or equivalently
        \begin{align*}
            |\text{supp}(a)|+|\text{supp}(b)|\leq \frac{1}{4}\sigma^2
        \end{align*}
        We therefore apply \cref{prop:F8} to conclude that it must be the case that with high probability,
        \begin{align*}
            \frac{1.25k}{\log p}< |\mathrm{supp}(a)|+|\mathrm{supp}(b)|+\sigma^2\leq \frac{5}{4}\sigma^2
        \end{align*}
        and therefore
        \begin{align*}
            \frac{k}{\log p}<\sigma^2
        \end{align*}
        and we get a contradiction.
\end{proof}

\subsection{Proof of \texorpdfstring{\cref{thm:2.1new}}{thm:2.1new}}\label{sec:10.1proof}
Here we prove \cref{thm:2.1new}. The strategy here is that we will first prove a similar result for a pure noise model (i.e. with no signal). We then show how the linear model can be reduced to the pure noise model. We again follow \cite{10.1214/21-AOS2130}. Note that, as we mentioned also above, the first part of \cref{thm:2.1new} is unchanged from \cite{10.1214/21-AOS2130}, so we only prove the second part.

\subsubsection{The Pure Noise Model}
Let $X\in\mathbb{R}^{n\times p}$ with i.i.d. standard Gaussian entries, and let $Y\in\mathbb{R}^n$ have i.i.d. $\mathcal{N}(0,\sigma^2)$ entries independently of $X$. Consider the optimization problem $\Psi_2$:
\begin{align*}
    (\Psi_2): \min n^{-\frac{1}{2}}\lVert Y-Xv\rVert_2\\
    \text{s.t. }v\in\{0,1\},\\
    \lVert v\rVert_0 = k
\end{align*}
Here there is no ``true" vector $v^*$; we are simply trying to find the exactly $k$-sparse binary vector that minimises this squared error. Denote the optimal value of this optimization problem by $\psi_2$. \cref{thm:A.1} is our main result for this subsection. It improves Theorem A.1 of \cite{GZ22_supp} by reducing the assumption that $\Omega(k\log k)=n=o(k\log p)$ (and hence $k=p^{o(1)}$) to $k=o( p^{1/3-\eta})$ for some $\eta>0$ and no assumption on $n$.

\begin{theorem}\label{thm:A.1}[Improved version of Theorem A.1 of \cite{GZ22_supp}]
    Suppose for some $\eta>0$, we have $k=o(p^{1/3-\eta})$. Denote $\alpha = (1+\eta)/(2+\eta)$. Then if $k\leq\sigma^2\leq 3k$, it holds that for every sufficiently large constant $D_0$ and $n\leq \frac{\alpha k\log(p/k)}{\log D_0}$, the cardinality of the set
    \begin{align*}
        \left\{v\in\left\{0,1\right\}^p : \lVert v\rVert_0 =k,n^{-\frac{1}{2}}\lVert Y-Xv\rVert_2 \leq R\sqrt{k+\sigma^2}\exp\left(-\frac{k\log \frac{p}{k}}{n}\right)\right\}
    \end{align*}
    is at least $R^{\frac{n}{5}}$ w.h.p. as $k\rightarrow\infty$, where $R=\exp\left(\alpha k\log(p/k)/n\right)$.
\end{theorem}
We will use three auxiliary lemmas from \cite{GZ22_supp}.
\begin{lemma}\label{lem:A.5}[Lemma A.5 of \cite{GZ22_supp}]
    If $m_1, m_2\in \mathbb{N}$ with $m_1 \geq 4$ and $m_2\leq \sqrt{m_1}$, then
    \begin{align*}
        \binom{m_1}{m_2}\geq \frac{m_1^{m_2}}{4m_2!}
    \end{align*}
\end{lemma}
\begin{lemma}\label{lem:A.6}[Lemma A.6 of \cite{GZ22_supp}]
    The function $f:[0,1)\rightarrow\mathbb{R}$ defined by 
    \begin{align*}
        f(\rho):=\frac{1}{\rho}\log\left(\frac{1-\rho}{1+\rho}\right)
    \end{align*}
    is concave.
\end{lemma}
When $Z$ is a standard normal random variable, define for any $t>0,y\in\mathbb{R}$,
\begin{align*}
    p_{t,y}=\mathbb{P}\left(|Z-y|\leq t\right)
\end{align*}
and when $Z_1$ and $Z_2$ are standard normal with correlation $\rho\in[0,1]$, define
\begin{align*}
    q_{t,y,\rho}=\mathbb{P}\left(|Z_1-y|\leq t,|Z_2-y|\leq t\right)
\end{align*}
Note that $q_{t,y,0}=p_{t,y}^2$ and $q_{t,y,1}=p_{t,y}$. Note also that 
\begin{align*}
    p_{t,y} = \int_{[-t,t]}\frac{1}{\sqrt{2\pi}}e^{-\frac{1}{2}(y+x)^2}dx\geq \sqrt{\frac{2}{\pi}}te^{-(y^2+t^2)}
\end{align*}
and so 
\begin{align}\label{eq:log_pty}
    \log p_{t,y} \geq \log t-t^2-y^2+\frac{1}{2}\log\frac{2}{\pi}.
\end{align} 
\begin{lemma}\label{lem:A.7}[Lemma A.7 of \cite{GZ22_supp}]
    For any $t>0$, $y\in\mathbb{R}$, and $\rho\in[0,1)$,
    \begin{align*}
        \frac{q_{t,y,\rho}}{p_{t,y}}\leq \sqrt{\frac{1+\rho}{1-\rho}}e^{\rho y^2}
    \end{align*}
\end{lemma}

To prove \cref{thm:A.1}, we note that for any $x\in\mathbb{R}^n$, $n^{-\frac{1}{2}}\lVert x\rVert_2\leq \lVert x\rVert_\infty$. It will therefore suffice to lower bound $Z_{s,\infty}$, where $Z_{s,\infty}$ is a counting random variable that is equal to the cardinality of the set
\begin{align*}
    \left\{\beta\in\left\{0,1\right\}^p : \lVert\beta\rVert_0 =k,\lVert Y-X\beta\rVert_\infty \leq s \right\}.
\end{align*}
We want to show that for some sufficiently large $D_0$ and for 
\begin{align}\label{eq:def_s}
    s=R\sqrt{k+\sigma^2}\exp\left(-\frac{k\log \frac{p}{k}}{n}\right)=\sqrt{k+\sigma^2}\exp\left(-(1-\alpha)\frac{k\log \frac{p}{k}}{n}\right)
\end{align} 
 we have that with high probability,
\begin{align*}
    Z_{s,\infty}\geq R^{\frac{n}{5}}
\end{align*}
To prove the lower bound on $Z_{s,\infty}$, we will use a truncated second moment method. First, consider the expression
\begin{align*}
    \Upsilon = \Upsilon(Y) \triangleq \frac{\mathbb{E}\left[Z_{t\sqrt{k},\infty}^2|Y\right]}{\mathbb{E}\left[Z_{t\sqrt{k},\infty}|Y\right]^2}
\end{align*}
where we define $t:=s/\sqrt{k}$. \cref{prop:A.8} concerns a truncated form of $\Upsilon$. Specifically, we show that $\mathbb{E}\left(\min\left\{1,\Upsilon-1\right\}\right)$ converges to 0 in expectation. 

\begin{proposition}\label{prop:A.8}[New version of Proposition A.8 of \cite{GZ22_supp}]
    For any sufficiently large constant $D>0$, if $n\leq \frac{k\log\frac{p}{k^2}}{2\log D}$ and $k=o(p^{1/3-\eta})$ for some $\eta>0$, then for $s$ per \eqref{eq:def_s} and $t:=s/\sqrt{k}$, we have
    \begin{align*}
        \mathbb{E}_Y\left(\min\left\{1,\Upsilon-1\right\}\right)\rightarrow 0
    \end{align*}
    as $k\rightarrow\infty$.
\end{proposition}
\begin{proof}[Proof of \cref{prop:A.8}]
    Throughout the proof of \cref{prop:A.8}, we will take
    \begin{align*}
        R_1 = \exp\left(\frac{k\log\frac{p}{k^2}}{2n}\right)
    \end{align*}and notice $R_1 \geq D$. In particular, by taking $D$ large enough, we can always assume $R_1$ is bigger than a sufficiently large constant. Recall also
    \[t=s/\sqrt{k}=\sqrt{1+\frac{\sigma^2}{k}}\exp\left(-(1-\alpha)\frac{k\log \frac{p}{k}}{n}\right)\]
    
    As shown in section A.5 of \cite{GZ22_supp}, $\Upsilon$ can equivalently be written as 
    \begin{align}\label{eq:A.8_eq1}
        \Upsilon = \Upsilon(Y) = \sum_{\ell=0}^k \frac{\binom{p}{k-\ell,k-\ell,\ell,p-2k-\ell}}{\binom{p}{k}^2}\prod_{i=1}^n\frac{q_{t,\frac{Y_i}{\sqrt{k}},\frac{\ell}{k}}}{p_{t,\frac{Y_i}{\sqrt{k}}}^2}
    \end{align}
    For $\ell = 0$, we have $q_{t,\frac{Y_i}{\sqrt{k}},0}=p_{t,\frac{Y_i}{\sqrt{k}}}^2$ for all $i$, and so the first term of the sum is 1. For the other terms: it holds for all $1\leq \ell \leq k$:
    \begin{align}\label{eq:A.8_eq2}
        \binom{k}{\ell}\leq \left(\frac{ke}{\ell}\right)^{\ell}
    \end{align}
    and
    \begin{align}\label{eq:A.8_eq3}
        \binom{p-k}{k-\ell}\leq \frac{\left(p-k\right)^{k-\ell}}{\left(k-\ell\right)!}
    \end{align}
    Our assumption on $k$ implies $k\leq \sqrt{p}$, so we can apply \cref{lem:A.5} to get
    \begin{align}\label{eq:A.8_eq4}
        \binom{p}{k}\geq \frac{p^k}{4k!}
    \end{align}
    We combine \eqref{eq:A.8_eq2}, \eqref{eq:A.8_eq3}, and \eqref{eq:A.8_eq4} to see that for every $1\leq \ell \leq k$, it holds that
    \begin{align*}
        \frac{\binom{p}{k-\ell,k-\ell,l,p-2k+\ell}}{\binom{p}{k}^2}=\binom{k}{\ell}\frac{\binom{p-k}{k-\ell}}{\binom{p}{k}}\leq (ke/\ell)^\ell \frac{(p-k)^{k-\ell}}{(k-\ell)!}\frac{4k!}{p^k}\leq 4\left(\frac{p\ell  }{e k^2}\right)^{-\ell}
    \end{align*}
    Therefore we get that 
    \begin{align}\label{UpsilonUB}
        \Upsilon \leq 1+4\sum_{\ell = 1}^k\left(\frac{p\ell }{e k^2}\right)^{-\ell}\prod_{i=1}^n\frac{q_{t,\frac{Y_i}{\sqrt{k}},\frac{\ell}{k}}}{p_{t,\frac{Y_i}{\sqrt{k}}}^2}
    \end{align}
    Now we fix a parameter $\zeta\in(0,1)$ to be optimized later. Note that by law of total probability, 
    \begin{align*}
        \mathbb{E}\left(\min\left\{1,\Upsilon-1\right\}\right) &= \mathbb{E}_Y\left(\min\left\{1,\Upsilon-1\right\}\mathbf{1}\left(\min\left\{1,\Upsilon-1\right\}\geq \zeta^n\right)\right)\\
        &+ \mathbb{E}_Y\left(\min\left\{1,\Upsilon-1\right\}\mathbf{1}\left(\min\left\{1,\Upsilon-1\right\}\leq \zeta^n\right)\right) \\
        &\leq \mathbb{P}\left(\min\left\{1,\Upsilon-1\right\}\geq \zeta^n\right)+\zeta^n
    \end{align*}
    Observe that if $\Upsilon\geq 1+\zeta^n$, then the upper bound on $\Upsilon$ in \eqref{UpsilonUB} implies that at least one of the summands of 
    \begin{align*}
        \sum_{\ell = 1}^k 4\left(\frac{p\ell }{e k^2}\right)^{-\ell}\prod_{i=1}^n\frac{q_{t,\frac{Y_i}{\sqrt{k}},\frac{\ell}{k}}}{p_{t,\frac{Y_i}{\sqrt{k}}}^2}
    \end{align*}
    must be at least $\frac{\zeta^n}{k}$. We thus apply the union bound:
    \begin{align*}
        \mathbb{P}\left(\min\left\{1,\Upsilon-1\right\}\geq \zeta^n\right) &\leq \mathbb{P}\left(\Upsilon\geq 1+\zeta^n\right)\\
        &\leq \mathbb{P}\left(\bigcup_{\ell = 1}^k\left\{4\left(\frac{p \ell }{e k^2}\right)^{-\ell}\prod_{i=1}^n\frac{q_{t,\frac{Y_i}{\sqrt{k}},\frac{\ell}{k}}}{p_{t,\frac{Y_i}{\sqrt{k}}}^2}\geq\frac{\zeta^n}{k}\right\}\right)\\
        &\leq \sum_{\ell = 1}^k \mathbb{P}\left(4\left(\frac{p \ell }{e k^2}\right)^{-\ell}\prod_{i=1}^n\frac{q_{t,\frac{Y_i}{\sqrt{k}},\frac{\ell}{k}}}{p_{t,\frac{Y_i}{\sqrt{k}}}^2}\geq\frac{\zeta^n}{k}\right)
    \end{align*}
    Introducing the parameter $\rho = \frac{\ell}{k}$, we get
    \begin{align}\label{eq:TruncMeanUB}
        \mathbb{E}\left(\min\left\{1,\Upsilon-1\right\}\right) \leq \zeta^n +\sum_{\rho = \frac{1}{k},\frac{2}{k},\dots,1}\mathbb{P}(\Upsilon_\rho)
    \end{align}
    where for $\rho = \frac{1}{k},\frac{2}{k},\dots,\frac{k-1}{k},1$, we define the event
    \begin{align*}
        \Upsilon_\rho \triangleq \left\{4\left(\frac{p \ell }{e k}\right)^{-\rho k}\prod_{i=1}^n\frac{q_{t,\frac{Y_i}{\sqrt{k}},\rho}}{p_{t,\frac{Y_i}{\sqrt{k}}}^2}\geq\frac{\zeta^n}{k}\right\}
    \end{align*}
    The next step of the proof is to bound $\mathbb{P}(\Upsilon_\rho)$. Now we set
    \begin{align*}
        \rho_* = 1-\frac{\log D}{6\log R_1}
    \end{align*}
    Note that since $R_1\geq D$, this means that $\rho_*\in[\frac{5}{6},1)$. Now we will show separate bounds on $\mathbb{P}\left(\Upsilon_\rho\right)$ for $\rho\leq \rho_*$ and for $\rho>\rho_*$.
    
    \begin{lemma}\label{lem:A.9}[New version of Lemma A.9 of \cite{GZ22_supp}]
    For all $\rho\in(\rho_*,1]$, $\zeta\in(0,1)$ and $t = \sqrt{1+\sigma^2/k}\exp\left(-(1-\alpha)\frac{k\log\frac{p}{k}}{n}\right)$, it holds that
    \begin{align}
            \mathbb{P}(\Upsilon_\rho)\leq \left(\frac{2^\frac{1}{2}}{\zeta^{\frac{1}{12}}R_1^{\frac{1}{24}}}\right)^n
    \end{align}
    \begin{proof}[Proof of \cref{lem:A.9}]
        By taking logarithms, we have that
        \begin{align*}
            \mathbb{P}\left(\Upsilon_\rho\right) &= \mathbb{P}\left(4\left(\frac{p \ell }{e k^2}\right)^{-\rho k}\prod_{i=1}^n \frac{q_{t,\frac{Y_i}{\sqrt{k}},\rho}}{p^2_{t,\frac{Y_i}{\sqrt{k}}}}\geq \frac{\zeta^n}{k}\right) \\
            &= \mathbb{P}\left(\log 4-\rho k\log\frac{p \ell }{e k^2}+\sum_{i=1}^n\log\frac{q_{t,\frac{Y_i}{\sqrt{k}},\rho}}{p^2_{t,\frac{Y_i}{\sqrt{k}}}}\geq n\log \zeta-\log k\right)
        \end{align*}
        Note also that $q_{t,\frac{Y_i}{\sqrt{k}},\rho}\leq p_{t,\frac{Y_i}{\sqrt{k}}}$, so $\frac{q_{t,\frac{Y_i}{\sqrt{k}},\rho}}{p^2_{t,\frac{Y_i}{\sqrt{k}}}} \leq p^{-1}_{t,\frac{Y_i}{\sqrt{k}}}$, which tells us that
        \begin{align*}
            \mathbb{P}\left(\Upsilon_\rho\right) \leq \mathbb{P}\left(\log 4-\rho k\log\frac{p \ell }{e k^2}+\sum_{i=1}^n-\log p_{t,\frac{Y_i}{\sqrt{k}}}\geq n\log \zeta-\log k\right)
        \end{align*}
        Divide by $n$ and rearrange to get
        \begin{align*}
            \mathbb{P}(\Upsilon_\rho)\leq \mathbb{P}\left(\frac{1}{n}\sum_{i=1}^n-\log p_{t,\frac{Y_i}{\sqrt{k}}}\geq \log \zeta -\frac{\log 4k}{n}+\rho \frac{k\log\frac{p \ell }{e k^2}}{n}\right)
        \end{align*}
        As noted above, it holds that $-\log p_{t,\frac{Y_i}{\sqrt{k}}}\leq -\log t+t^2+y^2-\frac{1}{2}\log \frac{2}{\pi}$, and so we get
        \begin{align*}
            \mathbb{P}(\Upsilon_\rho)\leq \mathbb{P}\left(-\log t+t^2+\frac{1}{n}\sum_{i=1}^n\frac{Y_i^2}{k}+\frac{1}{2}\log\frac{2}{\pi}\geq \log \zeta -\frac{\log 4k}{n}+\rho \frac{k\log\frac{p \ell }{e k^2}}{n}\right)
        \end{align*}Since $\ell/e =\rho k /e \geq 5k/(6e) $, it therefore holds

        \begin{align*}
            \mathbb{P}(\Upsilon_\rho)\leq \mathbb{P}\left(-\log t+t^2+\frac{1}{n}\sum_{i=1}^n\frac{Y_i^2}{k}+\frac{1}{2}\log\frac{2}{\pi}\geq \log \zeta -\frac{ \log (6/5) k+\log 4k}{n}+\rho \frac{k\log\frac{p}{ k}}{n}\right)
        \end{align*}
        Apply $\log t = \frac{1}{2}\log\left(1+\frac{\sigma^2}{k}\right)-\frac{(1-\alpha)k\log\frac{p}{k}}{n}$ to get
        \begin{align*}
\mathbb{P}(\Upsilon_\rho) \leq \mathbb{P} \Bigg( 
& -\frac{1}{2}\log\left(1+\frac{\sigma^2}{k}\right) 
+ \frac{(1-\alpha)k\log\frac{p}{k}}{n} + t^2 
+ \frac{1}{n}\sum_{i=1}^n\frac{Y_i^2}{k} 
+ \frac{1}{2}\log\frac{2}{\pi} \\
& \geq \log \zeta 
- \frac{\log (6/5) k + \log 4k}{n} 
+ \rho \frac{k\log\frac{p}{k^2}}{n} 
\Bigg)
\end{align*}
        Collect terms and rearrange to get
        \begin{align*}
\mathbb{P}(\Upsilon_\rho) \leq \mathbb{P} \Bigg( 
& -\frac{1}{2}\log\left(1+\frac{\sigma^2}{k}\right) + t^2 
+ \frac{1}{n}\sum_{i=1}^n\frac{Y_i^2}{k} 
+ \frac{1}{2}\log\frac{2}{\pi} \\
& \geq \log \zeta 
- \frac{\log (6/5) k + \log 4k}{n} 
+ \left(\rho - (1-\alpha)\right) \frac{k\log\frac{p}{k}}{n} 
\Bigg)
\end{align*}
        Note that $\alpha>\frac{1}{2}$, $k=o(p)$ and $\rho>\rho_*\geq \frac{5}{6}$, $\rho-(1-\alpha)$ is at least $\frac{1}{3}$. Note also that $\rho \frac{k\log\frac{p}{k}}{n}$ is of larger order than $\frac{\log (6/5) k+\log 4k}{n}$. We can therefore lower bound their sum for large enough $k$ by 
        \begin{align*}
            -\frac{\log(6/5)k+\log 4k}{n}+\left(\rho-(1-\alpha)\right) \frac{k\log\frac{p}{k}}{n}> \frac{1}{4}\frac{k\log\frac{p}{k^2}}{n} = \frac{1}{2}\log R_1
        \end{align*}
        which gives us
        \begin{align*}
            \mathbb{P}(\Upsilon_\rho)\leq \mathbb{P}\left(-\frac{1}{2}\log\left(1+\frac{\sigma^2}{k}\right)+t^2+\frac{1}{n}\sum_{i=1}^n\frac{Y_i^2}{k}+\frac{1}{2}\log\frac{2}{\pi}\geq \log \zeta +\frac{1}{2}\log R_1\right)
        \end{align*}
        Now we note that since $1\leq \sigma^2/k\leq 3$, and $R_1\geq D$, where $D$ is a constant we can choose to be sufficiently large, a large enough choice of $D$ means that we can absorb $\frac{1}{2}\log(1+\sigma^2/k)$ and $\frac{1}{2}\log(2/\pi)$ into $\frac{1}{2}\log R_1$. Note also that $t\leq 1$, so we can absorb $t^2$ into $\frac{1}{2}\log R_1$ as well, which means that
        \begin{align*}
            \mathbb{P}(\Upsilon_\rho)\leq \mathbb{P}\left(\frac{1}{n}\sum_{i=1}^n\frac{Y_i^2}{k}\geq \log \zeta +\frac{1}{2}\log R_1\right)
        \end{align*}
        Then multiply across by $\frac{n}{12}$ and take the exponential of both sides to get
        \begin{align*}
            \mathbb{P}(\Upsilon_\rho)&\leq \mathbb{P}\left(\sum_{i=1}^n\frac{Y_i^2}{12k}\geq \frac{n}{12}\log \zeta +\frac{n}{24}\log R_1\right)\\
            &=\mathbb{P}\left(\exp\left(\sum_{i=1}^n\frac{Y_i^2}{12k}\right)\geq \zeta^{\frac{n}{12}}R_1^{\frac{n}{24}}\right)
        \end{align*}
        
        Next, we note that the $Y_i$'s are i.i.d., and $\frac{Y_1}{\sqrt{k}}\sim\mathcal{N}(0,\frac{\sigma^2}{k})$ and $1\leq \sigma^2/k\leq 3$. We apply Markov's inequality to get
        \begin{align*}
            \mathbb{P}(\Upsilon_\rho)&\leq \frac{1}{\zeta^{n/12}R_1^{n/24}}\mathbb{E}\left[\exp\left(\sum_{i=1}^n\frac{Y_i^2}{12k}\right)\right]\\
            &=\frac{1}{\zeta^{n/12}R_1^{n/24}}\left(\mathbb{E}\left[\exp\left(\frac{Y_1^2}{12k}\right)\right]\right)^n
        \end{align*}
        We can calculate the expectation explicitly as 
        \begin{align*}
            \mathbb{E}\left[\exp\left(\frac{Y_1^2}{12k}\right)\right]=\frac{1}{\sqrt{1-\frac{\sigma^2}{6k}}}\leq \sqrt{2}
        \end{align*}
        and so we get 
        \begin{align*}
            \mathbb{P}(\Upsilon_\rho)\leq \left(\frac{2^\frac{1}{2}}{\zeta^{\frac{1}{12}}R_1^{\frac{1}{24}}}\right)^n
        \end{align*}
        as desired, which proves \cref{lem:A.9}.
        \end{proof}
    \end{lemma}
    \begin{lemma}\label{lem:A.10}[New version of Lemma A.10 of \cite{GZ22_supp}]
    Suppose that for some $\eta>0$, $k=o(p^{1/3-\eta})$. Then, for all $\rho\in[\frac{1}{k},\rho_*]$ and $\zeta\in(0,1)$, it holds that for large enough $k$,
    \begin{align*}
        \mathbb{P}(\Upsilon_\rho)&\leq 2^{\frac{n}{2}}\left(R_1^\gamma\zeta^k\right)^{-\frac{n}{12}}+ 2^{\frac{n}{2}}\left(R_1^{\frac{1}{2}}\zeta^2\right)^{-\frac{n}{12}} \\
    \end{align*}
    where $\gamma = \frac{2\eta}{1+\eta}>0$
    \begin{proof}[Proof of \cref{lem:A.10}]
        By \cref{lem:A.7}, we have
        \begin{align*}
            \mathbb{P}(\Upsilon_\rho) &= \mathbb{P}\left(4\left(\frac{p e \ell}{k^2}\right)^{-\rho k}\prod_{i=1}^n \frac{q_{t,\frac{Y_i}{\sqrt{k}},\rho}}{p^2_{t,\frac{Y_i}{\sqrt{k}}}}\geq \frac{\zeta^n}{k}\right) \\
            & \leq \mathbb{P}\left(4\left(\frac{p}{k^2}\right)^{-\rho k}\prod_{i=1}^n \frac{q_{t,\frac{Y_i}{\sqrt{k}},\rho}}{p^2_{t,\frac{Y_i}{\sqrt{k}}}}\geq \frac{\zeta^n}{k}\right) \\
            &\leq \mathbb{P}\left(4\left(\frac{p}{k^2}\right)^{-\rho k}\prod_{i=1}^n \left(\sqrt{\frac{1+\rho}{1-\rho}}\exp\left(\rho \frac{Y_i^2}{k}\right)\right)\geq \frac{\zeta^n}{k}\right) \\
        \end{align*}
        Take logs and then divide by $n$ to get
        \begin{align*}
            \mathbb{P}(\Upsilon_\rho) &\leq\mathbb{P}\left(\log 4-\rho k\log \frac{p}{k^2}+\frac{n}{2}\log\left(\frac{1+\rho}{1-\rho}\right)+\sum_{i=1}^n \rho\frac{Y_i^2}{k}\geq n\log \zeta-\log k\right) \\
            &=\mathbb{P}\left(\frac{1}{n}\log 4-\rho \frac{k\log \frac{p}{k^2}}{n}+\frac{1}{2}\log\left(\frac{1+\rho}{1-\rho}\right)+\rho\sum_{i=1}^n \frac{Y_i^2}{nk}\geq \log \zeta-\frac{\log k}{n}\right) \\
        \end{align*}
        Then rearrange terms and divide by $\rho$ to get
        \begin{align*}
            &= \mathbb{P}\left(\rho\sum_{i=1}^n\frac{Y_i^2}{kn}\geq \log \zeta-\frac{\log 4k}{n}+\frac{1}{2}\log\left(\frac{1-\rho}{1+\rho}\right)+\frac{\rho k\log\frac{p}{k^2}}{n}\right)\\
            &= \mathbb{P}\left(\sum_{i=1}^n\frac{Y_i^2}{kn}\geq \rho^{-1}\log \zeta-\rho^{-1}\frac{\log 4k}{n}+\frac{1}{2\rho}\log\left(\frac{1-\rho}{1+\rho}\right)+\frac{k\log\frac{p}{k^2}}{n}\right)\\
        \end{align*}
        Let 
        \begin{align*} 
            f(\rho) = \rho^{-1}\log \zeta-\rho^{-1}\frac{\log 4k}{n}+\frac{1}{2\rho}\log\left(\frac{1-\rho}{1+\rho}\right)+\frac{k\log\frac{p}{k^2}}{n}
        \end{align*}
        By \cref{lem:A.6} and because $\zeta<1$, $f$ is concave. Therefore the minimum value of $f$ on $[\frac{1}{k},\rho_*]$ is either $f(\frac{1}{k})$ or $f(\rho_*)$, and thus
        \begin{align*}
            \mathbb{P}(\Upsilon_\rho) &\leq \mathbb{P}\left(\sum_{i=1}^n \frac{Y_i^2}{kn}\geq \min\left\{f\left(\frac{1}{k}\right),f(\rho_*)\right\}\right)\\
            &\leq \mathbb{P}\left(\sum_{i=1}^n \frac{Y_i^2}{kn}\geq f\left(\frac{1}{k}\right)\right) +\mathbb{P}\left(\sum_{i=1}^n \frac{Y_i^2}{kn}\geq f(\rho_*)\right)
        \end{align*}
        Note that $\mathbb{E}\left(\frac{Y_i^2}{k}\right)=\frac{\sigma^2}{k}\in[1,3]$. From the moment generating function for a $\chi^2_1$ random variable, we have that
        \begin{align*}
            \mathbb{E}\left[\exp\left(\frac{Y_i^2}{12k}\right)\right]=\left(1-\frac{\sigma^2}{6k}\right)^{-\frac{1}{2}}\leq\sqrt{2}
        \end{align*}
        so we apply a Chernoff bound. For $w\in\left\{f\left(\frac{1}{k}\right),f\left(\rho_*\right)\right\}$, it holds that
        \begin{align*}
            \mathbb{P}\left(\sum_{i=1}^n\frac{Y_i^2}{12k}\geq \frac{nw}{12}\right)\leq \exp\left(-\frac{nw}{12}\right)\mathbb{E}\left(\exp\left(\frac{Y_1^2}{12k}\right)\right)^n
        \end{align*}
        which therefore gives us
        \begin{align*}
            \mathbb{P}(\Upsilon_\rho)\leq 2^{\frac{n}{2}}\exp\left(-\frac{n}{12}f\left(\frac{1}{k}\right)\right)+2^{\frac{n}{2}}\exp\left(-\frac{n}{12}f\left(\rho_*\right)\right)
        \end{align*}
        Now we want to lower bound $f\left(\frac{1}{k}\right)$ and $f(\rho_*)$. First we bound $f\left(\frac{1}{k}\right)$:
        \begin{align*}
            f\left(\frac{1}{k}\right) = k\log\zeta -\frac{k\log 4k}{n}+\frac{k}{2}\log\left(\frac{1-\frac{1}{k}}{1+\frac{1}{k}}\right)+\frac{k\log\frac{p}{k^2}}{n}
        \end{align*}
        Recall that $\log R_1 = k\log(p/k^2)/(2n)$, so we can rewrite $f(1/k)$ as
        \begin{align*}
            f\left(\frac{1}{k}\right) &= k\log\zeta -\frac{\left(2\log R_1\right)\left(\log 4k\right)}{\log \frac{p}{k^2}}+\frac{k}{2}\log\left(\frac{1-\frac{1}{k}}{1+\frac{1}{k}}\right)+2\log R_1\\
            &=k\log\zeta +2\log R_1\left[1-\frac{\log 4k}{\log\frac{p}{k^2}}\right]+\frac{k}{2}\log\left(\frac{1-\frac{1}{k}}{1+\frac{1}{k}}\right)
        \end{align*}
        Now we claim that for large enough $k$, $\log 4k/(\log (p/k^2))$ is at most $\frac{1}{1+\eta}<1$. This follows directly from the assumption that $k=o(p^{1/3-\eta})=o(p^{1/(3+\eta)})$ and therefore for large enough $k$ it holds $4^{1+\eta}k^{3+\eta}\leq p$.
        We can therefore lower bound $f(1/k)$ by
        \begin{align*}
            f\left(\frac{1}{k}\right)&\geq k\log \zeta +\left(\frac{2\eta}{1+\eta}\right)\log R_1+\frac{k}{2}\log\left(\frac{1-\frac{1}{k}}{1+\frac{1}{k}}\right)
        \end{align*}
        Finally, it is easy to check that $\frac{k}{2}\log\left(\frac{1-\frac{1}{k}}{1+\frac{1}{k}}\right)>-1.1$ for all $k\geq 2$, so we can absorb the $\frac{k}{2}\log\left(\frac{1-\frac{1}{k}}{1+\frac{1}{k}}\right)$ by choosing a larger $D$ (since $R_1\geq D$), and therefore (writing $\gamma$ for $2\eta/(1+\eta)$, and $\gamma>0$):
        \begin{align*}
            f\left(\frac{1}{k}\right)&\geq k\log \zeta +\gamma\log R_1
        \end{align*}
        which means that 
        \begin{align*}
            2^{\frac{n}{2}}\exp\left(-\frac{n}{12}f\left(\frac{1}{k}\right)\right)\leq 2^{\frac{n}{2}}\left(R_1^\gamma\zeta^k\right)^{-\frac{n}{12}}
        \end{align*}
        Now we will lower bound $f(\rho_*)$. Use the fact that $\log R_1= \frac{ k\log\frac{p}{k^2}}{2n}$:
        \begin{align*}
            f(\rho_*) &= \rho_*^{-1}\log \zeta-\rho_*^{-1}\frac{\log 4k}{n}+\frac{1}{2\rho_*}\log\left(\frac{1-\rho_*}{1+\rho_*}\right)+\frac{k\log\frac{p}{k^2}}{n} \\
            &=\rho_*^{-1}\log \zeta -\rho_*^{-1}\frac{\left(2\log R_1\right)\left(\log 4k\right)}{k\log\frac{p}{k^2}}+\frac{1}{2\rho_*}\log\left(\frac{1-\rho_*}{1+\rho_*}\right)+2\log R_1\\
            &=\rho_*^{-1}\log \zeta +\frac{1}{2\rho_*}\log\left(\frac{1-\rho_*}{1+\rho_*}\right)+2\log R_1\left[1-\rho_*^{-1}\frac{\log 4k}{k\log\frac{p}{k^2}}\right]\\
        \end{align*}
        By definition of $\rho_*$, we have $\rho_* = 1-\log D/(6\log R_1)$, so it holds that $1<\rho_*^{-1}\leq \frac{6}{5}$. It holds that $\log 4k/(k\log (p/k^2))$ tends to zero as $k$ grows. We can therefore lower bound $f(\rho_*)$ for large enough $k$ by
        \begin{align*}
            f(\rho_*)\geq \rho_*^{-1}\log \zeta +\frac{1}{2\rho_*}\log\left(\frac{1-\rho_*}{1+\rho_*}\right)+\frac{4}{5}\log R_1
        \end{align*}
        Since $\zeta < 1$, $\log\zeta <0$. It is therefore implied by $1<\rho_*^{-1}\leq \frac{6}{5}$ that we can lower bound $\rho_*^{-1} \log \zeta$ by $2\log \zeta$:
        \begin{align*}
            f(\rho_*)\geq 2\log \zeta +\frac{1}{2\rho_*}\log\left(\frac{1-\rho_*}{1+\rho_*}\right)+\frac{4}{5}\log R_1
        \end{align*}
        Next we split $\frac{1}{2\rho_*}\log\left(\frac{1-\rho_*}{1+\rho_*}\right)$ into $\frac{1}{2\rho_*}\log(1-\rho_*)$ and $-\frac{1}{2\rho^*}\log(1+\rho_*)$. Since $1<\rho_*^{-1}\leq \frac{6}{5}$, we can lower bound $-\frac{1}{2\rho^*}\log(1+\rho_*)$ by $-1$, which we can absorb in the $\frac{4}{5}\log R_1$. This gives us
        \begin{align*}
            f(\rho_*)\geq 2\log \zeta +\frac{1}{2\rho_*}\log\left(1-\rho^*\right)+\frac{4}{5}\log R_1
        \end{align*}
        Applying the definition $\rho_* = 1-\log D/(6\log R_1)$, we can calculate explicitly that\newline$\frac{1}{2\rho_*}\log(1-\rho_*)$ can be written as
        \begin{align*}
            \frac{1}{2\rho_*}\log(1-\rho_*) &= \frac{3\log R_1}{6\log R_1-\log D}\log\left(\frac{\log D}{6\log R_1}\right) \\
            &\geq \frac{1}{2}\log\left(\frac{\log D}{6\log R_1}\right)\\
            &= \frac{1}{2}\log\log D-3\log\log R_1
        \end{align*}
        This is dominated by $\frac{4}{5}\log R_1$, so we can lower bound $f(\rho_*)$ by
        \begin{align*}
            f(\rho_*)\geq 2\log \zeta +\frac{1}{2}\log R_1
        \end{align*}
        and therefore
        \begin{align*}
            2^{\frac{n}{2}}\exp\left(-\frac{n}{12}f(\rho_*)\right) \leq 2^{\frac{n}{2}}\left(R_1^{\frac{1}{2}}\zeta^2\right)^{-\frac{n}{12}}
        \end{align*}
        Combine the bounds to get 
        \begin{align*}
            \mathbb{P}(\Upsilon_\rho)&\leq 2^{\frac{n}{2}}\left(R_1^\gamma\zeta^k\right)^{-\frac{n}{12}}+ 2^{\frac{n}{2}}\left(R_1^{\frac{1}{2}}\zeta^2\right)^{-\frac{n}{12}} \\
        \end{align*}
        \end{proof}
    \end{lemma}
        We are now ready to finish the proof of \cref{prop:A.8}. Combine the results of \cref{lem:A.9} and \cref{lem:A.10} to get that for all $\rho\in[\frac{1}{k},1]$ and $\zeta\in(0,1)$,
        \begin{align*}
            \mathbb{P}(\Upsilon_\rho) &\leq 2^{\frac{n}{2}}\left(R_1^\gamma\zeta^k\right)^{-\frac{n}{12}}+ 2^{\frac{n}{2}}\left(R_1^{\frac{1}{2}}\zeta^2\right)^{-\frac{n}{12}}+2^{\frac{n}{2}}\left(R_1^\frac{1}{24}\zeta^\frac{1}{12}\right)^{-n} \\
        \end{align*}
        It was shown earlier by the union bound \eqref{eq:TruncMeanUB} that 
        \begin{align*}
            \mathbb{E}_Y\left(\min\{1,\Upsilon-1\}\right)&\leq \zeta^n+\sum_{\rho=\frac{1}{k},\frac{2}{k},\dots,1}\mathbb{P}(\Upsilon_\rho) \\
        \end{align*}
        Now we apply \cref{lem:A.9} and \cref{lem:A.10} to see that
        \begin{align}\label{propA.8:sumToBound}
            \mathbb{E}_Y\left(\min\{1,\Upsilon-1\}\right) &\leq \zeta^n + k\left[2^{\frac{n}{2}}\left(R_1^\gamma\zeta^k\right)^{-\frac{n}{12}}+ 2^{\frac{n}{2}}\left(R_1^{\frac{1}{2}}\zeta^2\right)^{-\frac{n}{12}}+2^{\frac{n}{2}}\left(R_1^\frac{1}{24}\zeta^\frac{1}{12}\right)^{-n}\right] 
        \end{align}

        To prove \cref{prop:A.8}, it will therefore suffice to show that the sum in \eqref{propA.8:sumToBound} is converging to 0 as $k\rightarrow\infty$. Let $\zeta = R_1^{-\frac{\gamma}{2k}}$ (which is less than 1 for large enough $D$), so that $\log \zeta = -\frac{\gamma}{2k}\log R_1$. Then:
        \begin{align*}
            \zeta^n &= \exp\left(n\log\zeta\right)\\
            &=\exp\left(-\frac{\gamma n\log R_1}{2k}\right)
        \end{align*}
        Apply $\log R_1 = \frac{k\log (p/k^2)}{2n}$ to get that 
        \begin{align*}
            \zeta^n = \exp\left(-\frac{\gamma}{4}\log\frac{p}{k^2}\right) = \left(\frac{p}{k^2}\right)^{-\gamma/4}
        \end{align*}
        Since $k^3\leq p$, we have $p/k^2 \geq k$, so $\left(p/k^2\right)^{-\gamma/4}$ converges to 0 as $k\rightarrow\infty$.

        Next we will bound $k\left[2^{\frac{n}{2}}\left(R_1^{\gamma}\zeta^k\right)^{-\frac{n}{12}}\right]$. We can write it as
        \begin{align*}
            &k\left[2^{\frac{n}{2}}\left(R_1^{\gamma}\zeta^k\right)^{-\frac{n}{12}}\right] \\
            = &k\exp\left(\frac{n}{2}\log 2-\frac{\gamma n}{12}\log R_1-\frac{nk}{12}\log\zeta\right)
        \end{align*}
        Apply $\log \zeta = -\frac{\gamma}{2k}\log R_1$ to get
        \begin{align*}
             2&k\exp\left(\frac{n}{2}\log 2-\frac{\gamma n}{12}\log R_1+\frac{\gamma n}{24}\log R_1\right)\\
             =2&k\exp\left(\frac{n}{2}\log 2-\frac{\gamma n}{24}\log R_1\right)
        \end{align*}
        Apply $\log R_1 = \frac{k\log (p/k^2)}{2n}$ to get
        \begin{align*}
            2&k\exp\left(\frac{n}{2}\log 2-\frac{\gamma}{48}k\log\frac{p}{k^2}\right)
        \end{align*}
        Since $n\leq k\log(p/k^2)/(2\log D)$, we can upper bound this expression by
        \begin{align*}
            2&k\exp\left(\frac{\log 2}{4\log D}k\log\frac{p}{k^2}-\frac{\gamma}{48}k\log\frac{p}{k^2}\right)
        \end{align*}
        Choose $D$ large enough so that $\log 2/(4\log D) < \gamma/48$, and then this expression converges to 0 as $k\rightarrow\infty$. 

        Next, we will bound the third term of \eqref{propA.8:sumToBound}, $2^{\frac{n}{2}}k\left(R_1^{\frac{1}{2}}\zeta^2\right)^{-\frac{n}{12}}$. Rewrite it as 
        \begin{align*}
            2^{\frac{n}{2}}k\left(R_1^{\frac{1}{2}}\zeta^2\right)^{-\frac{n}{12}} = k\exp\left(\frac{n}{2}\log 2-\frac{n}{24}\log R_1-\frac{n}{6}\log\zeta\right)
        \end{align*}
        Apply $\log \zeta = -\frac{\gamma}{2k}\log R_1$ to get
        \begin{align*}
            k\exp\left(\frac{n}{2}\log 2-\frac{n}{24}\log R_1+\frac{\gamma n}{12k}\log R_1\right)
        \end{align*}
        Then apply $\log R_1 = \frac{k\log (p/k^2)}{2n}$ to get
        \begin{align*}
            k\exp\left(\frac{n}{2}\log 2-\frac{1}{48}k\log \frac{p}{k^2}+\frac{\gamma }{24}\log\frac{p}{k^2}\right)
        \end{align*}
        Again, apply the upper bound $n\leq k\log(p/k^2)/(2\log D)$ to get that this expression is at most
        \begin{align*}
            k\exp\left(\frac{\log 2}{4\log D}k\log \frac{p}{k^2}-\frac{1}{48}k\log \frac{p}{k^2}+\frac{\gamma }{24}\log\frac{p}{k^2}\right)
        \end{align*}
        Choose $D$ large enough so that $\log 2/(4\log D)<1/48$, and then this converges to 0 as $k$ grows.
        
         Finally, we want to bound the fourth term of \eqref{propA.8:sumToBound}, $k\left[2^{\frac{n}{2}}\left(R_1^{\frac{1}{24}}\zeta^\frac{1}{12}\right)^{-n}\right]$. This follows by an identical argument as for the third term of \eqref{propA.8:sumToBound}, which completes the proof of \cref{prop:A.8}.
       
\end{proof}

Now we can use \cref{prop:A.8} to prove \cref{thm:A.1}. 

\begin{proof}[Proof of \cref{thm:A.1}]
    Since by assumption it holds that for large enough $k$, $k^{3+\eta}\leq p$, we have that for $\alpha = (1+\eta)/(2+\eta)<1$, it also holds that
    \[
    \left(\frac{p}{k}\right)^{\alpha}\leq \frac{p}{k^2}
    \]
    and therefore it holds that 
    \begin{align}\label{eq:A.1eq1}
    \alpha k\log\frac{p}{k}\leq k\log \frac{p}{k^2} 
    \end{align}
    By assumption of \cref{thm:A.1}, we have $n\leq \alpha k\log(\frac{p}{k})/(\log D_0)$, and so 
    \[
    n\leq \frac{\alpha k\log\frac{p}{k}}{\log D_0} \leq \frac{k\log \frac{p}{k^2}}{\log D_0} = \frac{k\log \frac{p}{k^2}}{2\log D_0^{1/2}} 
    \]
    We want to lower bound the cardinality of the set
    \begin{align}\label{eq:Ztau}
        \left\{v\in\{0,1\}^p:\lVert v\rVert_0=k,\lVert Y-Xv\rVert_\infty\leq \sqrt{k}\sqrt{1+\sigma^2/k}\exp\left(-(1-\alpha)\frac{k\log \frac{p}{k}}{n}\right)\right\}
    \end{align}
    Recall that $k\leq \sigma^2\leq 3k$, so let
    \begin{align*}
        t_0 = R\sqrt{1+\frac{\sigma^2}{k}}\exp\left(-\frac{k\log \frac{p}{k}}{n}\right) = \sqrt{1+\frac{\sigma^2}{k}}\exp\left(-(1-\alpha)\frac{k\log \frac{p}{k}}{n}\right)
    \end{align*} 
    then our goal is to obtain a lower bound on $Z_{t_0\sqrt{k}}$, where $Z_{t_0\sqrt{k}}$ denotes the cardinality of \eqref{eq:Ztau}.
    
    We let $D = D_0^{\frac{1}{2}}$, and then all the conditions of \cref{prop:A.8} are satisfied. We now require another improvement of a lemma from \cite{GZ22_supp}.
    \begin{lemma}\label{lem:A.11}[New version of Lemma A.11 of \cite{GZ22_supp}]
        The following bound holds with high probability with respect to $Y$ as $k$ increases:
        \begin{align*}
            n^{-1}\log \mathbb{E}\left[Z_{t_0\sqrt{k},\infty}|Y\right] \geq \frac{1}{4}\log R
        \end{align*}
        \begin{proof}
            We have for $Y=(Y_1,\dots,Y_n)$,
            \begin{align*}
                \mathbb{E}\left[Z_{t_0\sqrt{k},\infty}|Y\right]=\binom{p}{k}\prod_{i=1}^n\mathbb{P}\left(\left|\frac{Y_i}{\sqrt{k}}-X\right|<t_0|Y\right)=\binom{p}{k}\prod_{i=1}^n p_{t_0,\frac{Y_i}{\sqrt{k}}}
            \end{align*}
            where $X$ is a standard Gaussian independent of $Y$. Take logarithms to get
            \begin{align*}
                \log \mathbb{E}\left[Z_{t_0\sqrt{k},\infty}|Y\right]=\log \binom{p}{k}+\sum_{i=1}^n\log p_{t_0,\frac{Y_i}{\sqrt{k}}}
            \end{align*}
            Recall from \eqref{eq:log_pty} that for any $t>0$ and for any $y$,
            \begin{align*}
                \log p_{t,y} \geq \log t-t^2-y^2+\frac{1}{2}\log\frac{2}{\pi}
            \end{align*}
            and therefore
            \begin{align*}
                \frac{1}{n}\log \mathbb{E}\left[Z_{t_0\sqrt{k},\infty}|Y\right]\geq \frac{1}{n}\log \binom{p}{k}+\log t_0-t_0^2+\frac{1}{2}\log\frac{2}{\pi}-\frac{1}{n}\sum_{i=1}^n\frac{Y_i^2}{k}
            \end{align*}
            Then by definition of $t_0$, we have $t_0 \geq R \exp\left(-\frac{k\log(p/k)}{n}\right)$, and so
            \begin{align*}
                \frac{1}{n}\log \mathbb{E}\left[Z_{t_0\sqrt{k},\infty}|Y\right]\geq \frac{1}{n}\log \binom{p}{k}+\log R-\frac{k\log\frac{p}{k}}{n}+\frac{1}{2}\log\frac{2}{\pi}-\frac{1}{n}\sum_{i=1}^n\frac{Y_i^2}{k}
            \end{align*}
            Since $\binom{p}{k}\geq(p/k)^k$, we see that $\frac{1}{n}\log\binom{p}{k}\geq \frac{1}{n}k\log\frac{p}{k}$. Apply this to see that
            \begin{align*}
                \frac{1}{n}\log \mathbb{E}\left[Z_{t_0\sqrt{k},\infty}|Y\right]\geq \frac{1}{2}\log R+\frac{1}{2}\log\frac{2}{\pi}-\frac{1}{n}\sum_{i=1}^n\frac{Y_i^2}{k}
            \end{align*}
            Split the $\frac{1}{2}\log R$ into $\frac{1}{4}\log R+\frac{1}{4}\log R$.
            It will therefore suffice to show that with high probability as $k$ tends to infinity, 
            \begin{align*}
                \frac{1}{2}\log\frac{2}{\pi}+\frac{1}{4}\log R-\frac{1}{n}\sum_{i=1}^n\frac{Y_i^2}{k}\geq 0
            \end{align*}
            Substitute in for $R$ to see that the left hand side equals
            \begin{align*}
                \frac{1}{2}\log\frac{2}{\pi}+\frac{\alpha k\log\frac{p}{k}}{4n}-\frac{1}{n}\sum_{i=1}^n\frac{Y_i^2}{k}
            \end{align*}
            If $n$ is of order $k\log(p/k)$, then $n$ is growing to infinity with $k$, so the final term converges in probability to a constant by the law of large numbers, and so the expression is positive if $D_0$ is large enough. If $n$ is of lower order than this, the second term is growing to infinity while the third term is of constant order, and so their difference is positive with high probability.
        \end{proof}
    \end{lemma}
    Now we claim that with high probability as $k$ increases, 
        \begin{align*}
            Z_{t_0\sqrt{k},\infty}\geq \frac{1}{2}\mathbb{E}\left[Z_{t_0\sqrt{k},\infty}|Y\right]
        \end{align*}
    To see this, we first observe that
    \begin{align*}
        \mathbb{P}\left(Z_{t_0\sqrt{k},\infty}< \frac{1}{2}\mathbb{E}\left[Z_{t_0\sqrt{k},\infty}|Y\right]\right)\leq \mathbb{P}\left(|Z_{t_0\sqrt{k},\infty} -\mathbb{E}\left[Z_{t_0\sqrt{k}\infty}|Y\right]|\geq \frac{1}{2}\mathbb{E}\left[Z_{t_0\sqrt{k},\infty}|Y\right]\right)
    \end{align*}
    and by Chebyshev's inequality we get that
    \begin{align*}
        \mathbb{P}\left(|Z_{t_0\sqrt{k},\infty} -\mathbb{E}\left[Z_{t_0\sqrt{k}\infty}|Y\right]|\geq \frac{1}{2}\mathbb{E}\left[Z_{t_0\sqrt{k},\infty}|Y\right]|Y\right)&\leq 4\min\left[\frac{\mathbb{E}\left[Z_{t_0\sqrt{k},\infty}^2|Y\right]}{\mathbb{E}\left[Z_{t_0\sqrt{k},\infty}|Y\right]^2}-1,1\right]\\
    \end{align*}
    Take expectation over $Y$ to obtain
    \begin{align*}
        \mathbb{P}\left(|Z_{t_0\sqrt{k},\infty} -\mathbb{E}\left[Z_{t_0\sqrt{k}\infty}|Y\right]|\geq \frac{1}{2}\mathbb{E}\left[Z_{t_0\sqrt{k},\infty}|Y\right]|Y\right)&\leq 4\min\left[\frac{\mathbb{E}\left[Z_{t_0\sqrt{k},\infty}^2|Y\right]}{\mathbb{E}\left[Z_{t_0\sqrt{k},\infty}|Y\right]^2}-1,1\right]\\
        &\rightarrow 0
    \end{align*}
    as $k\rightarrow\infty$ by \cref{prop:A.8}. This proves the claim that $Z_{t_0\sqrt{k},\infty}\geq \frac{1}{2}\mathbb{E}\left[Z_{t_0\sqrt{k},\infty}|Y\right]$ with high probability. We combine this result with \cref{lem:A.11} to conclude that w.h.p., 
    \begin{align*}
        n^{-1}\log Z_{t_0\sqrt{k},\infty} &\geq n^{-1}\log \mathbb{E}\left[Z_{t_0\sqrt{k},\infty}|Y\right] -\frac{\log 2}{n}\\
        &\geq \frac{1}{4}\log R-\frac{\log 2}{n}
    \end{align*}
    and therefore w.h.p., $Z_{t_0\sqrt{k},\infty}\geq \frac{1}{2}R^\frac{n}{4}\geq R^{\frac{n}{5}}$ for sufficiently large $D_0$, which concludes the proof of \cref{thm:A.1}.
\end{proof}

\subsubsection{The Linear model}
Now we will extend to the linear model case. Let $X\in\mathbb{R}^{n\times p}$ with i.i.d. standard normal entries, and let $W\in\mathbb{R}^n$ be a vector with i.i.d. $\mathcal{N}(0,\sigma^2)$ entries, with $X$ and $W$ independent. Let $v^*\in\left\{0,1\right\}^p$, with $\lVert v^*\rVert_0 = k$. Let $Y=Xv^*+W$. Now the model is that we observe $X$ and $Y$, and we want to recover the support of $v^*$.

To prove \cref{thm:2.1new}, we will use a restricted optimization problem $\Phi_2(S)$ where $S$ is a subset of $\text{supp}(v^*)$:
\begin{align*}
    (\Phi_2(S))\text{ } \min_{v} n^{-\frac{1}{2}}\lVert Y-Xv\rVert_2 \\
    \text{s.t. } v\in\{0,1\}^p\\
    \lVert v\rVert_0 = k, \text{ supp}(v)\cap \text{supp}(v^*)=S
\end{align*}
and set $\phi_2(S)$ its optimal value. For any exactly $k$-sparse binary $v$ with $\text{supp}(v)\cap\text{supp}(v^*)=S$, we have
\begin{align*}
    Y-Xv &= Xv^*+W-Xv\\
    &=\sum_{i\in\text{supp}(v^*)}X_i + W - \sum_{i\in\text{supp}(v)}X_i \\
    &=\sum_{i\in\text{supp}(v^*)-S}X_i + W - \sum_{i\in\text{supp}(v)-S}X_i \\
    &=Y'-X'v_1
\end{align*}
where $X'\in\mathbb{R}^{n\times (p-k)}$ is $X$ with the columns corresponding to $\text{supp}(\beta^*)$ deleted,\newline$Y'=\sum_{i\in\text{supp}(v^*)-S}X_i+W\in\mathbb{R}^n$, and $v_1\in\{0,1\}^{p-k}$ is obtained from $v$ by deleting coordinates in $\text{supp}(v^*)$ (so $\lVert v_1\rVert_0 = k-|S|$). Then we can equivalently write $\Phi_2(S)$ as
\begin{align*}
    (\Phi_2(S))\text{ } \min_{v'} n^{-\frac{1}{2}}\lVert Y'-X'v'\rVert_2 \\
    \text{s.t. } v'\in\{0,1\}^{p-k}\\
    \lVert v'\rVert_0 = k-|S|
\end{align*}
Then $X'$ has i.i.d. $N(0,1)$ entries, and $Y'$ has i.i.d. $N(0,(k-|S|)+\sigma^2)$ entries, independently of $X'$. This satisfies all the assumptions of \cref{thm:A.1} except that we need $k-|S|\leq k-|S|+\sigma^2\leq 3(k-|S|)$, or equivalently, $\sigma^2\leq 2(k-|S|)$. In general this will not hold, but in the special case of $S=\emptyset$ we can apply \cref{thm:A.1} to get this result:
\begin{corollary}\label{cor:B.1}[New version of Corollary B.1 of \cite{GZ22_supp}]
    Suppose $\sigma^2\leq 2k$. For every sufficiently large constant $D_0$, if $4^{1+\eta}k^{3+\eta}\leq p$ for some $\eta>0$, and
    \[
    n\leq \frac{\alpha k\log\left(\frac{p-k}{k}\right)}{\log D_0}
    \]
    then the cardinality of the set
    \begin{align*}
        \left\{v\in\{0,1\}^{p-k}:\lVert v\rVert_0 = k, n^{-\frac{1}{2}}\lVert Y-Xv\rVert_2\leq R\sqrt{2k+\sigma^2}\exp\left(-\frac{k\log \left(\frac{p-k}{k}\right)}{n}\right)\right\}
    \end{align*}
    is at least $R^{\frac{n}{5}}$ w.h.p. as $k\rightarrow\infty$, where $R=\exp\left(\alpha k\log(p/k)/n\right)$ and $\alpha = (1+\eta)/(2+\eta)$.
\end{corollary}
\begin{proof}[Proof of \cref{thm:2.1new}]
    By a Taylor expansion, it holds that
    \[
    k\log \left(\frac{p-k}{k}\right) = k\log\frac{p}{k}-O\left(\frac{k^2}{p}\right)
    \]
    By assumption of \cref{thm:2.1new}, we have $n\leq \alpha k\log(p/k)/(3\log D_0)$, which therefore gives us
    \[
    n\leq \frac{\alpha k\log\left(\frac{p-k}{k}\right)+O\left(\frac{k^2}{p}\right)}{3\log D_0}
    \]
    Since $k^2/p\rightarrow 0$, this means that for large enough $D_0$ and large enough values of $k$, it holds that
    \begin{align}\label{eq:2D_0}
    n\leq \frac{\alpha k\log\left(\frac{p-k}{k}\right)}{\log 2D_0}
    \end{align}
    because $3\log D_0 > \log (2D_0)$ when $D_0> \sqrt{2}$.

    We can then apply \cref{cor:B.1} with $2D_0$ in place of $D_0$ to get that there are at least $R^{n/5}$ zero-overlap $v$'s that satisfy
    \begin{align}\label{Error_v_UpperBound}
        n^{-\frac{1}{2}}\lVert Y-Xv\rVert_2\leq R\sqrt{2k+\sigma^2}\exp\left(-\frac{k\log \left(\frac{p-k}{k}\right)}{n}\right)
    \end{align}

    Finally, note that we can write
    \begin{align*}
        \frac{k\log \frac{p-k}{k}}{n} = \frac{ k\log\frac{p}{k}}{n}\cdot \frac{\log\frac{p-k}{k}}{ \log \frac{p}{k}}
    \end{align*}
    Note also that $\log(\frac{p-k}{k})/(\log \frac{p}{k})\geq 1-o(1)$ as $k\rightarrow\infty$. It therefore holds that as $k$ grows,
    \begin{align*}
        -\frac{k\log\frac{p-k}{k}}{n}\leq -\frac{1-o(1)}{n} k \log \frac{p}{k}.
    \end{align*}
    which means that there are at least $R^{n/5}$ zero-overlap $v$'s that satisfy
    \begin{align}
        n^{-\frac{1}{2}}\lVert Y-Xv\rVert_2\leq R\sqrt{2k+\sigma^2}\exp\left(-\frac{(1-o(1))k\log \left(\frac{p-k}{k}\right)}{n}\right)
    \end{align}
    which completes the proof of \cref{thm:2.1new}.
\end{proof}

\subsection{Proof that structural conditions in \texorpdfstring{\cite{jordanMCMC}}{YWJ16} require \texorpdfstring{$n=\Omega(k\sigma)$}{n=Omega(k sigma)}:}\label{sec:structCondsJordan}
Here we show that the (deterministic) structural conditions in \cite{jordanMCMC} are, with high probability, not satisfied for a design matrix $X$ with i.i.d. Gaussian entries when $n=o(k\sigma)$. Let $X$ be an $n \times p$ matrix with i.i.d. standard Gaussian entries. For any subset $\gamma\subseteq [p]$, denote by $X_\gamma$ the $n\times |\gamma|$ submatrix of $X$ given by the corresponding columns of $X$. For any $\gamma$, let $\Phi_\gamma$ be the orthogonal projection operator onto the span of the column space of $X_\gamma$. Let $Z\sim N(0,I_n)$ independent of $X$.
\begin{lemma}\label{lem:structCondJordan}
    Let $n=\omega(k)$. It holds with high probability as $k\rightarrow\infty$ that 
\begin{align}\label{eq:structCondJordan}
    \mathbb{E}\left[\max_{|\gamma|=k}\max_{j\in[p]\setminus \gamma}\frac{1}{\sqrt{n}}\left|\left\langle (I-\Phi_\gamma)X_j,Z\right\rangle\right|\right]=\Omega\left(\sqrt{\frac{k^2}{n}}\right)
\end{align}
\end{lemma}
Consequently, Assumption B in \cite{jordanMCMC} (the Sparse Projection Condition) requires $L$ to be at least $\Omega(k^2/(n\log p))$ to hold with high probability under our setting. The High SNR condition (equation (9a) in that work) and the fact that we take $v^*\in\{0,1\}^p$ implies that their mixing time result (i.e. Theorem 2 of \cite{jordanMCMC}) requires $n=\Omega(L\sigma^2\log p)$, or equivalently, $n=\Omega(k\sigma)$. 
\begin{proof}[Proof of \cref{lem:structCondJordan}]
    Define the random set $T\subseteq \mathbb{R}^n$ as
\[T:=\left\{\frac{1}{\sqrt{n}}\left(I-\Phi_\gamma\right)X_j:\gamma\subseteq [p],|\gamma|= k,j\in[p]\setminus\gamma\right\}\]
By Sudakov's minorization inequality (see e.g., Corollary 7.4.3 of \cite{Ver_2018}), it holds for any $\varepsilon>0$ that 
\begin{align}\label{eq:sudakovmain}
    \mathbb{E}\left[\max_{t\in T}\left\langle t,Z\right\rangle\right]=\Omega\left(\varepsilon\sqrt{\log \mathcal{N}(T,\varepsilon)}\right)
\end{align}
where $\mathcal{N}(T,\varepsilon)$ is the $\varepsilon$-covering number of $T$. It therefore suffices to show that with high probability, at least $e^{\Omega(k)}$ points in $T$ are pairwise separated by at least $\sqrt{\frac{k}{n}}$, as then each of these points will require its own $\sqrt{\frac{k}{n}}$-sphere when forming a covering. 

Fix a choice of $j$. Fix any $\gamma\neq\gamma'$ satisfying $|\gamma\cap\gamma'|\geq k/3$ and $j\notin \gamma\cup\gamma'$. We first note that since $X_j$ is standard Gaussian independent of $X_\gamma$ and $X_{\gamma'}$, it holds that
\begin{align*}
    &\mathbb{E}\left(\norm{\frac{1}{\sqrt{n}}\left(I-\Phi_\gamma\right)X_j-\frac{1}{\sqrt{n}}\left(I-\Phi_{\gamma'}\right)X_{j}}^2\right)=\mathbb{E}\left(\norm{\frac{1}{\sqrt{n}}\left(\Phi_{\gamma'}-\Phi_\gamma\right)X_j}^2\right)\\
    =&\mathbb{E}\left(\mathbb{E}\left(\norm{\frac{1}{\sqrt{n}}\left(\Phi_{\gamma'}-\Phi_\gamma\right)X_j}^2\right)|X_\gamma,X_{\gamma'}\right)\\
    =&\mathbb{E}\left(\frac{1}{n}\norm{\Phi_{\gamma'}-\Phi_\gamma}_F^2\right)\\
    =&\mathbb{E}\left(\frac{1}{n}\left(\norm{\Phi_{\gamma'}}_F^2+\norm{\Phi_{\gamma}}_F^2-2\mathrm{tr}\left(\Phi_{\gamma'}\Phi_\gamma\right)\right)\right)\\
    =&\frac{1}{n}\left(k+k-2|\gamma\cap\gamma'|\right)=2\frac{k-|\gamma\cap\gamma'|}{n}
\end{align*}
where the last line holds because $\Phi_\gamma$ and $\Phi_{\gamma'}$ are orthogonal projection matrices.

Next, let $C>0$ be a large constant. Then, with probability at least $1-e^{-n}$, it holds that $\norm{X_j}\leq C\sqrt{n}$. Condition on this event. We next observe that the function $x\rightarrow\norm{\frac{1}{\sqrt{n}}(\Phi_{\gamma'}-\Phi_\gamma)x}^2$ is $\frac{2C}{\sqrt{n}}$-Lipschitz when restricted to the ball of radius $C\sqrt{n}$. Therefore, by Gaussian concentration of Lipschitz functions, it holds for some universal constant $c>0$ and any $t> 0$ that
\[\mathbb{P}\left(\left|\norm{\frac{1}{\sqrt{n}}(\Phi_{\gamma'}-\Phi_\gamma)X_j}^2-2\frac{k-|\gamma\cap\gamma'|}{n}\right|\geq t\right)\leq 2\exp^{-cnt^2}\]
Since we have $|\gamma\cap\gamma'|\leq k/3$ assumption, if we then choose $t=\sqrt{k/n}$, then we see that with probability at least $1-2\exp\left(-ck\right)$, it holds that 
\[\norm{\frac{1}{\sqrt{n}}(\Phi_{\gamma'}-\Phi_\gamma)X_j}^2\geq \frac{k}{n}\]
By a union bound over $\gamma$ and $\gamma'$, this holds with high probability for at least $\exp(c'k)$ choices of $(\gamma,\gamma')$ satisfying $|\gamma\cap\gamma'|\leq k/3$, where $c'<c$ is a small enough constant (such a set of subsets of $[p]$ is guaranteed to exist by Lemma 4.14 of \cite{rigollet2023highdimensionalstatistics}). Therefore, with high probability we have that the $\sqrt{\frac{k}{n}}$-covering number of $T$, $\mathcal{N}\left(T, \sqrt{\frac{k}{n}}\right)$, is at least $e^{\Omega(k)}$. The desired result then follows by \eqref{eq:sudakovmain}.
\end{proof}

\end{document}